\documentclass[a4paper,oneside,11pt]{amsart}
\usepackage{geometry}

\usepackage{mathtext}
\usepackage[T1]{fontenc}
\usepackage{color}
\usepackage{amsmath}
\usepackage{amsfonts,amssymb,amsthm}
\usepackage[pdftex]{graphicx}
\usepackage{float}
\usepackage{microtype}

\usepackage[colorlinks]{hyperref}
\hypersetup{
bookmarksnumbered,
pdfstartview={FitH},
breaklinks=true,
linkcolor=blue,
urlcolor=blue,
citecolor=blue,
bookmarksdepth=2
}

\usepackage{enumerate}

\usepackage{cite}

\DeclareMathOperator{\rk}{rk}
\DeclareMathOperator{\sgn}{sgn}
\DeclareMathOperator{\Aut}{Aut}

\DeclareMathOperator{\Ker}{Ker}
\DeclareMathOperator{\Hom}{Hom}
\DeclareMathOperator{\diag}{diag}
\DeclareMathOperator{\Id}{Id}
\DeclareMathOperator{\Sym}{Sym}
\DeclareMathOperator{\Span}{Span}

\DeclareMathOperator{\Mat}{Mat}
\DeclareMathOperator{\PSp}{PSp}
\DeclareMathOperator{\Sp}{Sp}
\DeclareMathOperator{\GL}{GL}

\DeclareMathOperator{\OO}{O}
\DeclareMathOperator{\SO}{SO}

\DeclareMathOperator{\Stab}{Stab}

\DeclareMathOperator{\spin}{\mathfrak{spin}}
\DeclareMathOperator{\spp}{\mathfrak{sp}}
\DeclareMathOperator{\ksp}{\mathfrak{ksp}}
\DeclareMathOperator{\Cl}{Cl}
\DeclareMathOperator{\ClGr}{ClGr}
\DeclareMathOperator{\clgr}{\mathfrak{clgr}}
\DeclareMathOperator{\End}{End}
\DeclareMathOperator{\Lie}{Lie}
\DeclareMathOperator{\Spin}{Spin}

\DeclareMathOperator{\Imm}{Im}
\DeclareMathOperator{\Ree}{Re}
\DeclareMathOperator{\Fix}{Fix}
\DeclareMathOperator{\KSp}{KSp}
\DeclareMathOperator{\Herm}{Herm}

\DeclareMathOperator{\Is}{Is}
\DeclareMathOperator{\PIs}{\mathbb{P}Is}
\DeclareMathOperator{\pol}{pol}
\DeclareMathOperator{\bpol}{\overline{pol}}
\DeclareMathOperator{\tr}{tr}

\DeclareMathOperator{\Der}{Der}

\newcommand{\midwd}{\,\middle\vert\,}
\newcommand{\R}{\mathbb R}
\newcommand{\CC}{\mathbb C}
\newcommand{\HH}{\mathbb H}
\newcommand{\PP}{\mathbb P}
\newcommand{\N}{\mathbb N}
\newcommand{\Z}{\mathbb Z}

\newcommand{\K}{\mathbb K}
\newcommand{\Oc}{\mathbb O}

\newcommand{\Om}{
\begin{pmatrix}
0 & 1\\
-1 & 0
\end{pmatrix}}

\newcommand{\Clm}[2]
{\begin{pmatrix}
#1\\
#2
\end{pmatrix}}

\theoremstyle{plain}
\newtheorem{teo}{Theorem}[section]
\newtheorem{cor}[teo]{Corollary}
\newtheorem{lem}[teo]{Lemma}
\newtheorem{prop}[teo]{Proposition}
\newtheorem{fact}[teo]{Fact}

\theoremstyle{definition}
\newtheorem{df}[teo]{Definition}

\theoremstyle{remark}
\newtheorem{rem}[teo]{Remark}

\newcommand{\bs}{\setminus}
\newcommand{\defin}{\emph}

\author{Eugen Rogozinnikov}

\begin{document}
\sloppy
\begin{abstract}
We introduce the symplectic group $\Sp_2(G, \sigma)$ associated to a Lie subgroup $G$ of a (possibly noncommutative) associative algebra $A$ equipped with an anti-involution $\sigma$. Our construction recovers several classical Lie groups as special cases, and in particular provides new realizations of spin groups as instances of $\Sp_2(G, \sigma)$ for suitable subgroups $G$ of the Clifford algebra. This case is not covered by the framework of~\cite{ABRRW}, which focuses on the specific situation $G = A^\times$, and is thus of particular interest.

We construct and study geometric spaces on which $\Sp_2(G, \sigma)$ acts. In particular, we define the space of $G$-isotropic elements and the corresponding space of $G$-isotropic lines, which generalize the classical projective line. We analyze the group action on these spaces and introduce natural invariants, such as the notion of positive triples and quadruples of $G$-isotropic lines and a generalized cross--ratio of positive quadruples of $G$-isotropic lines. Finally, when the Lie algebra of $G$ is Hermitian, we define the associated Riemannian symmetric space of $\Sp_2(G,\sigma)$ and provide several models for it.
\end{abstract}

\title{Symplectic groups over Lie subgroups of involutive algebras}

\maketitle
\tableofcontents

\section{Introduction}

The construction and study of symplectic groups over associative, possibly noncommutative, involutive unital algebras was systematically developed in \cite{ABRRW}. That paper introduces $\mathrm{Sp}_2(A,\sigma)$ for an associative (possibly noncommutative) unital algebra $A$ equipped with an anti-involution $\sigma$. The authors further construct the space of isotropic $A$-lines, define generalized cross--ratios and Kashiwara--Maslov type indices, and provide several models of the symmetric space in the Hermitian case. This approach already realizes many classical Lie groups (e.g.\ $\Sp_{2n}(\mathbb R)$, $\mathrm O(n,n)$, $\mathrm U(n,n)$, $\Sp_{2n}(\mathbb C)$, $\Sp(n,n)$) as instances of $\mathrm{Sp}_2(A,\sigma)$. However, certain other classical groups --- notably some Hermitian Lie groups of tube type (e.g.\ groups locally isomorphic to $\SO(2,n)$ or their spin covers) --- are not obtained directly by the constructions presented in \cite{ABRRW}.

A very recent article \cite{HKRW} further develops the study of groups over noncommutative involutive algebras and, complementing the results of \cite{ABRRW}, constructs symmetric-space models for several group constructions (such as $\mathrm O_{(1,1)}(A,\sigma)$, $A^\times$, and $\mathrm O(A_\mathbb C,\sigma_\mathbb C)$) in the Hermitian involutive algebra setting, and applies these models to Higgs--bundle theory by introducing analytic frameworks for polystability and the relevant harmonic equations.

The goal of this article is to introduce and analyse the group $\mathrm{Sp}_2(G,\sigma)$ associated with a Lie subgroup $G$ of the unit group $A^\times$ of a unital real finite-dimensional associative involutive algebra $(A,\sigma)$, and to study the natural geometric spaces on which this group acts. Concretely, starting from $(A,\sigma)$ and a connected Lie subgroup $G_0\le A^\times$ closed under~$\sigma$ with Lie algebra $B=\Lie(G_0)$, we construct its minimal (possibly disconnected) extension $G\le A^\times$ (cf. Section~\ref{discon-ext}) and the space $\mathrm{Sp}_2(G,\sigma)$ as the closure of a certain family of generic $2\times 2$ symplectic matrices built from elements of $G$ and from the $\sigma$-fixed part $B^\sigma$ of the Lie algebra $B$. Under several structural hypotheses on $B$ (for example, the Jordan-type and weakly Hermitian hypotheses; cf.~Definitions~\ref{df:JordanType} and~\ref{df:R-wHerm}), the resulting space $\mathrm{Sp}_2(G,\sigma)$ is a Lie group whose Lie algebra can be identified with the explicitly described Lie algebra
\[
\mathfrak{sp}_2(B,\sigma)
= \left\{\begin{pmatrix} x & z\\[4pt] y & -\sigma(x)\end{pmatrix}
\;\middle|\; x\in B,\; y,z\in B^\sigma\right\}.
\]

\begin{teo}[Theorem~\ref{thm:Sp_2-LieGroup}, Proposition~\ref{prop:LieAlgOfSp2}, Proposition~\ref{prop:Sp2Connected}]
\label{thm:main-intro}
For a Lie subgroup $G\leq A^\times$ as above, assume that its Lie algebra $B$ is of Jordan type with $1\in B$. Then $\mathrm{Sp}_2(G,\sigma)$ is a real Lie group, and, moreover, the Lie algebra of $\mathrm{Sp}_2(G,\sigma)$ agrees with $\mathfrak{sp}_2(B,\sigma)$.
If, in addition, $B$ is weakly Hermitian, then $\mathrm{Sp}_2(G,\sigma)$ is connected.
\end{teo}

We introduce the spaces of $G$-isotropic elements and $G$-isotropic lines and study their configuration invariants (positive triples, positive quadruples, and a generalized cross--ratio), as well as several models for the associated Riemannian symmetric space in the Hermitian case. An example of such a model is the upper half-space:

\begin{teo}
For a Lie subgroup $G\leq A^\times$ as above, assume that its Lie algebra $B$ is Hermitian. Then:
\begin{enumerate}
  \item The \emph{cone of squares} $B^\sigma_+\subset B^\sigma$ is a model for the Riemannian symmetric space of~$G$.
  \item The \emph{upper half-space}
  \[
    \mathfrak U:=\{z\in B^\sigma\otimes_\R\CC \mid \Imm(z)\in B^\sigma_+\}
  \]
  is a model for the Riemannian symmetric space of $\mathrm{Sp}_2(G,\sigma)$.
\end{enumerate}
\end{teo}
The final section presents a family of nontrivial examples, notably realizations of spin groups $\operatorname{Spin}_0(m,n)$ as instances of $\mathrm{Sp}_2(G,\sigma)$ arising from suitable subgroups $G$ inside Clifford algebras and from an appropriate anti-involution $\sigma$ (cf.\ Theorem~\ref{Spin_as_Sp2}). These examples demonstrate that the subgroup perspective is essential to capture important classical groups omitted by the full-algebra approach.

This program is motivated by two intertwined perspectives. First, the construction situates classical Lie groups and some of their modern generalizations within a uniform algebraic framework that blends associative algebra, Jordan-type structure, and Lie theory. Second, the geometric objects associated to $\mathrm{Sp}_2(G,\sigma)$ --- isotropic elements and lines, positive configurations, symmetric-space models, and the Shilov boundary --- form the natural stage for questions about dynamics, boundary maps, and geometric structures related to representation theory (for instance, in the context of Higgs bundles \cite{BGPG,BGPR} and of maximal~\cite{BIW,BILW,AGRW} and positive representations~\cite{GLW,GW2,GRW}). Both aspects are developed in detail in the body of the article.

\vspace{5mm}\noindent\textbf{Outline of the paper.}
In Section~\ref{B-case} we recall the basic algebraic prerequisites: the definition of an algebra with an anti-involution, the class of subalgebras $B\subset A$ closed under~$\sigma$, and the notions of \emph{Jordan type} and \emph{(weakly) Hermitian}, which are needed for the subsequent constructions. The spectral theorem for $B^\sigma:=\Fix_B(\sigma)$, the properties of the cone of squares $B^\sigma_+$, and the polar decomposition are used as key technical tools.

In Section~\ref{Sp_2(G)} we introduce the construction of $\mathrm{Sp}_2(G,\sigma)$. The group is defined in two ways: as the topological closure of a certain set of generic matrices inside $\mathrm{Sp}_2(A,\sigma)$, and as the group generated by a specified family of matrices. Under the standing hypotheses (Jordan type and $(B^\sigma)^\times\subset G$) we prove that $\mathrm{Sp}_2(G,\sigma)$ is a real Lie group and describe its Lie algebra $\mathfrak{sp}_2(B,\sigma)$. We also study connectedness, the center, and the quotient group $\mathrm{PSp}_2(G,\sigma)$.

In Section~\ref{G-lines} we introduce and analyze geometric invariants associated to $G$-isotropic elements and lines: the space of $G$-isotropic elements/lines, the action of $\mathrm{Sp}_2(G,\sigma)$ on these spaces, invariants of pairs, triples and quadruples of isotropic lines (including the generalized cross--ratio), and various types of positive configurations.

Sections~\ref{G-models} and~\ref{hatG-models} construct models of the Riemannian symmetric spaces associated to $\mathrm{Sp}_2(G,\sigma)$ and to $G$ in the case when $B=\Lie(G)$ is Hermitian. We present models via complex structures, projective space models, precompact models, and half-space models, and discuss the compactification and the Shilov boundary.

In Section~\ref{Clifford} we show how the spin groups $\operatorname{Spin}_0(m,n)$ can be realized as special cases of $\mathrm{Sp}_2(G,\sigma)$ by choosing $G$ inside a Clifford algebra. There we construct the spectral decomposition in the Clifford algebra, build the corresponding subalgebra $B$, and verify all necessary prerequisites (Jordan type, cone properties, etc.). These examples show that the proposed construction indeed encompasses new and interesting cases not covered by earlier approaches.

\vspace{5mm}

\textbf{Acknowledgements.}
The author would like to warmly thank Daniele Alessandrini, Arkady Berenstein, Michael Gekhtman, Pengfei Huang, Georgios Kydonakis, Vladimir Retakh, and Anna Wienhard for interesting and fruitful discussions. The author was supported by a postdoc scholarship of the German Academic Exchange Service (DAAD), the funding from the European Research Council (ERC) under the European Union’s Horizon 2020 research and innovation programme (grant agreement No 101018839), and the KIAS Individual Grant (MG100901) at Korea Institute for Advanced Study, and would like to express deep gratitude to the Institut f\"ur Mathematik, Universit\"at Heidelberg, and the Max Planck Institute for Mathematics in the Sciences for their kind hospitality.

\section{Hermitian Lie algebras with anti-involution}\label{B-case}

\subsection{Lie subalgebras of involutive algebras}

Let $A$ be a unital associative, possibly non-commutative, finite-dimensional semisimple algebra over a field $\mathbb K$.

\begin{df}\label{def:antiinv}
An \emph{anti-involution} on $A$ is a $\mathbb{K}$-linear map $\sigma\colon A\to A$ such that
\begin{itemize}
\item $\sigma(ab)=\sigma(b)\sigma(a)$;
\item $\sigma^2=\mathrm{Id}$.
\end{itemize}
An \emph{involutive $\mathbb{K}$-algebra} is a pair $(A,\sigma)$, where $A$ is an $\mathbb{K}$-algebra as above and $\sigma$ is an~anti-involution on $A$.
\end{df}

\begin{rem}
Sometimes in the literature, the maps satisfying Definition~\ref{def:antiinv} are called just \emph{involutions}. We add the prefix ``anti'' in order to emphasize that they exchange factors and to be consistent with~\cite{ABRRW,R-Thesis,HKRW}.
\end{rem}

\noindent The algebra $A$ can be turned into a Lie algebra with the Lie bracket $[x,y]=xy-yx$. Let $B\subseteq A$ be a Lie subalgebra that is closed under $\sigma$. We define:
$$B^{\sigma}:=\Fix_B(\sigma),$$
$$B^{-\sigma}:=\Fix_B(-\sigma).$$

\begin{prop}{\rm \cite[Proposition~2.30]{ABRRW}} If $B$ is a Lie subalgebra of $A$ containing $1$, then $B^\times$ is open and dense in $B$ and $(B^{\sigma})^\times:=B^\times\cap A^{\sigma}$ is open and dense in $B^{\sigma}:=B\cap A^{\sigma}$.
\end{prop}

\begin{df}\label{df:JordanType}
A Lie subalgebra $B$ is called \defin{of Jordan type} (with respect to $\sigma$), if for every $x,y\in B^{\sigma}$, $xy\in B$.
\end{df}

\begin{prop}\label{prop:JordanType}
A Lie subalgebra $B$ is of Jordan type if and only if for every $x\in B$ and for every $b\in B^{\sigma}$, $\sigma(x)b+bx\in B^{\sigma}$.
\end{prop}

\begin{proof} ($\Rightarrow$) Let $B$ be of Jordan type.
We take $x^s:=\frac{x+\sigma(x)}{2}$, $x^a:=\frac{x-\sigma(x)}{2}$, then $x=x^s+x^a$, $x^s\in B^{\sigma}$ and $\sigma(x^a)=-x^a$. Then we can write
$$\sigma(x)b+bx=(x^s-x^a)b+b(x^s+x^a)=(x^sb+bx^s)+(bx^a-x^ab).$$
Since $B$ is of Jordan type, $x^sb,bx^s\in B$ and so $x^sb+bx^s\in B^{\sigma}$. Further, $bx^a-x^ab=[b,x^a]\in B$ and $\sigma(bx^a-x^ab)=-x^ab+bx^a$, i.e., $bx^a-x^ab\in B^{\sigma}$. So we obtain, $\sigma(x)b+bx\in B^{\sigma}$.

\vspace{2mm}
\noindent ($\Leftarrow$) Let $x,y\in B^\sigma$, then
$$2xy=xy+yx+xy-yx=\sigma(x)y+yx+[x,y]\in B.$$
Therefore, $B$ is of Jordan type.
\end{proof}

\begin{rem}
    Let $B$ be of Jordan type. Notice that for $b\in B^\sigma$, $b^k\in B^\sigma$ for every $k$. However, for $b\in B^{-\sigma}$ and $k\in\N_{>1}$, $b^k$ is not necessarily an element of $B$.
\end{rem}

\begin{rem}
    The following properties hold:
    $$[ B^{-\sigma}, B^{-\sigma} ] \subseteq B^{-\sigma},\;\;[ B^{-\sigma}, B^{\sigma} ] \subseteq B^{\sigma},\;\;[ B^{\sigma}, B^{\sigma} ] \subseteq B^{-\sigma}.$$
    In particular, $B^{-\sigma}$ is a Lie subalgebra of $B$.
\end{rem}

\subsection{Weakly Hermitian Lie subalgebras}
Let $B$ be a Lie subalgebra of a real involutive algebra $(A,\sigma)$. We denote by $B_{\geq 0}^\sigma$ the closed convex cone in $B^\sigma$ generated by squares of elements in $B^\sigma$. 

\begin{df}\label{df:R-wHerm}
A Lie subalgebra $B$ of Jordan type (with respect to $\sigma$) is called \defin{weakly Hermitian}, if:
\begin{enumerate}
\item $1\in B$;
\item The convex cone $B^{\sigma}_{\geq 0}$ is proper, i.e., it does not contain lines;
\item\label{R-no_nilp} The space $B^{\sigma}$ does not contain nilpotent elements, i.e., for every $b\in B^{\sigma}$, $b^2=0$ if and only if $b=0$.
\end{enumerate}
\end{df}

If $B$ is real weakly Hermitian, we define $B^{\sigma}_+:=(B^{\sigma}_{\geq 0})^\times.$
Then $B^{\sigma}_{\geq 0}$ is the topological closure of $B^{\sigma}_+$. In this case, $B^{\sigma}_+$ and $B^{\sigma}_{\geq 0}$ are proper convex cones in $B^{\sigma}$. Notice that a~convex cone is proper if and only if for every two its elements $b_1,b_2$, $b_1+b_2=0$ if and only if $b_1=b_2=0$.

\vspace{2mm}

Let $A_\mathbb C$ be the complexification of $A$. Slightly abusing our notation, we denote by $\sigma\colon A_\mathbb C\to A_\mathbb C$ the complex linear extension $\sigma\colon A\to A$, and we denote by $\bar\sigma\colon A\to A$ the composition of the complex anti-involution $\sigma$ and the complex conjugation in $A_\mathbb C$. Notice that $\bar\sigma$ is a real anti-involution on $A_\mathbb C$.

Let $B_\mathbb C$ be a complex Lie subalgebra of $A_\mathbb C$ which is closed under the complex conjugation and under $\sigma$. Notice that $B$ is always a complexification of a unique real Lie subalgebra $B\subseteq A$ closed under the anti-involution $\sigma\colon A\to A$.  As above, we denote by $(B_\mathbb C^{\bar\sigma})_{\geq 0}$ the closed convex cone in $B^{\bar\sigma}$ generated by elements of the form $\bar  b b\in B^{\bar\sigma}$ where $b\in B^{\sigma}$, and by $(B_\mathbb C^{\bar\sigma})_+:=(B_\mathbb C^{\bar\sigma})_{\geq 0}^\times$. 

\begin{prop} Let $B$ be a Lie subalgebra of $(A,\sigma)$ closed under $\sigma$.
\begin{enumerate}
    \item The Lie subalgebra $B$ is of Jordan type with respect to $\sigma$ if and only if $B_\mathbb C$ is of Jordan type with respect to $\sigma$.
    \item The cone $B^\sigma_{\geq 0}$ is proper if and only if the cone $(B^{\bar\sigma}_\mathbb C)_{\geq 0}$ is proper.
\end{enumerate}
\end{prop}

\begin{proof}
    {\it (1)} Let $B$ be of Jordan type. Let $b,b'\in B_\mathbb C^\sigma$ such that $b=b_1+b_2i$, $b'=b'_1+b'_2i$ where $b_1,b_2,b_1',b_2'\in B^\sigma$. Then 
    $$bb'=(b_1b_1'-b_2b_2')+(b_1b_2'+b_2b_1')i\in B^\sigma_\mathbb C.$$
    If $B_\mathbb C$ is of Jordan type, then clearly $B=B_\mathbb C\cap A$ is it as well.

\vspace{2mm}
    \noindent {\it (2)} Let the cone $B^\sigma_{\geq 0}$ be proper. By contradiction, we assume that there exists $0\neq c\in (B^\sigma_\mathbb C)_{\geq 0}$ such that $ct\in (B^\sigma_\mathbb C)_{\geq 0}$ for all $t\in\R$. We can write $c=\bar c_1 c_1+\dots+\bar c_k c_k$ where $c_i=c_i'+c_i''i$, $c_i',c_i''\in B^\sigma$ for all $i$. Then
    $$\Ree(c)=\sum_{i=1}^k\left((c_i')^2+(c_i'')^2\right)\in B^\sigma_{\geq 0}.$$
    Therefore, $\Ree(c)\neq 0$, otherwise, $c_i'=c_i''=0$ for all $i$ because $B^\sigma_+$ is a proper cone, i.e., $c=0$. Therefore, the line $\Ree(c)\R$ is contained in $B^\sigma_{\geq 0}$. This contradicts to the properness of the cone $B^\sigma_{\geq 0}$.

    Let the cone $(B^{\bar\sigma}_\mathbb C)_{\geq 0}$ be proper. Since $B^\sigma_{\geq 0}\subseteq (B^{\bar\sigma}_\mathbb C)_{\geq 0}$, the cone $B^\sigma_{\geq 0}$ contains no line, i.e., it is proper.
\end{proof}

\noindent Notice that, if $B$ is of Jordan type or weakly Hermitian with respect to $\sigma$, then, in general,  it is not of Jordan type with respect to $\bar\sigma$ (cf. Remark~\ref{rem:nonHerm.complexification}).

\vspace{2mm}
We recall the definition of Jordan algebra and formally real Jordan algebra.

\begin{df}
Let $(V,\circ)$ be an possibly non-associative algebra over some field $\K$. $(V,\circ)$ said to be a \defin{Jordan algebra} if for all $x,y\in V$
\begin{enumerate}
\item $x\circ y=y\circ x$;
\item $(x\circ y)\circ(x\circ x)=x\circ(y\circ(x\circ x))$\;\;(Jordan identity).
\end{enumerate}

A Jordan algebra $(V,\circ)$ is called \defin{formally real} if for all $x,y\in V$, $x^2+y^2=0$ implies $x,y=0$.
\end{df}

\begin{prop} Let $B$ be a Lie subalgebra of an involutive algebra $(A,\sigma)$.
\begin{itemize}
\item If $B$ is of Jordan type, then the space $(B^{\sigma},\circ)$ is a Jordan algebra where
$$x\circ y= \frac{xy+yx}{2}.$$
\item If additionally $B$ is weakly Hermitian, then the Jordan algebra $(B^{\sigma},\circ)$ is formally real.
\end{itemize}
\end{prop}

\begin{proof}
Since for all $x,y\in B^{\sigma}$, $xy\in B$, so we obtain $x\circ y\in B^{\sigma}$. Moreover, $x\circ y=y\circ x$ is clear. Further, the Jordan identity holds:
$$(x\circ y)\circ(x\circ x)=\frac{xy+yx}{2}\circ x^2=\frac{xyx^2+yx^3+x^3y+x^2yx}{4}=$$
$$=\frac{xyx^2+x^3y+yx^3+x^2yx}{4}=x\circ \frac{yx^2+x^2y}{2}=x\circ(y\circ(x\circ x)).$$
Thus $(B^{\sigma},\circ)$ is a Jordan algebra.

Assume now $B$ to be weakly Hermitian. Let $a_1, a_2\in B^{\sigma}$, then $a_i^2\in B^{\sigma}_{\geq 0}$. The convex cone $B^{\sigma}_{\geq 0}$ is proper, so $a_1^2+a_2^2$ vanishes if and only if $a_1^2=a_2^2=0$. Therefore, $a_1=a_2=0$ by~(\ref{R-no_nilp}) in Definition~\ref{df:R-wHerm}.
\end{proof}

\subsection{Classification of simple formally real Jordan algebras}

In this section, we remind the well-known classification of simple formally real Jordan algebras (for more details, see~\cite{Hanche84, Faraut}).

\begin{fact}
Every simple formally real Jordan algebra is isomorphic to one of the following Jordan algebras:
\begin{enumerate}
\item $(\Sym(n,\R),\circ)$ where $a\circ b=\frac{ab+ba}{2}$ for $a,b\in \Sym(n,\R)$;
\item $(\Herm(n,\CC),\circ)$ where $a\circ b=\frac{ab+ba}{2}$ for $a,b\in \Herm(n,\CC)$;
\item $(\Herm(n,\HH),\circ)$ where $a\circ b=\frac{ab+ba}{2}$ for $a,b\in \Herm(n,\HH)$;
\item $(B^{\sigma}(1,n),\circ)$ where $a\circ b=\frac{ab+ba}{2}$ for $a,b\in B^{\sigma}(1,n)$ (cf.~Section~\ref{ex:Clifford});
\item $(\Herm(3,\Oc),\circ)$ where $a\circ b=\frac{ab+ba}{2}$ for $a,b\in \Herm(3,\Oc)$
\end{enumerate}
where $\Herm(3,\Oc)$ is the space of $3\times 3$ Hermitian octonionic matrices.
\end{fact}

\begin{fact}{\rm \cite[Corollary~2.8.5]{Hanche84}}
The Jordan algebra $(\Herm(3,\Oc),\circ)$ is exceptional. This means that there is no associative real algebra $A$ that contains $\Herm(3,\Oc)$ as a Jordan subalgebra.
\end{fact}

\subsection{Spectral theorem}

In this section, we assume $B$ to be weakly Hermitian. As we have seen, $(B^{\sigma},\circ)$ is a formally real Jordan algebra.

We are going to state the first versions of the spectral theorem for formally real Jordan algebras. But before we do it, first, we give some necessary definitions:

\begin{df}\begin{itemize}
\item An element $c\in B^{\sigma}$ is called an \defin{idempotent} if $c^2=c$.
\item Two idempotents $c,c'\in B^{\sigma}$ are called \defin{orthogonal} if $c\circ c'=0$.
\item A tuple $(c_1,\dots,c_k)$ of pairwise orthogonal idempotents is called a \defin{complete orthogonal system of idempotents} if $c_1+\dots+c_k=1$.
\end{itemize}
\end{df}

\begin{rem}
Every idempotent $c\in B^{\sigma}_{\geq 0}$.
\end{rem}

\begin{teo}[Spectral theorem, first version~{\cite[Theorem~III.1.1]{Faraut}}]\label{Spec_teo_B1}
For every $b\in B^{\sigma}$, there exist a unique $k\in\N$, unique real numbers $\lambda_1,\dots,\lambda_k\in\R$, all distinct, and a unique complete system of orthogonal idempotents $c_1,\dots,c_k\in B^{\sigma}$ such that
$$b=\sum_{i=1}^k\lambda_ic_i.$$
\end{teo}

\begin{cor}\label{theta_Bsym+}
For $b\in B^{\sigma}_{\geq 0}$, the numbers $\lambda_1,\dots,\lambda_k\geq 0$. For $b\in B^{\sigma}_+$, the numbers $\lambda_1,\dots,\lambda_k > 0$. In particular,
$$B^{\sigma}_+=\{b^2\mid b\in (B^\sigma)^\times\}
,\; B^{\sigma}_{\geq 0}=\{b^2\mid b\in B^\sigma\},$$
i.e., the sets of squares already build cones.
\end{cor}

\begin{cor}
The set of all invertible elements $(B^{\sigma})^\times$ of $B^{\sigma}$ consists of elements such that all $\lambda_i\neq 0$. If all $\lambda_i\neq 0$, then
$$\left(\sum_{i=1}^k\lambda_ic_i\right)^{-1}=\sum_{i=1}^k\lambda_i^{-1}c_i.$$
\end{cor}

\begin{cor}
The cone $B^{\sigma}_+$ is open in $B^{\sigma}$, open and closed in $(B^{\sigma})^\times$. The cone $B^{\sigma}_{\geq 0}$ is closed in $B^{\sigma}$.
\end{cor}

\begin{cor}
For every (continuous or smooth) function $f\colon \R \to \R$, the (continuous, resp. smooth) map
$$\hat f\colon B^{\sigma} \to B^{\sigma}$$
can be defined: if
$$b=\sum_{i=1}^k\lambda_ic_i,$$
then
$$\hat f(b):=\sum_{i=1}^k f(\lambda_i)c_i.$$
This map is well defined because the spectral decomposition is unique. Analogously, for any function $f\colon \R_{\geq 0} \to \R$ or $f\colon \R_+ \to \R$, $\hat f\colon B^{\sigma}_{\geq 0} \to B^{\sigma}$ resp. $\hat f\colon B^{\sigma}_+ \to B^{\sigma}$ can be defined.

In particular, for every $b\in B^{\sigma}_{\geq 0}$, the element $b^t\in B^{\sigma}_{\geq 0}$ for $t>0$ is well-defined. This definition is compatible with integer powers of element.
\end{cor}

\begin{cor}\label{Bsym+_contr_open}
The space $B^{\sigma}_+$ is homeomorphic to $B^{\sigma}$, more precisely $\exp(B^{\sigma})=B^{\sigma}_+$. In particular, $B^{\sigma}_+$ is open in $B^{\sigma}$ and contractible. $\{1\}\subset B^{\sigma}_+$ is a deformation retract of $B^{\sigma}_+$.
\end{cor}

\begin{proof}
As an $\R$-vector space, $B^{\sigma}$ is contractible. $\{0\}\subset B^{\sigma}$ is a deformation retract of $B^{\sigma}$. The map $\hat f\colon B^{\sigma}\to B^{\sigma}_+$ for $f(t)=\exp(t)$ provides a homeomorphism.
\end{proof}

To state the second version of the spectral theorem, we need to give some additional definitions:

\begin{df}\begin{itemize}
\item An idempotent $0\neq c\in B^{\sigma}$ is called \defin{primitive} if it cannot be written as a sum of two orthogonal non-zero idempotents.
\item A complete orthogonal system of primitive idempotents $(c_1,\dots,c_k)$ is called a \defin{Jordan frame}.
\end{itemize}
\end{df}

\begin{teo}[Spectral theorem, second version~{\cite[Theorem~III.1.2]{Faraut}}]\label{Spec_teo_B2}
Suppose, $B^{\sigma}$ has rank $n$. For every $b\in B^{\sigma}$ there exist a Jordan frame $(e_1,\dots,e_n)$ and real numbers $\lambda_1,\dots,\lambda_n\in\R$ such that
$$b=\sum_{i=1}^n\lambda_ie_i.$$
The numbers $\lambda_1,\dots,\lambda_n\in\R$ (with their multiplicities) called \defin{eigenvalues} of $b$ are uniquely determined by $b$. In particular, they do not depend (up to permutations) on the Jordan frame $e_1,\dots,e_n\in B^{\sigma}$.
\end{teo}

\begin{rem}
A Jordan frame $e_1,\dots,e_n\in B^{\sigma}$ associated to the element $b\in B^{\sigma}$ as in Theorem~\ref{Spec_teo_B2} is, in contrast to the complete system of orthogonal idempotents from Theorem~\ref{Spec_teo_B1}, in general not unique. In fact, it is unique if and only if all eigenvalues of $b$ are distinct.
\end{rem} 

\begin{prop} Let $B$ be weakly Hermitian.
    Let $b_1,b_2\in B^\sigma$ be two commuting elements. Then there exists an element $b\in\R[b_1,b_2]$ such that $\R[b_1,b_2]=\R[b]$. in particular, $b,b_1,b_2$ share a Jordan frame.
\end{prop}

\begin{proof}
    The associative $\R$-algebra $\R[b_1,b_2]$ is commutative and does not contain any nilpotent elements because $\R[b_1,b_2]\subseteq B^\sigma$. This implies that $\R[b_1,b_2]$ is semisimple, and by Wedderburn's theorem, it is isomorphic to a finite direct sum of copies of $\R$ and copies of $\CC$. In particular, there exists a generating element $b\in \R[b_1,b_2]$. Since $b\in B^\sigma$, we apply the spectral theorem: $b=\sum_{i=1}^n \lambda_i e_i$ for some $\lambda_i\in\R$ and $e=(e_1,\dots,e_n)$ a Jordan frame. Since $b_1,b_2$ are polynomials in $b$, they share the Jordan frame $e$.   
\end{proof}

\begin{df}\label{tr_det}
Let $b\in B^{\sigma}$ and $\lambda_1,\dots,\lambda_n$ are all its eigenvalues (with multiplicities). We define the \defin{trace} and the \defin{determinant} of $b$:
$$\tr(b):=\sum_{i=1}^n \lambda_i,\; \det(b):=\prod_{i=1}^n \lambda_i.$$
\end{df}

\begin{prop}[{\cite[Proposition~III.1.5]{Faraut}}]\label{inner_prod_B}
The following map defines an inner product on the $\R$-vector space $B^{\sigma}$:
\begin{align*}
\beta\colon  B^{\sigma}\times B^{\sigma} & \to \; \R\\
 (b_1,b_2)\; & \mapsto  \tr(b_1\circ b_2).
\end{align*}
\end{prop}

\subsection{Lie group corresponding to weakly Hermitian Lie algebra}\label{Herm_Lie}

As before, let $(A,\sigma)$ be an $\R$-algebra with an anti-involution. The space $A^\times$ of invertible elements of $A$ is a Lie group and its Lie algebra is $A$ with the Lie bracket given by $[x,y]=xy-yx$. We take $G_0<A^\times$ a connected Lie subgroup of $A^\times$ closed under $\sigma$ with the Lie algebra $B:=\Lie(G_0)$. Notice that $G_0$ is uniquely defined by $B$, and it is generated by $\exp(B)$. Since $G_0$ is closed under $\sigma$, $B$ is closed under $\sigma$ as well. We denote: 
\begin{align*}
    G^{\sigma}_0& :=G_0\cap A^{\sigma},\\
    \OO_1(G_0,\sigma):=\OO(G_0,\sigma)& :=\{u\in G_0\mid \sigma(u)u=1\}.    
\end{align*}

\begin{rem}
The Lie algebra of $\OO(G_0,\sigma)$ agrees with $B^{-\sigma}$. Moreover, $\OO(G_0,\sigma)$ acts on Jordan frames of $B^\sigma$ by the adjoint action and, therefore, it acts on $B^\sigma$ preserving the form $\beta$. However, $\OO(G_0,\sigma)$ is in general not compact and the action on Jordan frames of $B^\sigma$ might be non-transitive.
\end{rem}

\begin{df}
    A Lie subalgebra $B$ is called \defin{Hermitian} if it is weakly Hermitian, $\OO(G_0,\sigma)$ is compact, and it acts transitively on Jordan frames of $B^\sigma$.
\end{df}

\begin{prop}\label{prop:JordanType+}
Let $B$ be of Jordan type. For every $g\in G_0$ and for every $b\in B^{\sigma}$, $\sigma(g)bg\in B^{\sigma}$.
\end{prop}

\begin{proof}
We consider the following map:
$$\begin{array}{rcl}
F\colon \OO(G_0,\sigma)\times \exp(B^{\sigma}) & \to & G_0\\
\quad (u,b) & \mapsto & ub
\end{array} $$
We notice that since for all $b\in B^{\sigma}$, $b^2\in B^{\sigma}$, we have $b^n\in B^{\sigma}$ for all positive integers $n$. In~particular, $\exp(B^\sigma)\subseteq G^\sigma_0\cap B^\sigma$. Moreover, $B^{\sigma}$ is closed in $A$, therefore, $\exp(b)-1\in B^{\sigma}$. Since $\exp$ is a~local diffeomorphism at $0\in B^{\sigma}$, $T_1\exp(B^{\sigma})= B^{\sigma}$.

The differential of $F$ at $(1,1)$ is a bijection. Indeed:
$$D_{(1,1)}F(x,y)=x+y\in B$$
where $x\in B^{-\sigma}=\Lie(\OO(G_0,\sigma))$, $y\in B^{\sigma}=T_1\exp(B^{\sigma})$.
Therefore, in a small neighborhood $V$ of $(1,1)\in \OO(G_0,\sigma)\times B^{\sigma}_+$, the map $F$ is a homeomorphism. In particular, $G_0$ is generated by $F(V)$, i.e., for every $g\in G_0$, there exist $r\geq 0$ and $u_1,\dots,u_r\in \OO(G_0,\sigma)$, $b_1,\dots,b_r\in \exp(B^{\sigma})$ such that $g=u_1b_1\dots u_rb_r$.

Since $\sigma(u)=u^{-1}$ for $u\in \OO(G_0,\sigma)$, $\sigma(u)bu=u^{-1}bu=\mathrm{Ad}(u^{-1})b\in B^{\sigma}$ for all $b\in B^{\sigma}$.

Let $b\in B^{\sigma}$, $b'\in B^{\sigma}_+$, then $b'=1+b_0$ for $b_0\in B^{\sigma}$. $$\sigma(b')bb'=b'bb'=(1+b_0)b(b+b_0)=b+b_0b+bb_0+b_0bb_0.$$
By Proposition~\ref{prop:JordanType}, $b_0b+bb_0=\tilde b$ for $\tilde b\in B^{\sigma}$. Therefore, $b_0bb_0=b_0\tilde b-(b_0)^2b\in B$ and, since $b'bb'\in A^{\sigma}$, we obtain $b'bb'\in B^{\sigma}$.

Finally, by induction, we obtain $\sigma(g)bg\in B^{\sigma}$ for all $g\in G_0$.
\end{proof}

 Thus, the following maps define a right and a left actions of $G_0$ on $B^\sigma$:
$$\begin{array}{rcl}
\psi\colon  G_0 & \to & \Aut(B^{\sigma})\\
 g & \mapsto & [\psi(g)\colon b\mapsto\sigma(g)bg],
\end{array}$$
$$\begin{array}{rcl}
\psi'\colon  G_0 & \to & \Aut(B^{\sigma})\\
 g & \mapsto & [\psi'(g)\colon b\mapsto gb\sigma(g)].
\end{array}$$

\begin{prop}\label{prop:PositiveElts}
If $B$ is weakly Hermitian, then $B^{\sigma}_+\subseteq G_0$.
\end{prop}

\begin{proof}
Let $b\in B^{\sigma}_+$. Take its spectral decomposition: $b=\sum_{i=1}^k\lambda_ic_i$ where $\lambda_1,\dots,\lambda_k>0$, $(c_1,\dots,c_k)$ is a complete system of orthogonal idempotents. Then 
\begin{align*}
\log(b)&=\sum_{i=1}^k\log(\lambda_i)c_i\in B^{\sigma}
\end{align*}
and $\exp(\log(b))=b\in G_0$ because the map $\exp$ defined on $\R$ extended to $B^{\sigma}$ and the exponential map $\exp\colon B\to G$ restricted to $B^{\sigma}$ defined by the same power series and thus they agree.
\end{proof}

\begin{cor} Let $B$ be weakly Hermitian.
\begin{itemize}
\item For every $b\in B^{\sigma}_+$ and for every $g\in G_0$, $\sigma(g)bg\in B^{\sigma}_+$.
\item For every $b\in B^{\sigma}_{\geq 0}$ and for every $g\in G_0$,
$\sigma(g)bg\in B^{\sigma}_{\geq 0}$.
\end{itemize}
In particular, $\sigma(g)g\in B^{\sigma}_+$.
\end{cor}

\begin{proof} It is clear that $\sigma(g)bg\in (B^{\sigma})^\times$ for $g\in G_0$ and $b\in B^{\sigma}_+$. Since $G_0$ is connected, $\sigma(g)g$ is in the connected component of $1\in (B^{\sigma})^\times$ which is $B^{\sigma}_+$.

\noindent The second one follows from the fact that $B^{\sigma}_{\geq 0}$ is a topological closure of $B^{\sigma}_+$ in $B^{\sigma}$.
\end{proof}

\begin{teo}[Polar decomposition for $G_0$]\label{pol_decomp0}
Let $B$ be weakly Hermitian. The following maps are homeomorphisms:
$$\begin{matrix}
\pol\colon & \OO(G_0,\sigma)\times B^{\sigma}_+ & \to & G_0 \\
& (u,b) & \mapsto & ub,
\end{matrix}$$
$$\begin{matrix}
\widetilde{\pol}\colon& B^{\sigma}_+ \times \OO(G_0,\sigma) & \to & G_0 \\
&(b,u) & \mapsto & bu.
\end{matrix}$$
\end{teo}

\begin{proof} We prove the statement for the map $\pol$.
This map is well-defined because $B^{\sigma}_+\subseteq G_0$. 

First, we prove surjectivity: Let $g\in G$, then $\sigma(g)g\in B^{\sigma}_+$. We define $b:=(\sigma(g)g)^{\frac{1}{2}}$, then $u:=g(\sigma(g)g)^{-\frac{1}{2}}\in \OO(G_0,\sigma)$. Indeed,
$$\sigma(u)u=(\sigma(g)g)^{-\frac{1}{2}}\sigma(g)g(\sigma(g)g)^{-\frac{1}{2}}=1.$$

Further, we prove injectivity: Let $g=ub=u'b'$ where $u,u'\in \OO(G,\sigma)$, $b,b'\in B^{\sigma}_+$. Then $\sigma(g)g=(b')^2=b^2\in B^{\sigma}_+$. We take the spectral decompositions of $b$ and $b'$:
$$b=\sum_{i=1}^k\lambda_i c_i,\;b'=\sum_{i=1}^{k'}\lambda_i' c_i'$$
where all $k,k'\in \N$, $\lambda_i,\lambda_i'>0$ and $(c_i)_{i=1}^k$, $(c'_i)_{i=1}^{k'}$ complete orthogonal systems of idempotents of $B^{\sigma}$. Then
$$b^2=\sum_{i=1}^k\lambda_i^2 c_i=\sum_{i=1}^{k'}(\lambda_i')^2 c_i'=(b')^2.$$
Because of the uniqueness of the spectral decomposition, $k=k'$ and, up to reordering, all $\lambda_i^2=(\lambda_i')^2$, $c_i=c_i'$. But all $\lambda_i>0$, therefore, $\lambda_i=\lambda_i'$, i.e., $b=b'$ and $$u=gb^{-1}=g(b')^{-1}=u'.$$

Finally, by definition, $\pol$ is continuous. Moreover, $$\pol^{-1}(g)=(g(\sigma(g)g))^{-\frac{1}{2}},(\sigma(g)g))^\frac{1}{2})$$
is continuous as well. Therefore, $\pol$ is a homeomorphism.
\end{proof}

\begin{cor}
    The group $\OO(G_0,\sigma)$ is a strong deformation retract of $G_0$. In particular, $\OO(G_0,\sigma)$ is connected. Moreover, if $\OO(G_0,\sigma)$ is compact, it is a maximal compact subgroup of $G_0$.
\end{cor}

Note that the restrictions of actions $\psi$ and $\psi'$ to $\OO(G_0,\sigma)$ agree with the adjoint actions of $\OO(G_0,\sigma)$ on $B^\sigma$. In particular, they preserve Jordan frames.

\begin{cor}\label{Spec_teo_B2a} Assume, $\psi|_{\OO(G_0,\sigma)}$ is transitive on Jordan frames of $B^{\sigma}$. Suppose, $B^{\sigma}$ has rank~$n$. For every Jordan frame $e_1,\dots,e_n\in B^{\sigma}$ and for every $b\in B^{\sigma}$ there exist $u\in \OO(G_0,\sigma)$ such that
$$\psi(u)b=\sum_{i=1}^n\lambda_ie_i,$$
where $\lambda_1,\dots,\lambda_n\in\R$ are all eigenvalues of $b$ (with their multiplicities).
\end{cor}

\begin{rem}
In general, for a fixed Jordan frame $e_1,\dots,e_n$ and $b\in B^{\sigma}$, the element $u\in \OO(G_0,\sigma)$ is not unique.
\end{rem}

\begin{cor}
For $b\in B^{\sigma}_{\geq 0}$, the numbers $\lambda_1,\dots,\lambda_n\geq 0$. For $b\in B^{\sigma}_+$, the numbers $\lambda_1,\dots,\lambda_n > 0$.
\end{cor}

\subsection{Minimal extension of \texorpdfstring{$G_0$}{G0}}\label{discon-ext}

Let $B\subseteq A$ be weakly Hermitian Lie subalgebra of $A$. As we have seen, the group $G_0$ is supposed to be connected and $B^{\sigma}_+\subseteq G_0$ (cf. Proposition~\ref{prop:PositiveElts}). In this section, we study the subgroup $G\leq A^\times$ which is generated by $G_0$ and $(B^{\sigma})^\times$. We call $G$ the \defin{minimal extension} of $G_0$.

\begin{prop}
The group $G_0$ is a normal subgroup of $G$.
\end{prop}

\begin{proof}
It is enough to show that $b^{-1}gb\in G_0$ for all $b\in (B^{\sigma})^\times$, $g\in G_0$. Since $G_0$ is generated by $\exp(B)$, it is enough to check it for all $g=\exp(b')$ for $b'\in B$. Notice that in this case, $b^{-1}gb=\exp(b^{-1}b'b)$.

By Proposition~\ref{prop:JordanType}, $b'b+bb'=\tilde b\in B^{\sigma}$. Therefore, $b^{-1}b'b=b^{-1}\tilde b-b'$. Since $B$ is of Jordan type, $b^{-1}\tilde b\in B$. Therefore, $b^{-1}b'b\in B$ and $\exp(b^{-1}b'b)\in G_0$.
\end{proof}

From now on, we assume that $\OO(G_0,\sigma)$ acts transitively on Jordan frames of $B^{\sigma}$. In particular, the rest of this section holds for Hermitian $B$.

\begin{teo}\label{Conn_comp_of_G}
The factor group $G/G_0$ is finite. In particular, $G$ has finitely many connected components and $G_0$ is one of them containing $1$. In every connected component of $G$, there is an element of $B^{\sigma}$.
\end{teo}

\begin{proof}
Let $g\in G$, then by definition of the group $G$, there exist $g_0,g_1,\dots, g_r\in G_0$, $b_1,\dots, b_r\in (B^{\sigma})^\times$ such that $g=g_0b_1g_1\dots b_rg_r$. We take such presentation with minimal $r$. We choose a Jordan frame $(e_1,\dots,e_n)$ of $B^{\sigma}$ and take a spectral decomposition according Corollary~\ref{Spec_teo_B2a}:
$$b_i=u^{-1}_i\sum_{j=1}^n\varepsilon_{ij}\lambda_{ij} e_ju_i=\left(u^{-1}_i\sum_{j=1}^n\varepsilon_{ij}e_j u_i\right)\left(u^{-1}\sum_{j=1}^n\lambda_{ij} e_ju_i\right)$$
where all $\lambda_{ij}>0$, $\varepsilon_{ij}\in\{1,-1\}$, $u_i\in \OO(G_0,\sigma)$. We denote:
$$b_i':=\sum_{j=1}^n\lambda_{ij} e_j,\;s_i:=\sgn(b_i):=\sum_{j=1}^n\varepsilon_{ij}e_j.$$
Notice, $b_i'\in B^{\sigma}_+\subseteq G_0$, $s_i\in (B^{\sigma})^\times$. So we obtain:
$$g=g_0b_1g_1\dots g_{r-1}b_rg_r=g_0b_1g_1\dots g_{r-2}u_{r-1}^{-1}b_{r-1}'s_{r-1}u_{r-1} g_{r-1}u_r^{-1}b_r's_ru_rg_r.$$
We denote $g_r':=u_rg_r$, $g_{r-1}':=u_{r-1} g_{r-1}u_r^{-1}b_r'\in G_0$, $g_{r-2}':=g_{r-2}u_{r-1}^{-1}b_{r-1}'\in G_0$. Then
$$g= g_0b_1g_1\dots g_{r-2}'s_{r-1}g'_{r-1}s_rg'_r=b_1g_1\dots g_{r-2}'s_{r-1}s_rs_r^{-1}g'_{r-1}s_rg'_r.$$
Since $G_0$ is a normal subgroup in $G$, $g''_{r-1}:=s_r^{-1}g'_{r-1}s_rg_r\in G_0$. Moreover, 
$$s'_{r-1}:=s_{r-1}s_r=\sum_{j=1}^n\varepsilon_{r-1,j}\varepsilon_{rj}e_j\in B^{\sigma}.$$
So we obtain:
$$g=g_0b_1g_1\dots g_{r-2}'s'_{r-1}g''_{r-1}.$$
So we reduced the number $r$. Therefore, $g$ can be written as
$$g=g_0b_1g_1=g_0u_1^{-1}b_1's_1u_1g_1=g_0's_1g_1'$$
where $g_0'=g_0u_1^{-1}b_1'$, $g_1'=u_1g_1$. Further,
$$g=g_0's_1g_1'=s_1s_1^{-1}g_0's_1g_1'=s_1g'$$
where $g':=s_1^{-1}g_0's_1g_1'\in G_0$ because $G_0$ is a normal subgroup in $G$. Therefore, $gG_0=s_1G_0$. Consider the group
$$S:=\left\{\sum_{i=1}^n\varepsilon_ie_i\mid \varepsilon_i\in\{1,-1\}\right\}\subset (B^{\sigma})^\times.$$
This is a finite abelian subgroup of $A^\times$ isomorphic to $(\Z/2\Z)^n$. The map 
$$S\ni s\mapsto sG_0\in G/G_0$$ 
is a surjective group homomorphism. Therefore, $G/G_0$ is finite. In particular, the dimensions $\dim(G)$ and $\dim(G_0)$ agree, thus $G_0$ is open in $G$. $G_0$ is closed in $A^\times$ as a Lie subgroup. Therefore, $G_0$ is also closed in $G$, i.e., $G_0$ is a connected component of $G$. Every connected component of $G$ has form $sG_0$ for $s\in S$, so $G$ has finitely many connected components, and, since $d\in B^{\sigma}$, in every connected component of $G$ there is an element from $B^{\sigma}$.
\end{proof}

\begin{cor}
If $B$ is Hermitian, then $\OO(G,\sigma)=\{g\in G\mid \sigma(g)g=1\}$ is compact.
\end{cor}

\begin{prop}
The following maps define a right and a left actions of $G$ on $B^\sigma$:
$$\begin{matrix}
\psi\colon & G\times B^{\sigma} & \mapsto & B^{\sigma}\\
& (g,b) & \to & \sigma(g)bg
\end{matrix}$$
$$\begin{matrix}
\psi'\colon & G\times B^{\sigma} & \mapsto & B^{\sigma}\\
& (g,b) & \to & gb\sigma(g)
\end{matrix}$$
preserving $B^{\sigma}_+$ and $B^{\sigma}_{\geq 0}$. In particular, $\sigma(g)g\in B^{\sigma}_+$ for any $g\in G$.
\end{prop}

\begin{proof} We prove the proposition for the map $\psi$. The statement about $\psi'$ can be proven similarly. First, we note that every element $g\in G$ can be written as $g=sg_0$ for $g_0\in G_0$ and $s\in S$ from the proof of the previous theorem.

Since the construction of $G$ does not depend on the choice of the Jordan frame $(e_1,\dots e_n)$ from the proof of the previous theorem, we assume this basis corresponds to the spectral decomposition of $b$, i.e., $b=\sum_{i=1}^n\lambda_ie_i$. Then $g=sg_0$ for $s=\sum_{i=1}^n\varepsilon_ie_i$, $g_0\in G_0$.

Then $\sigma(g)bg=\sigma(g_0)(sbs)g_0$. But $sbs=\sum_{i=1}^n\varepsilon_i^2\lambda_ie_i=b\in B$. Therefore, 
$$\sigma(g)bg=\sigma(g_0)bg_0\in B^{\sigma}$$
because $G$ acts on $B^{\sigma}$ in this way.

From the same reason, if $b\in B^{\sigma}_+$ or $b\in B^{\sigma}_{\geq 0}$, then $\sigma(g)bg\in B^{\sigma}_+$ resp. $\sigma(g)bg\in B^{\sigma}_{\geq 0}$.
\end{proof}

\subsection{Polar decomposition in \texorpdfstring{$G$}{G} and maximal compact subgroup of \texorpdfstring{$G$}{G}}
In the section, we assume the Lie algebra $B\subseteq A$ to be weakly Hermitian, and $G$ to be either the connected group $G_0$ or its minimal extension $G$ from Section~\ref{discon-ext}. 

Similarly to Theorem~\ref{pol_decomp0}, the following theorem can be proven:
\begin{teo}[Polar decomposition, first version]\label{pol_decomp1}
The following two maps define homeomorphisms:
$$\begin{matrix}
\pol\colon & \OO(G,\sigma)\times B^{\sigma}_+ & \to & G \\
& (u,b) & \mapsto & ub,
\end{matrix}$$
$$\begin{matrix}
\widetilde{\pol}\colon& B^{\sigma}_+ \times \OO(G,\sigma) & \to & G \\
&(b,u) & \mapsto & bu.
\end{matrix}$$
\end{teo}

\begin{cor}\label{Max_Comp_G}
The group $\OO(G,\sigma)<G$ is a strong deformation retract of $G$. In particular, if $\OO(G,\sigma)$ is compact, it is a maximal compact subgroup of $G$.
\end{cor}

\begin{cor}\label{pol_decomp_A}
The polar decomposition~\ref{pol_decomp1} as well as Corollary~\ref{Max_Comp_G} hold also for any Lie subgroup $G\leq A^\times$ such that $\Lie(G)=B$. In particular, it holds in the case $B=A$ for a Hermitian algebra $A$.
\end{cor}

\begin{cor}
    Every element $g\in G^\sigma$ can be uniquely written as $g=ub$ where $b\in B^\sigma_+$ and $u\in \OO(G,\sigma)\cap G^\sigma$ such that $u^2=1$ and $ub=bu$.
\end{cor}

\begin{proof}
    Let $g=ub\in G^\sigma$ be the polar decomposition of $g$ where $b\in B^\sigma_+$, $u\in \OO(G,\sigma)$. Then $g=\sigma(g)=bu^{-1}$ and, therefore, $b=ubu=u^2(u^{-1}bu)$. Since $u^2\in \OO(G,\sigma)$ and $u^{-1}bu\in B^\sigma_+$, this is the polar decomposition of $b$. By its uniqueness, $u^2=1$ and $u^{-1}bu=b$. In particular, $u\in G^\sigma$, and $b$ and $u$ commute.
\end{proof}

\begin{teo}[Polar decomposition, second version]
Let $G$ be the minimal extension of $G_0$ as in Section~\ref{discon-ext}. The map
$$\begin{matrix}
\pol'\colon & \OO(G_0,\sigma)\times (B^{\sigma})^\times & \to & G \\
& (u,b) & \mapsto & ub
\end{matrix}$$
is surjective and continuous.
\end{teo}

\begin{proof}
Let $g\in G$. By Theorem~\ref{pol_decomp1}, $g=ub_0$ for some $u\in \OO(G,\sigma)$, $b_0\in B^{\sigma}_+$. We fix a Jordan frame $e=(e_i)_{i=1}^n$ such that $b_0=\sum_{i=1}^n\lambda_ie_i$ and take a group
$$S_e:=\left\{\sum_{i=1}^n\varepsilon_ie_i\mid \varepsilon_i\in\{1,-1\}\right\}\subset (B^{\sigma})^\times\cap \OO(G,\sigma).$$
Then, as we have seen in the proof of Theorem~\ref{Conn_comp_of_G}, every connected component of $G$ contains an element form $S_e$. Moreover, since $\OO(G,\sigma)$ is a deformation retract of $G$, and $S_e\in \OO(G,\sigma)$, every connected component of $\OO(G,\sigma)$ contains an element of $S_e$. Therefore, there exists $s\in S_e$ such that $u=u_0s$ for an $u_0\in \OO(G_0,\sigma)$. Then
$$g=ub_0=u_0sb_0=:u_0b$$
for $b:=sb_0=\sum_{i=1}^n\varepsilon_i\lambda_ie_i\in (B^{\sigma})^\times$.
\end{proof}

\begin{rem}
The map $\pol'$ is in general not injective. For example, if we take
$$B=A=\Mat(2,\R)$$
with $\sigma$ to be the transposition, then
$$\OO(G_0,\sigma)=\SO(2,\R),\; B^{\sigma}=\Sym(2,\R).$$
Then the matrix
$$\Id=u_1b_1=u_2b_2$$
for $u_1=b_1=\Id$, $u_2=b_2=-\Id$. The reason for that is the fact that $\OO(G_0,\sigma)\cap B^{\sigma}\neq \{1\}$.
\end{rem}

Let $\bar G$ be the topological closure of $G$ in $A$.
\begin{prop}
$\bar G$ is a monoid.
\end{prop}

\begin{proof}
Let $g,g'\in \bar G\subseteq A$, then $gg'\in A$. We want to show that there exists $\{h_i\}\subset G$ such that $\lim h_i=gg'$. Since $g,g'\in\bar G$, there exist $\{g_i\},\{g'_i\}\subset G$ such that $\lim g_i=g$, $\lim g'_i=g'$. Take $h_i=g_ig'_i\in G$, then
\begin{equation*}
\lim h_i=\lim g_ig'_i=\lim g_i\lim g'_i=gg'.\qedhere
\end{equation*}
\end{proof}

By taking closure in the polar decomposition, we get the following map:
$$\begin{matrix}
\bpol\colon & \OO(G,\sigma)\times B^{\sigma}_{\geq 0} & \to & \bar G \\
& (u,b) & \mapsto & ub.
\end{matrix}$$
This map is not a homeomorphism anymore, but it is surjective. If $B$ is Hermitian, it is also proper because $\OO(G,\sigma)$ is compact.   We define the following surjective map:
$$\begin{matrix}
\theta\colon &\bar G &\to &B^{\sigma}_{\geq 0}\\
& g &\mapsto & \sigma(g)g.
\end{matrix}$$
The map $\theta$ maps surjectively $G$ to $B^\sigma_+$.

\begin{prop}
If $B$ is Hermitian, the map $\theta\colon \bar G\to B^{\sigma}_{\geq 0}$ is proper.
\end{prop}

\begin{proof}
Let $K\subset B^{\sigma}_{\geq 0}$ be a compact subset. Then
$$\theta^{-1}(K)=\{ub^{\frac{1}{2}}\mid u\in \OO(G,\sigma), b\in K\}=
\bpol(\OO(G,\sigma)\times K).$$
Since $\OO(G,\sigma)\times K$ is compact in $\OO(G,\sigma)\times B^{\sigma}_{\geq 0}$ and the map $\bpol$ is continuous, the set $\theta^{-1}(K)$ is compact.
\end{proof}

\begin{prop}\label{comp_disc}
If $B$ is a Hermitian subalgebra in $A$, the following spaces
    \begin{align*}
        D(\bar G,\sigma):=& \{a\in \bar G\mid 1-\sigma(a)a\in B^{\sigma}_{\geq 0}\}\subseteq \bar G\\
        D(B_\mathbb C^\sigma,\bar\sigma):=& \{a\in  B^{\sigma}_\mathbb C\mid 1-\bar aa\in (B_\mathbb C^{\bar\sigma})_{\geq 0}\}\subseteq B_\mathbb C^\sigma
    \end{align*}
    are compact.
\end{prop}

\begin{proof} First, we need the following Lemma:

\begin{lem}\label{comp_cone}
    Let $C$ be a closed proper convex cone in some finite-dimensional $\R$-vector space $V$. Then for every $c\in V$, the set $K:=C\cap(c-C)$ is compact.
\end{lem}

\begin{proof}
    By contradiction, we assume $K$ is not compact. We fix some norm $\|\cdot\|$ on $V$. Since $K$ is closed, by Heine-Borel it must be unbounded, i.e., there exists a sequence $(x_n)$ such that $\|x_n\|\to\infty$. Since $y_n:=\frac{x_n}{\|x_n\|}\in S^1$ and for finite-dimensional $V$, $S_1$ is compact, there exists a limit point $y$ of $(y_n)$. Since $C$ is a closed cone, $y\R_+\subseteq C\cap(c-C)$ and, therefore $c-y\R_+\subseteq C\cap(c-C)$. Analogously, $c+y\R_+\subseteq C\cap(c-C)$ and, therefore $-y\R_+\subseteq C\cap(c-C)$. That means, $y\R\in C\cap(-C)$, so $y=0$. This contradicts to $y\in S^1$. Therefore, $K$ is compact.
\end{proof}

    \noindent By Lemma~\ref{comp_cone}, the set
    \begin{align*}
        K:=\{x\in B^{\sigma}_{\geq 0}\mid 1-x\in B^{\sigma}_{\geq 0}\}=B^{\sigma}_{\geq 0}\cap(1-B^{\sigma}_{\geq 0})
    \end{align*}
    is compact. Since $\theta^{-1}(K)=D(\bar G,\sigma)$ and $\theta$ is proper, $D(\bar G,\sigma)$ is compact.

    Further, notice that $D(B_\mathbb C^\sigma,\bar\sigma)$ is closed. Let now $c=c_1+c_2i\in D(B_\mathbb C^\sigma,\bar\sigma)$ where $c_1,c_2\in B^\sigma$. Then 
    $$1-\bar cc=(1-c_1^2-c_2^2)-[c_1,c_2]i\in B^\sigma_{\geq 0} +B^{-\sigma}i.$$
    In particular, $1-c_1^2-c_2^2\in B^\sigma_{\geq 0}$. This means, $c_1,c_2$ are elements of the compact domain $D(\bar G,\sigma)$, i.e., $D(B_\mathbb C^\sigma,\bar\sigma)$ is compact.
\end{proof}

\begin{cor}
Let $G$ be the minimal extension of $G_0$. The map
$$\begin{matrix}
\overline{\pol'}\colon & \OO(G_0,\sigma)\times B^{\sigma} & \to & \bar G \\
& (u,b) & \mapsto & ub
\end{matrix}$$
is surjective and continuous. In particular, $\bar G$ is connected.
\end{cor}

\section{Symplectic group over \texorpdfstring{$G$}{G}}\label{Sp_2(G)}

\subsection{Symplectic group \texorpdfstring{$\Sp_2(A,\sigma)$}{Sp2(A,sigma)} over an involutive algebra}

Let $(A,\sigma)$ be an~involutive algebra. The symplectic group over $(A,\sigma)$ was introduced and studied in~\cite{ABRRW}. In this section, we recall the definition of this group and state its properties that will be needed in the sequel of this chapter.

\begin{df}\label{osp}
A \defin{$\sigma$-sesquilinear form} $\omega$ on a right $A$-module $V$ is a map
$$\omega\colon V\times V\to A$$
such that for all $x,y,z\in V$ and for all $r_1,r_2\in A$
$$\omega(x+y,z)=\omega(x,z)+\omega(y,z),$$
$$\omega(x,y+z)=\omega(x,y)+\omega(x,z),$$
$$\omega(x_1r_1,x_2r_2)=\sigma(r_1)\omega(x_1,x_2)r_2.$$

We denote by $$\Aut(\omega):=\{f\in\Aut(V)\mid \forall x,y\in V:\omega(f(x),f(y))=\omega(x,y)\}$$ the group of symmetries of $\omega$. We also define the corresponding Lie algebra:
$$\End(\omega):=\{f\in\End(V)\mid \forall x,y\in V:\omega(f(x),y)+\omega(x,f(y))=0\}$$
with the usual Lie bracket $[f,g]=fg-gf$.
\end{df}

We now set $V=A^2$. We view $V$ as the set of columns and endow it with the structure of a right $A$-module. The anti-involution $\sigma$ extends to an involutive map on $A^2$ and on the space $\Mat_2(A)$ of $2\times2$-matrices with coefficients in $A$ componentwisely.

\begin{df}
We make the following definitions:
\begin{enumerate}
\item A pair $(x,y)$ for $x,y\in A^2$ is called a \defin{basis} of $A^2$ if for every $z\in A^2$ there exist $a,b\in A$ such that $z=xa+yb$.
\item The element $x\in A^2$ is called \defin{regular} if there exists $y\in A^2$ such that $(x,y)$ is a basis of $A^2$.
\item $l\subseteq A^2$ is called a \defin{line} if $l=xA$ for a regular $x\in A^2$. We denote the space of lines of $A^2$ by $\PP(A^2)$.
\item Two regular elements $x,y\in A^2$ are called \defin{linearly independent}
if $(x,y)$ is a basis of $A^2$.
\item Two lines $l,m$ are called \defin{transverse} if $l=xA$, $m=yA$ for linearly independent $x,y\in A^2$.
\item An element $x\in A^2$ is called \defin{isotropic} with respect to $\omega$ if $\omega(x,x)=0$. The set of all isotropic regular elements of $(A^2,\omega)$ is denoted by $\Is(\omega)$.
\item A line $l$ is called isotropic if $l=xA$ for a regular isotropic $x\in A^2$. The set of all isotropic lines of $(A^2,\omega)$ is denoted by $\PIs(\omega)$.
\end{enumerate}
\end{df}

From now on, we assume $\omega(x,y):=\sigma(x)^t\Omega y$ with $\Omega=\begin{pmatrix}0 & 1 \\ -1 & 0\end{pmatrix}$ for $x,y\in A^2$. The form $\omega$ is called the \emph{standard symplectic form} on $A^2$.

\begin{df}
Let $(A,\sigma)$ be an involutive algebra. If $\omega$ is the standard symplectic form, then the group $\mathrm{Sp}_2(A,\sigma):=\mathrm{Aut}(\omega)$ is called the \emph{symplectic group} $\mathrm{Sp}_2$ over $(A,\sigma)$.
\end{df}

We have
\begin{align}\label{eq:groups}
    \begin{aligned}
    \mathrm{Sp}_2(A,\sigma)&=\left\{\begin{pmatrix}
    a & b \\
    c & d
    \end{pmatrix}\midwd \sigma(a)c,\,\sigma(b)d\in A^\sigma,\,\sigma(a)d-\sigma(c)b=1\right\}\subseteq \mathrm{GL}_2(A).
    \end{aligned}
\end{align}
We can determine the Lie algebra $\mathfrak{sp}_2(A,\sigma)$ of $\mathrm{Sp}_2(A,\sigma)$:
\begin{align}\label{eq:Lie_algebras}
    \begin{aligned}
	    \mathfrak{sp}_2(A,\sigma)& =\left\{g\in \mathrm{Mat}_2(A)\midwd \sigma(g)^t\begin{pmatrix}
        0 & 1 \\
        -1 & 0
      \end{pmatrix}+ \begin{pmatrix}
        0 & 1 \\
        -1 & 0
      \end{pmatrix} g=0\right\}\\
    & =\left\{\begin{pmatrix}
		x & z \\
		y & -\sigma(x)
	\end{pmatrix}\midwd x\in A,\;y,z\in A^\sigma\right\}.
    \end{aligned}
\end{align}
A basis $(x,y)$ of $A^2$ is called symplectic if $\omega(x,x)=\omega(x,y)=0$ and $\omega(x,y)=1$.
    
\subsection{Lie algebra \texorpdfstring{$\mathfrak{sp}_2(B,\sigma)$}{sp2(B,sigma)}}\label{sec:sp2(B,sigma)}

Let $(A,\sigma)$ be an involutive real algebra. Let $B\subseteq A$ be a Lie subalgebra. 

Consider the vector space
$$\spp_2(B,\sigma)=\left\{\begin{pmatrix}
    x & z \\
    y & -\sigma(x)
    \end{pmatrix}\midwd x\in B, y,z\in B^{\sigma}\right\}\subseteq \spp_2(A,\sigma)$$
In general, it is not a Lie algebra. We need to take the following additional assumption:

\begin{prop}\label{spp_sublie}
The vector space $\spp_2(B,\sigma)$ is a Lie subalgebra of $\spp_2(A,\sigma)$ if and only if $B$ is of Jordan type.
\end{prop}

\begin{proof}
Matrices
$$r(z):=\begin{pmatrix}
    0 & z \\
    0 & 0
    \end{pmatrix},\,
l(x):=\begin{pmatrix}
    0 & 0 \\
    y & 0
    \end{pmatrix}\text{ and }
d(z):=\begin{pmatrix}
    x & 0 \\
    0 & -\sigma(x)
    \end{pmatrix}$$
generate $\spp_2(B,\sigma)$ as a vector space. The vector space $\spp_2(B,\sigma)$ is a Lie subalgebra of $\spp_2(A,\sigma)$ if and only if all Lie bracket of these elements are in $\spp_2(B,\sigma)$.

For $y,z\in B^{\sigma}$
$$[r(z),l(y)]=
\begin{pmatrix}
    zy & 0 \\
    0 & -yz
    \end{pmatrix}\in\spp_2(B,\sigma)$$
so $zy,yz\in B$ and we need the condition that $B$ is of Jordan type.

For $a\in B$, $z\in B^{\sigma}$:
$$[d(x),r(z)]=
\begin{pmatrix}
    0 & xz+z\sigma(x) \\
    0 & 0
    \end{pmatrix}\in\spp_2(B,\sigma)$$
so $xz+z\sigma(x)\in B^{\sigma}$. This holds for $B$ of Jordan type by Proposition~\ref{prop:JordanType}.

For $a\in B$, $y\in B^{\sigma}$:
$$[l(y),d(x)]=
\begin{pmatrix}
    0 & 0 \\
    \sigma(x)y+yx & 0
    \end{pmatrix}\in\spp_2(B,\sigma)$$
so $\sigma(x)y+yx\in B^{\sigma}$. This holds for $B$ of Jordan type by Proposition~\ref{prop:JordanType}.
\end{proof}

\subsubsection{Symplectic Lie algebra over Jordan algebras}

The symplectic Lie algebra $\mathfrak{sp}_2$ can be defined in a more general context in terms of Jordan algebras and their derivations without the assumption for them to be embedded into an associative algebra.

Let $(J,\circ)$ be a Jordan algebra. Let $\Aut(J)$ be the group of automorphisms of $J$. The Lie algebra of  $\Aut(J)$ is the Lie algebra $\Der(J)$ of all derivations of $J$ (with respect to the commutator bracket).

The space $B:=J\oplus\Der(J)$ admits a structure of a Lie algebra: for $x,y\in J$, $[x,y]\in \Der(J)$ is the derivation acting on $z\in J$ as follows 
$$[x,y]z:=x\circ (y\circ z)-y\circ (x\circ z)$$
and for  $D\in \Der(J)$ and $x\in J$, $[D,x]=Dx\in J$. This provides a Lie algebra structure on $B$ (for more details see~\cite[Section~II.4]{Faraut}). We define the anti-involution $\sigma\colon B\to B$ by the rule $\sigma(x)=x$ for $x\in J$ and $\sigma(x)=-x$ for $x\in \Der(J)$.

The following space of matrices:
$$\spp_2(B,\sigma):=\left\{\begin{pmatrix}
x & y \\
z & -\sigma(x)
\end{pmatrix}\midwd x\in B,\;y,z\in J\right\}.$$
is a Lie algebra. If $B$ is a Lie subalgebra of an associative involutive algebra $(A,\sigma)$, then this definition agrees with the definition from Section~\ref{sec:sp2(B,sigma)}.

In this way, the exceptional Hermitian Lie algebra of tube type $\mathfrak e_7^{-25}$ can be realized as follows: let $J=\Herm_3(\mathbb O)$ be the exceptional Jordan algebra of all Hermitian $3\times 3$-octonionic matrices. Its automorphism group agrees with the compact exceptional Lie group $F_4$. Therefore, $\Der(J)=\mathfrak f_4=\Lie(F_4)$ and $B=\mathfrak f_4\oplus \Herm_3(\mathbb O)\cong \mathfrak e_6^{-26}\oplus\R$. However, the exceptional Jordan algebra $J$ cannot be embedded into an associative algebra (cf.~\cite[Corollary 2.8.5]{Hanche84}). Thus we cannot realize it in the way provided in Section~\ref{sec:sp2(B,sigma)}. This fact will also prevent us from describing the corresponding Lie group in terms of $2\times 2$-matrices with coefficients in an appropriate submonoid of $A$.

\subsubsection{Indefinite orthogonal Lie group over involutive algebras}\label{sec:indef.orth}

Similarly as the symplectic group, the indefinite orthogonal group $\OO_{(1,1)}(A,\sigma)$ over an involutive algebra $(A,\sigma)$ can be defined.

Let $V=A^2$ be the right $A$-module as above, let $h(x,y):=\sigma(x)^tH y$, where $x,y\in V$ and $H=\begin{pmatrix}
 0 & 1 \\
 1 & 0
\end{pmatrix}$, be the standard indefinite orthogonal. Then the group $\OO_{(1,1)}(A,\sigma)$ is defined as the group $\Aut(h)\subset\GL_2(A)$. Alternatively, one can take $H=\begin{pmatrix}
 -1 & 0 \\
 0 & 1
\end{pmatrix}$, the corresponding group $\Aut(h)$ is then conjugated to $\OO_{(1,1)}(A,\sigma)$.

The group $\OO_{(1,1)}(A,\sigma)$ and its Lie algebra $\mathfrak{o}_{(1,1)}(A,\sigma)$ can be described explicitly: 
\begin{align}
\begin{aligned}
    \mathrm{O}_{(1,1)}(A,\sigma)&=\left\{\begin{pmatrix}
    a & b \\
    c & d
    \end{pmatrix}\midwd \sigma(a)c,\,\sigma(b)d\in A^{-\sigma},\,\sigma(a)d+\sigma(c)b=1\right\}\subseteq \mathrm{GL}_2(A),\\
    \mathfrak o_{(1,1)}(A,\sigma)& =\left\{g\in \mathrm{Mat}_2(A)\midwd \sigma(g)^t\begin{pmatrix}
        0 & 1 \\
        1 & 0
      \end{pmatrix}+ \begin{pmatrix}
        0 & 1 \\
        1 & 0
      \end{pmatrix} g=0\right\}\\
      & =\left\{\begin{pmatrix}
		x & z \\
		y & -\sigma(x)
	\end{pmatrix}\midwd x\in A,\;y,z\in A^{-\sigma}\right\}.
\end{aligned}
\end{align}

For more details about the group $\mathrm{O}_{(1,1)}(A,\sigma)$ we refer to~\cite[Part 1]{HKRW}. However, for a Lie subalgebra $B\subset A$ of Jordan type, the vector space 
$$\mathfrak o_{(1,1)}(B,\sigma)=\left\{\begin{pmatrix}
		x & z \\
		y & -\sigma(x)
	\end{pmatrix}\midwd x\in B,\;y,z\in B^{-\sigma}\right\}$$
is (in general) not a Lie algebra, in contrast to the Lie algebra $\mathfrak {sp}_{2}(B,\sigma)$~(cf. Proposition~\ref{spp_sublie}). The obstruction to being a Lie algebra is the fact that for $x,y\in B^{-\sigma}$, their product  is, in general, not contained in $B$.

\subsection{First definition of \texorpdfstring{$\Sp_2(G,\sigma)$}{Sp2(G,sigma)}}

Let $(A,\sigma)$ be an involutive real algebra. Let $G\leq A^\times$ be a Lie subgroup of $A^\times$ which is closed under $\sigma$ with the Lie algebra $B:=\Lie(G)$ such that $(B^{\sigma})^\times\subseteq G$. The connected component of the identity element in $G$ is denoted by $G_0$. Then $B=\Lie(G_0)=\Lie(G)\leq A$. If $B$ is weakly Hermitian, we always take $G$ to be the minimal extension from Section~\ref{discon-ext}.

We consider the following matrices:
\begin{equation}\label{eq:sp_2-generators}
D(x):=\begin{pmatrix}
    x & 0 \\
    0 & \sigma(x)^{-1}
    \end{pmatrix},\;
L(y):=\begin{pmatrix}
    1 & 0 \\
    y & 1
    \end{pmatrix},\;
R(z):=\begin{pmatrix}
    1 & z \\
    0 & 1
    \end{pmatrix},
\end{equation}
where $x\in G$, $y,z\in B^{\sigma}$. These matrices, acting on $A^2$, preserve the standard symplectic form
$\omega$. Therefore, the set $\{L(y)D(x)R(y)\mid y,z\in B^{\sigma},\;x\in G\}$ is contained in $\Sp_2(A,\sigma)$.

\begin{df} We denote by $\Sp_2(G,\sigma)$ the topological closure of the space 
$$\{L(y)D(x)R(z)\mid y,z\in B^{\sigma},\;x\in G\}$$
in $\Sp_2(A,\sigma)$. Matrices of the form $L(y)D(x)R(z)$ for $y,z\in B^{\sigma}$, $x\in G$ are called \defin{generic}.
\end{df}

\begin{prop}\label{prop:generic.elements}
    An element $M:=\begin{pmatrix}
        a & b \\
        c & d
    \end{pmatrix}\in\Sp_2(G,\sigma)$ is generic if and only if $a\in A^\times$. In~this case, $a\in G$, $ca^{-1},a^{-1}b\in B^\sigma$. Moreover, there is a neighborhood of the identity element in $\Sp_2(G,\sigma)$ which consists of generic elements.
\end{prop}

\begin{proof} First, notice that if $a\in A^\times$, then $M=L(ca^{-1})D(a)R(a^{-1}b)$ with $ca^{-1}, a^{-1}b\in A^\sigma$. Since $M\in\Sp_2(G,\sigma)$, there exists a sequence $M_i=L(y_i)D(x_i)R(z_i)$ such that $x_i\in G$, $y_i,z_i\in B^\sigma$, and $\lim(M_i)=M$. Since the decompositions of $M_i$ and $M$ into $L$-, $D$-, and $R$-matrices is unique, $\lim(x_i)=a$, i.e., $a\in \bar G$ and $a\in A^\times$, i.e., $a\in G$. Similarly, $\lim (y_i)=ca^{-1}\in B^\sigma$ and $\lim (z_i)=a^{-1}b\in B^\sigma$. Moreover, this implies that there is a neighborhood of $M$ containing only generic matrices. In particular, it holds for $M=\Id$.
\end{proof}

\begin{teo}\label{thm:Sp_2-LieGroup}
Let $B$ be of Jordan type such that $1\in B$. Then the space $\Sp_2(G,\sigma)$ is a~Lie group.
\end{teo}

\begin{proof} First, we prove the following two Lemmata:

\begin{lem}\label{invert_1+y}
Let $1+y$ is not invertible for some $y\in A$. Then there exists a neighborhood $U$ of $0\in\R$ such that for every $t\in U\bs\{0\}$, $1+y(1+t)$ is invertible.
\end{lem}

\begin{proof}
By Proposition~\cite[Proposition~2.24]{ABRRW}, $A$ can be embedded as a subalgebra into $\Mat(r,\R)$ for some $r\in \N$. We identify $B$ with a Lie subalgebra of $\Mat(r,\R)$. Since $1+y$ is not invertible, $-1$ is an eigenvalue of $y$. Since $y$ has only finitely many eigenvalues, there exists a neighborhood $U$ of $0\in\R$ such that for every $t\in U\bs\{0\}$, $1+y(1+t)$ is invertible.
\end{proof}

We remind that the topological closure $\bar G$ of $G$ in $A$  is a monoid.

\begin{lem}\label{zy-invert}
Let $y,z\in B^{\sigma}$ then $1+zy\in \bar G$. In particular, if $1+zy$ is invertible, then $1+zy\in G$.
\end{lem}

\begin{proof}
First, assume $z\in (B^{\sigma})^\times\subset G$. Then $1+zy=(z+zyz)z^{-1}\in \bar G$ because $z^{-1}\in G$, $z,zyz\in B^{\sigma}\subseteq\bar G$.

If $z$ is not invertible, take a sequence of invertible $(z_i)$ such that $\lim z_i=z$. Then all $1+z_iy\in \bar G$ and $\lim(1+z_iy)=1+zy$. Since $\bar G$ is closed, $1+zy\in \bar G$.
\end{proof}

Let now $a:=\lim L(y_i)D(x_i)R(z_i)\in\Sp_2(G,\sigma)$, $b:=\lim L(y'_i)D(x'_i)R(z'_i)\in\Sp_2(G,\sigma)$ for some sequences $\{x_i\},\{x'_i\}\subset G$ and $\{y_i\},\{y'_i\},\{z_i\},\{z'_i\}\subset B^{\sigma}$. We want to show that there exist sequences $\{y''_i\},\{z''_i\}\subset B^{\sigma}$, $\{x''_i\}\subset G$ such that $ab=\lim L(y''_i)D(x''_i)R(z''_i)$.

Since limits for $a$ and $b$ exist, we can write:
$$ab=\lim L(y_i)D(x_i)R(z_i)L(y'_i)D(x'_i)R(z'_i).$$
Consider the term
$$R(z_i)L(y'_i)=\begin{pmatrix}
    1 & z_i \\
    0 & 1
    \end{pmatrix}
    \begin{pmatrix}
    1 & 0 \\
    y'_i & 1
    \end{pmatrix}=
    \begin{pmatrix}
    1+z_iy'_i & z_i \\
    y'_i & 1
    \end{pmatrix}.$$

If $1+z_iy'_i$ is invertible then by the Lemma~\ref{zy-invert}, $1+z_iy'_i\in G$, then we can write $$R(z_i)L(y'_i)=L(z_i(1+z_iy'_i)^{-1})D(1+z_iy'_i)R((1+z_iy'_i)^{-1}y'_i).$$

If $1+z_iy'_i$ is not invertible in $B$ then we take a sequence $\{t_i\}\subset \R$ such that $\lim t_i =0$ and $t_i\in U_i$ where $U_i$ is the neighborhood of $0\in \R$ form the Lemma~\ref{invert_1+y} for the element $1+z_iy'_i$. Then $b=\lim L(y'_i(1+t_i))D(x'_i)R(z'_i)$ and $1+z_iy'_i(1+t_i)$ is invertible in $B$.

To conclude the proof, note the following permutation rules:
$$D(x)L(y)=L(\sigma(x)^{-1}yx^{-1})D(x),$$
$$R(z)D(x)=D(x)R(x^{-1}z\sigma(x)^{-1})$$
that always make possible to reorder matrices and
$$L(y)L(y')=L(y+y'),\;R(x)R(x')=R(z+z'),\;D(x)D(x')=D(xx')$$
for all $x,x'\in G$, $y,y',z,z'\in B^{\sigma}$.

Now, we show that for $a\in\Sp_2(G,\sigma)$, $a^{-1}\in\Sp_2(G,\sigma)$. Indeed, $\Sp_2(G,\sigma)\subseteq\Sp_2(A,\sigma)$ and $\Sp_2(A,\sigma)$ is a group, therefore, $a^{-1}$ exists in $\Sp_2(A,\sigma)$. As before, we take a sequence $a_i:=L(y_i)D(x_i)R(z_i)$ where $x_i\in G$, $y_i,z_i\in B^{\sigma}$ such that $\lim a_i=a$. Then 
$$a_i^{-1}=R(-z_i)D(x_i^{-1})L(-y_i)\in\Sp_2(G,\sigma).$$ Therefore, $\lim a_i^{-1}=a^{-1}\in\Sp_2(G,\sigma)$ and $\Sp_2(G,\sigma)$ is a closed subgroup of the Lie group $\Sp_2(A,\sigma)$. Thus by Cartan's theorem, it is a Lie subgroup of $\Sp_2(A,\sigma)$.
\end{proof}

\begin{cor}
    For every element $M\in\Sp_2(G,\sigma)$, its matrix coefficients $M_{ij}\in \bar G$ for $i,j\in\{1,2\}$.
\end{cor}

\begin{prop}\label{prop:LieAlgOfSp2}
The Lie algebra of $\Sp_2(G,\sigma)$ agrees with $\spp_2(B,\sigma)$.
\end{prop}

\begin{proof}
There is a neighborhood of the identity element in $\Sp_2(G,\sigma)$ that consists only of generic matrices. Consider a smooth path $p(t):=L(y(t))D(x(t))R(z(t))$ such that $y,z\colon (-1,1)\to B^{\sigma}$ are smooth and $y(0)=z(0)=0$ and $x\colon (-1,1)\to G$ smooth and $x(0)=1$. Then
$$p'(0)=\begin{pmatrix}
    x'(0) & z'(0) \\
    y'(0) & -\sigma(x'(0))
    \end{pmatrix}\in \spp_2(B,\sigma).$$
Moreover, for every $m:=\begin{pmatrix}
    x & z \\
    y & -\sigma(x)
    \end{pmatrix}\in \spp_2(B,\sigma)$, $y,x\in B^{\sigma}$, $x\in B$
the path 
$$p(t):=L(yt)D(\exp(xt))R(zt)\in\Sp_2(G,\sigma)$$
and $p'(0)=m$. Therefore, $\Lie(\Sp_2(G,\sigma))=\spp_2(B,\sigma)$.
\end{proof}

If $B$ is weakly Hermitian, we can prove connectedness of $\Sp_2(G,\sigma)$.

\begin{prop}\label{prop:Sp2Connected}
Let $B$ be weakly Hermitian, then the Lie group $\Sp_2(G,\sigma)$ is connected.
\end{prop}

\begin{proof}
We show that for every generic element $g:=L(y)D(x)R(z)$ such that $y,z\in B^{\sigma}$ and $x\in G$, there exists a path $g_t\colon [0,1)\to \Sp_2(G,\sigma)$ such that $g_0=g$, $\lim_{t\to 1} g_t=\Om$.

Using polar decomposition, we obtain $x=ub$ for some $u\in \OO(G_0,\sigma)$, $b\in (B^{\sigma})^\times$. We take $u_t\colon [0,1]\to \OO(G_0,\sigma)$ such that $u_0=u$, $u=1$ for $t\geq\frac{1}{2}$. It is possible because $\OO(G_0,\sigma)$ is connected. We take $b_t\colon [0,1)\to (B^{\sigma})^\times$ such that $\lim_{t\to 1}b_t=0$ and $y_t,z_t\colon [0,1)\to B^{\sigma}$ such that $z_0=z$ $y_0=y$ and $x_t=-y_t=b_t^{-1}$ for $t\geq \frac{1}{2}$. It is possible since $B^{\sigma}$ is connected. Define $g_t=L(y_t)D(x_t)R(z_t)$. Then:
\begin{align*}
\lim_{t\to 1}g_t & =\lim_{t\to 1}
\begin{pmatrix}
u_tb_t & u_tb_tz_t\\
y_tu_tb_t & y_tu_tb_tz_t+\sigma(u_tb_t)^{-1}
\end{pmatrix}\\
& =\lim_{t\to 1}
\begin{pmatrix}
b_t & b_tb_t^{-1}\\
b_t^{-1}b_t & b_t^{-1}b_tb_t^{-1}+b_t^{-1}
\end{pmatrix}=\Om.    
\end{align*}

\noindent Therefore,
$$\Sp_2(G,\sigma)=\overline{\{L(y)D(x)R(z)\mid y,z\in B^{\sigma},\;x\in G\}\cup\left\{\Om\right\}}$$
and $\{L(y)D(x)R(z)\mid y,z\in B^{\sigma},\;x\in G\}\cup \left\{\Om\right\}$ is connected. Therefore, the group $\Sp_2(G,\sigma)$ is connected.
\end{proof}

\subsection{Second definition of \texorpdfstring{$\Sp_2(G,\sigma)$}{Sp2(G,sigma)}}

We denote by $\Sp_2'(G,\sigma)$ the subgroup of $\Sp_2(A,\sigma)$ generated by matrices:
$$D(x):=\begin{pmatrix}
    x & 0 \\
    0 & \sigma(x)^{-1}
    \end{pmatrix},\;
I:=\begin{pmatrix}
    0 & 1 \\
    -1 & 0
    \end{pmatrix},\;
R(z):=\begin{pmatrix}
    1 & z \\
    0 & 1
    \end{pmatrix}$$
where $x\in G_0$, $z\in B^{\sigma}$. Since all generators of $\Sp'_2(G,\sigma)$ are elements of $\Sp_2(G,\sigma)$, we obtain the inclution $\Sp'_2(G,\sigma)\leq \Sp_2(G,\sigma)$. In this subsection, we show that actually $\Sp'_2(G,\sigma)=\Sp_2(G,\sigma)$.

\begin{prop}\begin{enumerate}
    \item Matrices $L(z):=\begin{pmatrix}
    1 & 0 \\
    z & 1
    \end{pmatrix}$ are in $\Sp'_2(G,\sigma)$ for all $z\in B^{\sigma}$. In~particular, all generic elements of $\Sp_2(G,\sigma)$ are in $\Sp'_2(G,\sigma)$.
    \item The group $\Sp'_2(G,\sigma)$ is a connected smooth manifold. The tangent space at the identity element of  $\Sp'_2(G,\sigma)$ agrees with $\spp_2(B,\sigma)$. In particular, if $B$ is weakly Hermitian then $\Sp'_2(G,\sigma)=\Sp_2(G,\sigma)$.
\end{enumerate}
\end{prop}

\begin{proof} (1) Indeed,
$$L(z):=\begin{pmatrix}
    1 & 0 \\
    z & 1
    \end{pmatrix}=
    -\begin{pmatrix}
    0 & 1 \\
    -1 & 0
    \end{pmatrix}
    \begin{pmatrix}
    1 & z \\
    0 & 1
    \end{pmatrix}
    \begin{pmatrix}
    0 & 1 \\
    -1 & 0
    \end{pmatrix}.$$
Therefore, using Proposition~\ref{prop:generic.elements}, we conclude that all generic elements are in  $\Sp_2(G,\sigma)$.

(2) From (1) follows that there is a neighborhood of the identity of $\Sp'_2(G,\sigma)$ which consists entirely of generic elements. This means that it is a neighborhood of the identity of $\Sp_2(G,\sigma)$ which generates $\Sp_2(G,\sigma)$. This implies that $\Sp'_2(G,\sigma)=\Sp_2(G,\sigma)$, since $\Sp_2(G,\sigma)$ is connected if $B$ is weakly Hermitian.
\end{proof}

\subsection{Center of \texorpdfstring{$\Sp_2(G,\sigma)$}{Sp2(G,sigma)} and the group \texorpdfstring{$\PSp_2(G,\sigma)$}{PSp2(G,sigma)}}

\begin{prop}
The center $Z(\Sp_2(G,\sigma))$ of $\Sp_2(G,\sigma)$ is isomorphic to $Z(G)\cap \OO(G,\sigma)$ where $Z(G)$ is the center of $G$. More precisely,
$$Z(\Sp_2(G,\sigma))=\{\diag(a,a)\mid a\in Z(G)\cap \OO(G,\sigma)\}.$$
\end{prop}

\begin{proof}
Let $M=
\begin{pmatrix}
a & b \\
c & d
\end{pmatrix}\in Z(\Sp_2(G,\sigma))$, then $M$ commutes with $\Om$. This gives: $d=a$, $c=-b$. Also $M$ commutes with $\begin{pmatrix} 1 & 1 \\ 0 & 1\end{pmatrix}\in\Sp_2(G,\sigma)$. This gives $b=0$. Since $M=\diag(a,a)\in\Sp_2(A,\sigma)$, $\sigma(a)^{-1}=a$. Moreover, $M$ commutes with all $\diag(g,\sigma(g)^{-1})$, i.e., $a\in Z(G)$. Therefore,
$$Z(\Sp_2(G,\sigma))<\{\diag(a,a)\mid a\in Z(G)\cap \OO(G,\sigma)\}.$$

It is also easy to see that matrices $\diag(a,a)$ for $a\in Z(G)\cap \OO(G,\sigma)$ commute with all elements of $\Sp_2(G,\sigma)$. Therefore,
\begin{equation*}
    Z(\Sp_2(G,\sigma))=\{\diag(a,a)\mid a\in Z(G)\cap \OO(G,\sigma)\}.\qedhere    
\end{equation*}
\end{proof}

\begin{cor}
For $B$ Hermitian, $Z(\Sp_2(G,\sigma))$ is compact.
\end{cor}

\begin{df}
The quotient group
$$\PSp_2(G,\sigma):=\Sp_2(G,\sigma)/Z(\Sp_2(G,\sigma))$$
is called \defin{projective symplectic group}.
\end{df}

\subsection{Subgroups of \texorpdfstring{$\Sp_2(G,\sigma)$}{Sp2(G,sigma)}}

\subsubsection{Maximal compact subgroup of \texorpdfstring{$\Sp_2(G,\sigma)$}{Sp2(G,sigma)}}\label{Max_Comp_R}

In this section, we assume $B$ to be Hermitian. We describe a maximal compact subgroup of $\Sp_2(G,\sigma)$.

We denote:
\begin{align*}
\OO_2(A,\sigma):=&\{M\in\Mat_2(A)\mid \sigma(M)^tM=\Id\}\\
=&\left\{\begin{pmatrix}
a & b \\
-b & a
\end{pmatrix}\midwd a,b\in A,\;\begin{matrix}
\sigma(a)a+\sigma(b)b=1\\
\sigma(a)b-\sigma(b)a=0
\end{matrix}\right\},\\
\KSp_2(G,\sigma):=&\Sp_2(G,\sigma)\cap\OO_2(A,\sigma)\\
\subseteq &\left\{\begin{pmatrix}
a & b \\
-b & a
\end{pmatrix}\midwd a,b\in\bar G,\;\begin{matrix}
\sigma(a)a+\sigma(b)b=1\\
\sigma(a)b-\sigma(b)a=0
\end{matrix}\;\right\}.    
\end{align*}

The group $\KSp_2(G,\sigma)$ can be seen as the subgroup of $\Sp_2(G,\sigma)$ preserving the following $\sigma$-sesquilinear form $h$ on $A^2$: for $x,y\in A^2$, $h(x,y):=\sigma(x)^ty$.

\begin{teo}\label{maxcomp-Sp_R}
The group $\KSp_2(G,\sigma)$ is a maximal compact subgroup of $\Sp_2(G,\sigma)$.
\end{teo}

\begin{proof} By definition, $\KSp_2(G,\sigma)$ is a closed subgroup of $\Sp_2(G,\sigma)$. We take a matrix $M:=\begin{pmatrix}
a & b \\
-b & a
\end{pmatrix}\in \KSp_2(G,\sigma)$.
Then
$$\sigma(M)^tM=\begin{pmatrix}
\sigma(a) & -\sigma(b) \\
\sigma(b) & \sigma(a)
\end{pmatrix}
\begin{pmatrix}
a & b \\
-b & a
\end{pmatrix}$$
$$=
\begin{pmatrix}
\sigma(a)a+\sigma(b)b & * \\
* & \sigma(b)b+\sigma(a)a
\end{pmatrix}=\begin{pmatrix}
1 & 0 \\
0 & 1
\end{pmatrix}$$
Since $\sigma(a)a+\sigma(b)b=1$, i.e.
$$a,b\in D=\{x\in \bar G\mid 1-\sigma(x)x\in B^{\sigma}_{\geq 0}\}\subseteq \bar G$$
which is compact by Proposition~\ref{comp_disc}, $\KSp_2(G,\sigma)$ can be identified with a closed subset of the compact $D^2$, so it is compact.

Now, we show that $\KSp_2(G,\sigma)$ is a maximal compact subgroup of $\Sp_2(G,\sigma)$. Let $K$ be some compact subgroup containing $\KSp_2(G,\sigma)$ as a proper subgroup. We consider the following decomposition of $\spp_2(G,\sigma)$:
$$\spp_2(G,\sigma)=\ksp_2(G,\sigma)\oplus\Sym_2(G,\sigma)$$
where
$$\ksp_2(G,\sigma)=\Lie(\KSp_2(G,\sigma))=\left\{
\begin{pmatrix}
a & b \\
-b & a
\end{pmatrix}
\midwd a\in B^{-\sigma}, b\in B^{\sigma}
\right\},$$
$$\Sym_2(G,\sigma)=\left\{
\begin{pmatrix}
c & d \\
d & -c
\end{pmatrix}
\midwd c,d\in B^{\sigma}
\right\}.$$
By our assumption, $\Lie(K)$ contains $\ksp_2(G,\sigma)$ and has nontrivial intersection with $\Sym_2(G,\sigma)$. Take some $\begin{pmatrix}
c & d \\
d & -c
\end{pmatrix}\in\Lie(K)\cap\Sym_2(G,\sigma)$, $c,d\in B^{\sigma}$. The matrix
$\begin{pmatrix}
0 & d \\
-d & 0
\end{pmatrix}\in\ksp_2(G,\sigma)\subset\Lie(K),$
therefore,
$$\begin{pmatrix}
c & 2d \\
0 & -c
\end{pmatrix}=\begin{pmatrix}
c & d \\
d & -c
\end{pmatrix}+\begin{pmatrix}
0 & d \\
-d & 0
\end{pmatrix}\in\Lie(K)\bs\ksp_2(G,\sigma).$$
Using the exponential map of $\spp_2(G,\sigma)$ restricted to $\Lie(K)$, we obtain that there exists a matrix $M:=\begin{pmatrix}
g & gx \\
0 & g^{-1}
\end{pmatrix}\in K\bs\KSp_2(G,\sigma)$ where $g=\exp(c)\in G^{\sigma}$, $x\in B^{\sigma}$. Consider the spectral decomposition of $g=\sum_{i=1}^k\lambda_ic_i$ for some $\lambda_i>0$ and $(c_i)_{i=1}^k$ a complete orthogonal system of idempotents. Take a sequence $\{M^r\}\subseteq K$. Then $$M^r_{11}=g^k=\sum_{i=1}^k\lambda_i^rc_i,$$
$$M^r_{22}=g^{-k}=\sum_{i=1}^k\lambda_i^{-r}c_i.$$
Assume there exists $s\in\{1,\dots,k\}$ such that $\lambda_s\neq\pm 1$. Then either $0<|\lambda_s|<1$ or $0<|\lambda_s^{-1}|<1$. Without loss of generality, assume $0<|\lambda_s|<1$. Since $K$ is compact, $\{M^r\}\subseteq K$ has a convergent subsequence $\{M^{r_j}\}\subseteq K$:
$$\lim M^{r_j}_{11}=\lim\sum_{i=1}^k\lambda_i^{r_j}c_i=\sum_{i=1}^k\hat\lambda_ic_i$$
where $\hat\lambda_i=\lim\lambda_i^{r_j}$. But $\hat\lambda_s=\lim\lambda_s^{r_j}=0$ for any subsequence $\{r_j\}$. Therefore $\lim M^{r_j}_{11}$ is not invertible and so $\lim M^{r_j}$ is not invertible as well. Therefore, all $\lambda_i=\pm 1$ and $g^2=1$. The element $L:=\begin{pmatrix}
g & 0\\
0 & g^{-1}
\end{pmatrix}\in\KSp_2(G,\sigma)\subset K$. Then $ML=\begin{pmatrix}
1 & x \\
0 & 1
\end{pmatrix}\in K$. Take $(ML)^r=\begin{pmatrix}
1 & rx \\
0 & 1
\end{pmatrix}\in K$. This sequence does not have any convergent subsequence unless $x=0$. So we get $M=L\in\KSp_2(G,\sigma)$. This contradicts to the assumption $M\notin\KSp_2(G,\sigma)$ and we obtain that $\KSp_2(G,\sigma)$ is a maximal compact subgroup of $\Sp_2(G,\sigma)$.
\end{proof}

\begin{cor}
    The linear map $\Psi\colon \ksp_2(A,\sigma)\to B_\mathbb C^{-\bar\sigma}$ with $\begin{pmatrix}
        a & b \\
        -b & a
    \end{pmatrix}\mapsto a+bi$ provides an isomorphism of Lie algebras. Therefore, in particular, the group $\OO(G_\mathbb C,\bar\sigma)$ and $\KSp_2(G,\sigma)$ are isomorphic, and the isomorphism is provided by the same map $\Psi$.
\end{cor}

\subsubsection{Group $G$ as a subgroup of $\Sp_2(G,\sigma)$}\label{sec:hat_G}

The group $G$ can be naturally seen as a subgroup of $\Sp_2(G,\sigma)$ embedded diagonally. We denote: 
$$\hat G:=\left\{\begin{pmatrix}
    g & 0 \\
    0 & \sigma(g)^{-1}
\end{pmatrix}\midwd g\in G\right\}.$$
The Lie algebra of $\hat G$ is:
$$\hat B=\left\{\begin{pmatrix}
    b & 0 \\
    0 & -\sigma(b)
\end{pmatrix}\midwd b\in B\right\}.$$
The group $\hat G$ preserves the following $\sigma$-sesquilinear form $h$ on $A$: for $x,y\in A^2$, $h(x,y)=\sigma(x)^t\begin{pmatrix}
    0 & 1 \\
    1 & 0
\end{pmatrix}y$. As we have seen in Section~\ref{sec:indef.orth}, $\Aut(h)=\OO_{(1,1)}(A,\sigma)$. Moreover, one can show that $\hat G=\Sp_{2}(G,\sigma)\cap\OO_{(1,1)}(A,\sigma)$.

\section{Invariants of \texorpdfstring{$G$}{G}-isotropic lines}\label{G-lines}

Let $B$ be a weakly Hermitian Lie subalgebra of $(A,\sigma)$. In this section, we introduce the space of $G$-isotropic lines which is a generalization of the real projective line. We also study the action of $\Sp_2(G,\sigma)$ on tuples of  $G$-isotropic lines.

\subsection{\texorpdfstring{$G$}{G}-isotropic elements and \texorpdfstring{$G$}{G}-isotropic lines}
The orbit $\Is_G(\omega)=\Sp_2(G,\sigma)(1,0)^t$  of $(1,0)^t\in A^2$ under the left action of $\Sp_2(G,\sigma)$ by matrix multiplication is called the \defin{space of $G$-isotropic elements}.

Since $\Sp_2(G,\sigma)\subset\Mat_2(\bar G)$, for every $x=(x_1,x_2)^t\in \Is_G(\omega)$, $x_1,x_2\in \bar G$. Moreover, $G$ acts on $\Is_G(\omega)$ by right multiplication. Indeed, let $x=M(1,0)^t\in \Is_G(\omega)$ for some $M\in\Sp_2(G,\sigma)$, and  let $\hat g:=\begin{pmatrix}
        g & 0 \\
        0 & \sigma(g)^{-1}
    \end{pmatrix}\in\Sp_2(G,\sigma)$. Then $xg=M \hat g (1,0)^t\in \Is_G(\omega)$.

\begin{prop}
    An element $(x_1,x_2)\in A^2$ is $G$-isotropic if and only if the element $(x_2,-x_1)$ is $G$-isotropic. If $(x_1,x_2)$ is isotropic, then $\sigma(x_1)x_2\in G^\sigma$.
\end{prop}

\begin{proof}
    Let $x=(x_1,x_2)\in A^2$ is $G$-isotropic, i.e., $x=g(1,0)^t$ for some $g\in \Sp_2(G,\sigma)$. Then $x'=(x_2,-x_1)=\begin{pmatrix}
        0 & 1 \\
        -1 & 0
    \end{pmatrix}g(1,0)^t\in\Is_G(\omega)$.

    If $x=(x_1,x_2)\in A^2$ is $G$-isotropic, then, in particular,
    $$0=\omega(x,x)=-\sigma(x_1)x_2+\sigma(x_2)x_2.$$
    This means $\sigma(x_1)x_2\in \bar G^\sigma$.
\end{proof}

The space of $G$-isotropic elements $\Is_G(\omega)$ is a homogenious space of $\Sp_2(G,\sigma)$, i.e., $\Is_G(\omega)$ is isomorphic to $\Sp_2(G,\sigma)/\Stab_{G}((1,0)^t)$, where 
$$\Stab_{G}((1,0)^t)=\left\{\begin{pmatrix}
1 & x \\
0 & 1
\end{pmatrix} \midwd x\in B^\sigma \right\}.$$ The~element $(0,1)^t$ is in $\Is_G(\omega)$ because $-I(1,0)^t=(0,1)$ where
$I=\begin{pmatrix}
0 & 1 \\
-1 & 0
\end{pmatrix}$. Therefore, similarly, $\Is_G(\omega)$ is isomorphic to $\Sp_2(G,\sigma)/\Stab_{G}((0,1)^t)$, where 
$$\Stab_{G}((0,1)^t)=\left\{\begin{pmatrix}
1 & 0 \\
x & 1
\end{pmatrix} \midwd x\in B^\sigma \right\}.$$
Both isomorphisms are given by the orbit maps.

\begin{df}\begin{enumerate}
    \item A subset $l\subset A^2$ is called a \defin{$G$-isotropic line} if $l=yA$ for some $y\in\Is_G(\omega)$. We denote the space of all $G$-isotropic lines by $\PIs_G(\omega)$.
    \item An element $y\in \bar G^2$ is \defin{regular} if for every $g\in\bar G\bs\{0\}$, $yg\neq 0$.
\end{enumerate}
\end{df}

\noindent Notice that all $G$-isotropic elements are regular.

For an element $x=(x_1,x_2)^t\in\bar G^2$, we denote 
$$\mathcal N(x)=(\sigma(x_1)x_1+\sigma(x_2)x_2)^{\frac{1}{2}}\in B^\sigma_{\geq 0}.$$
The function $\mathcal N$ is well-defined because $\sigma(x_1)x_1+\sigma(x_2)x_2\in B^\sigma_{\geq 0}$ for all $x_1,x_2\in \bar G$. We call $\mathcal N(x)$ the \defin{$B^\sigma$-valued norm} of $x$. 

\begin{prop}\label{RegEl-h}
Let $x\in \bar G^2$ be a regular element, then $\mathcal N(x)\in B^{\sigma}_+.$ In particular, this holds for all elements of $\Is_G(\omega)$.
\end{prop}

\begin{proof} First, we prove the following Lemma:
\begin{lem}\label{lem:invert}
Let $b\in B^{\sigma}$ be not invertible, then there exists $b'\in B^{\sigma}_{\geq 0}\bs\{0\}$ such that $bb'=0$
\end{lem}

\begin{proof}
Assume $b$ to be not invertible and consider its spectral decomposition $b=\sum_{i=1}^k\lambda_ic_i$ for some $(c_i)$ complete system of orthogonal idempotents. Since $b$ is not invertible, there exist $j\in\{1,\dots,k\}$ such that $\lambda_j=0$. Take $b'=c_j\in B^{\sigma}_{\geq 0}$.
\end{proof}

Since for every $g\in\bar G$, $\sigma(g)g\in B^{\sigma}_{\geq 0}$, $b:=\sigma(x_1)x_1+\sigma(x_2)x_2\in B^{\sigma}_{\geq 0}$. Assume $b$ is not invertible for some regular $x\in \bar G^2$. Take $b'$ as in Lemma~\ref{lem:invert}, then
$$0=b'bb'=\sigma(x_1b')x_1b'+\sigma(x_2b')x_2b'$$
and $\sigma(x_1b')x_1b', \sigma(x_2b')x_2b'\in B^{\sigma}_{\geq 0}$. Since $B^{\sigma}_{\geq 0}$ is a proper convex cone, 
$$\sigma(x_1b)x_1b = \sigma(x_2b)x_2b=0.$$
Since $x_1b,x_2b\in\bar G$, take its polar decomposition: $x_1b=u_1y_1$, $x_2b=u_2y_2$ where $y_1,y_2\in B^{\sigma}$, $u_1,u_2\in \OO(G,\sigma)$. Then $\sigma(x_1b)x_1b=y_1^2$, $\sigma(x_2b)x_2b=y_2^2$. Since $B^{\sigma}$ does not contain nilpotents, $y_1=y_2=0$. Therefore, $x_1b=x_2b=0$, i.e., $x=(x_1,x_2)^t$ is not regular. This contradicts to our assumption that $x$ is regular.
\end{proof}

\subsection{Action of \texorpdfstring{$\Sp_2(G,\sigma)$}{Sp2(G,sigma)} on \texorpdfstring{$G$}{G}-isotropic lines}

In this section, we study properties of the action of the group $\Sp_2(G,\sigma)$ on $\PIs_G(\omega)$.

\begin{prop}\label{stab1} The group $\Sp_2(G,\sigma)$ acts transitively on $\PIs_G(\omega)$. Moreover,
\begin{align*}
    \Stab_{\Sp_2(G,\sigma)}(1,0)^tA&:=\left\{
\begin{pmatrix}
x & xy \\
0 & \sigma(x)^{-1}
\end{pmatrix}
\midwd
x\in G, y\in B^{\sigma}
\right\}\text{, and}\\
\Stab_{\Sp_2(G,\sigma)}(0,1)^tA&:=\left\{
\begin{pmatrix}
x & 0 \\
zx & \sigma(x)^{-1}
\end{pmatrix}
\midwd
x\in G, z\in B^{\sigma}
\right\}.
\end{align*}

\end{prop}

\begin{proof} The group $\Sp_2(G,\sigma)$ acts transitively on the space of $G$-isotropic lines since it acts transitively on $\Is_G(\omega)$.

We prove only the statement for the first stabilizer. The second one can be proved analogously. Since
$$\begin{pmatrix}
x & y \\
z & t
\end{pmatrix}
\begin{pmatrix}
1 \\ 0
\end{pmatrix}=\begin{pmatrix}
x \\ z
\end{pmatrix},$$
where $x\in G$ and $z=0$. Therefore, the matrix is generic and has the form $D(x)R(y)$.
\end{proof}

\begin{prop} The group $\KSp_2(A,\sigma)$ acts transitively on
    \begin{itemize}
    \item the space of elements of $\Is_G(\omega)$ of norm $1$. The stabilizer of every element under this action is trivial. In particular, $\Is_G(\omega)$ is homeomorphic to $\KSp_2(A,\sigma)$.
    \item the space $\PIs_G(\omega)$. The stabilizer of the line $(1,0)^tA$ agrees with the stabilizer of the line $(1,0)^tA$ and agrees with the group 
    $$\hat U:= \left\{\begin{pmatrix} u & 0 \\ 0 & u\end{pmatrix}\midwd u\in \OO(G,\sigma)\right\}\subset \KSp_2(A,\sigma).$$
    The space of $G$-isotropic lines can be identified with the homogeneous space $\KSp_2(A,\sigma)/\hat U$.
\end{itemize}
\end{prop}

\begin{proof}
    Let $x=(x_1,x_2)^t\in\Is_G(\omega)$ such that $\mathcal N(x)=1$. Then the matrix 
    $$M:=\begin{pmatrix}
        x_1 & -x_2\\
        x_2 & x_1
    \end{pmatrix}\in \KSp_2(A,\sigma)$$
    and $M(1,0)^t=x$. Since the action is transitive, it is enough to check that the stabilizer of $(1,0)^t$ is trivial. Since elements of $\KSp_2(A,\sigma)$ have form $M:=\begin{pmatrix}
        a & b \\
        -b & a
    \end{pmatrix}$, we obtain that $M(1,0)^t=(1,0)^t$ implies $a=1$, $b=0$, i.e., $M=\Id$.

    For a $G$-isotropic line $l=yA$, $y\in\Is_G(\omega)$, we take $x:=y\mathcal N(y)^{-1}\in \Is_G(\omega)$ and $\mathcal N(x)=1$, then $l=xA$. Since the action of $\KSp_2(A,\sigma)$ is transitive on the space of elements of $\Is_G(\omega)$ of norm $1$, it is also transitive on the space of $G$-isotropic lines.

    The stabilizer of the line $(1,0)^tA$ agrees with the stabilizer of the line $(0,1)^tA$ and agrees with the subgroup $\hat U\subset \KSp_2(A,\sigma)$. Thus, the space of $G$-isotropic lines can be identified with the space $\KSp_2(A,\sigma)/\hat U$
\end{proof}

\begin{cor}
    If $\OO(G,\sigma)$ is compact, then $\PIs_G(\omega)$ and the subspace of $\Is_G(\omega)$ of norm $1$ are compact manifolds.
\end{cor}

\subsection{Action of \texorpdfstring{$\Sp_2(G,\sigma)$}{Sp2(G,sigma)} on pairs of \texorpdfstring{$G$}{G}-isotropic lines}

\begin{prop}
Two elements $u,v\in\Is_G(\omega)$ are linearly independent if and only if, up to action of $\Sp_2(G,\sigma)$, $u=(1,0)^t$, $v=(a,b)^t$ with $b\in G$. Moreover, if $\omega(u,v)=1$, then $a\in B^{\sigma}$, $b=1$.
\end{prop}

\begin{proof}
The group $\Sp_2(G,\sigma)$ acts transitively on $\Is_G(\omega)$, therefore, up to $\Sp_2(G,\sigma)$-action, we can assume $u=(1,0)^t$.

Since $u$ and $v$ are linearly independent, $b\in A^\times$ and $v=g(1,0)^t$ for some $g\in\Sp_2(G,\sigma)$. If $g=L(y)D(x)R(z)$ for some $x\in G$, $y,z\in B^{\sigma}$, then $v=(x,yx)^t=(1,y)^tx$. Therefore, $y\in (B^{\sigma})^\times\subseteq G$ and so $b=yx\in G$.

If $g$ is not generic, take a sequence $\{g_n\}$ of generic elements such that $g_n\to g$. Then $G\ni y_nx_n\to b\in A^\times$. Since $G$ is closed in $A^\times$ and $x_n,y_n\in G$, $b\in G$.

Let now $\omega(u,v)=1$, then $1=\omega(u,v)=yx$. So if $g$ generic, then $a=x=y^{-1}\in B^{\sigma}$. If $g$ is not generic, then $a=\lim(x_n)=\lim(y_n^{-1})$. But all $y_n^{-1}\in B^{\sigma}$ and $B^{\sigma}$ is closed in $A$, so $a\in B^{\sigma}$.
\end{proof}

\begin{cor}
If $x,y\in\Is_G(\omega)$ linearly independent, then $\omega(x,y)\in G$.
\end{cor}

\begin{df}
A symplectic basis $(x,y)$ of $(A^2,\omega)$ is called \defin{$(G,\sigma)$-symplectic} if $x,y\in\Is_G(\omega)$.
\end{df}

\begin{prop}\label{trans_bas}
If $(x,y)$ is a $(G,\sigma)$-symplectic basis then there exists the unique $g\in \Sp_2(G,\sigma)$ such that $g(1,0)^t=x$, $g(0,1)^t=y$. In particular, $\Sp_2(G,\sigma)$ acts transitively on $(G,\sigma)$-symplectic bases.
\end{prop}

\begin{proof}
We can assume $x=(1,0)^t$, $y=(a,1)^t$ and $a\in B^{\sigma}$. Take $g:=R(-a)$, then $R(-a)x=x$, $R(-a)y=(0,1)^t$.
\end{proof}

\begin{cor}\label{cor:(c,1)-isotropic}
    An element $(c,1)$ is $G$-isotropic if and only if $c\in B^\sigma$.
\end{cor}

\begin{proof}
    An element $y=(c,1)$ is $G$-isotropic if and only if  $(x,y)$ is a $(G,\sigma)$-symplectic basis where $x=(1,0)$. This holds if and only if the matrix $\begin{pmatrix}
        1 & c \\
        0 & 1
    \end{pmatrix}\in\Sp_2(G,\sigma)$, i.e., $c\in B^\sigma$.
\end{proof}

\begin{cor}\label{trans2}
Let $xA$, $yA$ be two transverse isotropic lines with $x,y\in\Is_G(\omega)$. Then there exist $M\in \Sp_2(G,\sigma)$ and $y'\in\Is_G(\omega)$ such that $y'A=yA$ and $Mx=(1,0)^t$, $My'=(0,1)^t$. In particular, $\omega(x,y')=1$.
\end{cor}

\begin{prop}\label{stab2} The group $\Sp_2(G,\sigma)$ acts transitively on pairs of transverse $G$-isotropic lines.
$$\Stab_{\Sp_2(G,\sigma)}((1,0)^tA,(0,1)^tA):=\left\{
\begin{pmatrix}
x & 0 \\
0 & \sigma(x)^{-1}
\end{pmatrix}
\mid
x\in G\right\}\cong G.$$
\end{prop}

\begin{proof} By Corollary~\ref{trans2}, every pair of transverse $G$-isotropic lines can be mapped to $((1,0)^tA,(0,1)^tA)$ by an element of $\Sp_2(G,\sigma)$. So $\Sp_2(G,\sigma)$ acts transitively on pairs of transverse $G$-isotropic lines.

By Proposition~\ref{stab1}, for every $M\in \Stab_{\Sp_2(G,\sigma)}((1,0)^tA,(0,1)^tA)$,
$M=D(x)R(y)$ for $x\in G$, $y\in B^{\sigma}$.
Moreover, $D(x)R(y)(0,1)^t=(xy,\sigma(x)^{-1})$. Therefore, $y=0$.
\end{proof}

\subsection{Action of \texorpdfstring{$\Sp_2(G,\sigma)$}{Sp2(G,sigma)} on positive triples of \texorpdfstring{$G$}{G}-isotropic lines}

Let $(x_1A,x_3A,x_2A)$ be a triple of pairwise transverse $G$-isotropic lines where all $x_i\in \Is_G(\omega)$. Because of transversality of $x_1A$ and $x_2A$, we can assume $\omega(x_1,x_2)=1$. Up to action of $\Sp_2(G,\sigma)$, we can assume $x_1=(1,0)^t$, $x_2=(0,1)^t$. We can also normalize $x_3$ so that $\omega(x_3,x_2)=1$. Then $x_3=(1,b)^t$, $b=\omega(x_1,x_3)\in (B^{\sigma})^\times$.

\begin{df}
A triple of pairwise transverse $G$-isotropic lines $(x_1A,x_3A,x_2A)$ is called \defin{positive} if $\omega(x_1,x_2)=\omega(x_3,x_2)=1$ and $\omega(x_1,x_3)\in B^{\sigma}_+$.
\end{df}

\begin{prop}
The definition of positivity of a triple of $G$-isotropic lines does not depend on the choice of $x_1,x_2,x_3$.
\end{prop}

\begin{proof}
Let $y_i\in \Is_G(\omega)$ such that $y_iA=x_iA$ for all $i\in\{1,2,3\}$. Then $y_i=x_ig_i$ for some $g_i\in G$. Since $1=\omega(y_1,y_2)=\sigma(g_1)\omega(x_1,x_2)g_2=\sigma(g_1)g_2$, $g_2=\sigma(g_1)^{-1}$. Similarly $g_2=\sigma(g_3)^{-1}$. Therefore, $g_1=g_3$, and
$$\omega(y_1,y_3)=\sigma(g_1)\omega(x_1,x_3)g_1=\sigma(g_1)bg_1\in B^{\sigma}_+$$
if and only if $b\in B^{\sigma}_+$.
\end{proof}

\begin{rem}
For every transverse triple $(x_1A,x_3A,x_2A)$, up to action of $\Sp_2(G,\sigma)$, we can write $x_1=(1,0)^t,x_2=(0,1)^t$, and:
$$x_3A=\begin{pmatrix}
1 & 0 \\
b & 1
\end{pmatrix}
\begin{pmatrix}
1 \\
0
\end{pmatrix}A
=\begin{pmatrix}
1 \\
b
\end{pmatrix}A$$
for $b\in B^{\sigma}$. The triple is positive if and only if $b\in B^{\sigma}_+$. Matrices of the form $\begin{pmatrix}
1 & 0 \\
b & 1
\end{pmatrix}$ for $b\in B^{\sigma}_+$ form a subsemigroup of $\Sp_2(G,\sigma)$.
\end{rem}

\begin{lem}\label{trans3}
For every positive triple $(l_1,l_3,l_2)$ of isotropic lines, there exist elements $y_1,y_2,y_3\in\Is_G(\omega)$ such that
\begin{itemize}
\item $l_i=y_iA$ for $i\in\{1,2,3\}$;
\item $\omega(y_1,y_2)=1$;
\item $y_3=y_1+y_2$.
\end{itemize}
\end{lem}

\begin{proof}
Let $l_i=x_iA$ for some regular $x_i\in \Is_G(\omega)$. By transversality, $(x_1,x_2)$ form a basis. As above, we can assume $\omega(x_1,x_2)=1$, $x_3=x_1+x_2a$ for $a\in B^{\sigma}_+$. Take $c:=a^\frac{1}{2}\in B^{\sigma}_+\subseteq G$. Consider a new basis $y_1=x_1c^{-1}$, $y_2=x_2\sigma(c)$. Then,
$$\omega(y_1,y_2)=\omega(x_1c^{-1},x_2\sigma(c))=\sigma(c)^{-1}1\sigma(c)=1.$$
Moreover, $x_3=y_1c+y_2\sigma(c)^{-1}\sigma(c)c=y_1c+y_2c$. If we take $y_3:=x_3c^{-1}=y_1+y_2$, we get $y_3A=x_3A$.
\end{proof}

\begin{prop}\label{stab3} $\Sp_2(G,\sigma)$ acts transitively on the space of positive triples of pairwise transverse isotropic lines.

The stabilizer of the positive triple
$$\left(\begin{pmatrix}1 \\ 0\end{pmatrix}A,\begin{pmatrix}1 \\ 1\end{pmatrix}A,\begin{pmatrix}0 \\ 1\end{pmatrix}A\right)$$
in $\Sp_2(G,\sigma)$ coincides with the subgroup
$$\hat U=\left\{
\begin{pmatrix}
u & 0 \\
0 & u
\end{pmatrix}
\mid u\in \OO(G,\sigma)\right\}\cong \OO(G,\sigma).$$

The stabilizer of every positive triple of $G$-isotropic lines is conjugated to $\hat U$ in $\Sp_2(G,\sigma)$.
\end{prop}

\begin{proof} Let $(l_1,l_3,l_2)$ be a positive triple. By Lemma~\ref{trans3}, there exist $y_i\in l_i$, $i\in\{1,2,3\}$ such that $l_i=y_iA$, $\omega(y_1,y_2)=1$ and $y_3=y_1+y_2$. By Proposition~\ref{trans_bas}, there exists $M\in\Sp_2(G,\sigma)$ such that $My_1=(1,0)^t$, $My_2=(0,1)^t$. Therefore, $Ml_1=(1,0)^tA$, $Ml_2=(0,1)^tA$, $Ml_3=(1,1)^tA$ i.e., every positive triple can be mapped to the standard positive triple $((1,0)^tA,(1,1)^tA,(0,1)^tA)$.

By Proposition~\ref{stab2}, for every $M$ that stabilizes $\left(\begin{pmatrix}1 \\ 0\end{pmatrix}A,\begin{pmatrix}1 \\ 1\end{pmatrix}A,\begin{pmatrix}0 \\ 1\end{pmatrix}A\right)$, $M=D(x)$ for $x\in G$. Moreover, $M(1,1)^t=(x,\sigma(x)^{-1})$. Therefore, $x=\sigma(x)^{-1}$, i.e., $x\in \OO(G,\sigma)$.
\end{proof}

\begin{prop}
Positivity of triples is invariant under cyclic permutations, i.e., if $(l_1,l_3,l_2)$ is positive, then $(l_2,l_1,l_3)$ is positive as well.
\end{prop}

\begin{proof}
The triple $(l_1,l_3,l_2)$ is positive if and only if there exist $x_1,x_2,x_3\in\Is_G(\omega)$ such that $l_i=x_iA$, $i\in\{1,2,3\}$ and $\omega(x_1,x_2)=\omega(x_3,x_2)=1$, $\omega(x_1,x_3)=b\in B^{\sigma}_+$.

The triple $(l_2,l_1,l_3)$ is positive if and only if there exist $y_1,y_2,y_3\in\Is_G(\omega)$ such that $l_i=y_iA$, $i\in\{1,2,3\}$ and $\omega(y_2,y_3)=\omega(y_1,y_3)=1$, $\omega(y_2,y_1)\in B^{\sigma}_+$.

We take $y_1=x_1b^{-1}$, $y_2=-x_2$, $y_3=x_3$, then 
\begin{align*}
    \omega(y_2,y_3)& =\omega(-x_2,x_3)=1\\
    \omega(y_1,y_3)& =\omega(x_1b^{-1},x_3)=1\\
    \omega(y_2,y_1)& =\omega(-x_2,x_1b^{-1})=b^{-1}\in B^{\sigma}_+.   \qedhere
\end{align*}
\end{proof}

\subsection{Invariant of a positive quadruple of \texorpdfstring{$G$}{G}-isotropic lines}\label{pos_quadr}

\begin{df}
A quadruple $(l_1,l_3,l_2,l_4)$ of pairwise transverse $G$-isotropic lines is called \defin{positive} if the triples $(l_1,l_3,l_2)$, $(l_2,l_4,l_1)$ are positive.
\end{df}

The following proposition follows immediately from Proposition~\ref{stab3} and the spectral theorem:

\begin{prop}\label{diagon} Let $(l_1,l_3,l_2,l_4)$ be a positive quadruple of $G$-isotropic lines. Then there exist $y_1,\dots,y_4\in \Is_G(\omega)$ such that $l_i=y_iA$, $y_3=y_1+y_2$, $y_4=y_1-y_2a$, $a\in B^{\sigma}_+$.

For any such choice of $(y_1,y_2,y_3,y_4)$, there exists a Jordan frame $(e_i)_{i=1}^n$ of $B^{\sigma}$, where $n=\rk(B^{\sigma})$ and a unique tuple $(\lambda_1,\dots,\lambda_n)$ with $\lambda_1\geq\dots\geq\lambda_n>0$ such that $a=\sum_{i=1}^n\lambda_ie_i$.
\end{prop}

\begin{rem}
\begin{itemize}
\item The $n$-tuple $(\lambda_1,\dots,\lambda_n)$ with $\lambda_1\geq\dots\geq\lambda_n>0$ does not depend on the choice of $(y_1,y_2,y_3,y_4)$. 
We denote the $n$-tuple $$[l_1,l_3,l_2,l_4]:=(\lambda_1,\dots,\lambda_n)$$
and call it the \defin{cross--ratio} of the quadruple lines $(l_1,l_3,l_2,l_4)$. It is invariant under the action of $\Sp_2(G,\sigma)$ on the space of positive quadruples of $G$-isotropic lines. Moreover, if $B$ is Hermitian, two positive quadruples are related by an element of $\Sp_2(A,\sigma)$ if and only if their cross--ratios agree.
\item Although the tuple $(\lambda_1,\dots,\lambda_n)$ is completely determined by the quadruple $(l_1,l_3,l_2,l_4)$, the Jordan frame $(e_i)_{i=1}^n$ is in general not unique.
\end{itemize}
\end{rem}

\begin{prop}
Positivity of a quadruple is an invariant under cyclic permutations, i.e., if a quadruple $(l_1,l_3,l_2,l_4)$ is positive, then the quadruple $(l_3,l_2,l_4,l_1)$ is positive as well. Moreover, if $[l_1,l_3,l_2,l_4]=(\lambda_1,\dots,\lambda_n)$, then $$[l_3,l_2,l_4,l_1]=(\lambda_n^{-1},\dots,\lambda_1^{-1}).$$
\end{prop}

\begin{proof}
Let $(l_1,l_3,l_2,l_4)$ be a positive quadruple of isotropic lines. Then up to action of $\Sp_2(G,\sigma)$, there exist $y_1,\dots,y_4\in \Is_G(\omega)$ such that $l_i=y_iA$, $y_3=y_1+y_2$, $y_4=y_1-y_2a$, $a=\sum_{i=1}^n\lambda_ie_i$ with $\lambda_1\geq\dots\geq\lambda_n>0$.

The triple $(l_3,l_2,l_4)$ is positive. Indeed, since $a+1\in B^{\sigma}_+$, it is invertible. So we can take
\begin{align*}
x_3&:=y_3(a+1)^{-\frac{1}{2}}\in l_3,\\
x_4&:=-y_4(a+1)^{-\frac{1}{2}}\in l_4,\\
x_2&:=y_2(a+1)^{\frac{1}{2}}\in l_2.    
\end{align*}
Therefore,
\begin{align*}
\omega(x_3,x_4)&=\omega((y_1+y_2)(a+1)^{-\frac{1}{2}},-(y_1-y_2a)(a+1)^{-\frac{1}{2}})\\
&=-(a+1)^{-\frac{1}{2}}(-a-1)(a+1)^{-\frac{1}{2}}=1,\\
\omega(x_2,x_4)&=\omega(y_2(a+1)^{\frac{1}{2}},-(y_1-y_2a)(a+1)^{-\frac{1}{2}})=1,\\
\omega(x_3,x_2)&=\omega((y_1+y_2)(a+1)^{-\frac{1}{2}},y_2(a+1)^{\frac{1}{2}})=1\in B^{\sigma}_+.
\end{align*}
Analogously, the triple $(l_4,l_1,l_3)$ is positive as well. So we get that the quadruple $(l_3,l_2,l_4,l_1)$ is positive.

Take $x_1:=y_1(a+1)^{\frac{1}{2}}\in l_1$. Then an easy calculation shows that
$$x_2=x_3+x_4,\; x_1=x_3-x_4a^{-1}$$
where $a^{-1}=\sum_{i=1}^n\lambda_i^{-1}e_i$. Therefore, $[l_3,l_2,l_4,l_1]=(\lambda_n^{-1},\dots,\lambda_1^{-1})$.
\end{proof}

\subsection{Tangent space to the space of isotropic lines}

In this section, we describe the tangent spaces to $\Is_G(\omega)$ and $\PIs_G(\omega)$.

\begin{prop}
    Let $v\in\Is_G(\omega)$, let $v'\in\Is_G(\omega)$ such that $(v,v')$ is a $(G,\sigma)$-symplectic basis. Then
    $$T_v\Is_G(\omega)=\{w\in A^2\mid \omega(w,v)\in B^\sigma,\;\omega(w,v')\in B\}.$$
    Moreover, this description does not depend on $v'$.
\end{prop}

\begin{proof}
    Let $v_t\colon (-\varepsilon,\varepsilon)\to \Is_G(\omega)$ be a smooth path such that $\varepsilon>0$, $v_0=v$, let $w:=\dot v_0$. Then $v=g(1,0)^t$ and $v'=g(0,1)^t$ for some $g\in\Sp_2(G,\sigma)$, and, if $\varepsilon$ is small enough, $v_t=g L(y_t)D(x_t)(1,0)^t$ where $x_t\colon(-\varepsilon,\varepsilon)\to G$ and $y_t\colon(-\varepsilon,\varepsilon)\to B^\sigma$ are smooth functions such that $x_0=1$, $y_t=0$ and $D$ and $L$ are the matrices from~\eqref{eq:sp_2-generators}. For a fixed $g$, elements $x_t$ and $y_t$ are determined by $v_t$. Then
    $$w=g(L(\dot y_0)+D(\dot x_0))(1,0)^t=g(\dot x_0,\dot y_0)^t.$$
    This implies: $\omega(w,v)=-\sigma(\dot y_0)\in B^\sigma$ and $\omega(w,v')=\sigma(\dot x_0)\in B=T_1G$, i.e.,
    $$T_v\Is_G(\omega)\subseteq \{w\in A^2\mid \omega(w,v)\in B^\sigma,\;\omega(w,v')\in B\}.$$

    Let $w\in A^2$ such that $\omega(w,v)=c\in B^\sigma$ and $\omega(w,v')=b\in B$, we take $x_t:=\exp(tb)$ and $y_t:=ct$ and $v_t:=g L(y_t)D(x_t)(1,0)^t$. Then $\dot v_0=w$. Therefore,
    $$T_v\Is_G(\omega)\supseteq \{w\in A^2\mid \omega(w,v)\in B^\sigma,\;\omega(w,v')\in B\}.$$

    Let now $v''\in \Is_G(\omega)$ be another element such that $(v,v'')$ is a $(G,\sigma)$-symplectic basis.  Let $v''=vb+v'b'=g((1,0)^tb+(0,1)^tb')$ for $b,b'\in A$. Therefore, $\omega(v,v'')=b'=1$ and $\omega(v'',v'')=\sigma(b)-b=0$, i.e., $b\in A^\sigma$. Since  $(v,v'')$ is a $(G,\sigma)$-symplectic basis, there exists $g'\in\Sp_2(G,\sigma)$ such that $g(v,v'')=((1,0)^t,(0,1)^t)$. Therefore, $g(g')^{-1}((1,0)^t,(0,1)^t)=((1,0)^t,(b,1)^t)$, but, $g(g')^{-1}\in\Sp_2(G,\sigma)$ and it fixes $(1,0)^t$. This means $g(g')^{-1}=\begin{pmatrix}
        1 & b \\
        0 & 1
    \end{pmatrix}$ where $b\in B^\sigma$. Thus, $v''=vb+v'$ where $b\in B^\sigma$. Now, $\omega(w,v')\in B$ if and only if $\omega(w,v'')=\omega(w,v')+\omega(w,v)b\in B$ because $\omega(w,v)b\in B$ if $\omega(w,v),b\in B^\sigma$ by the property to be of Jordan type for $B$.
\end{proof}

\begin{prop}\label{prop:islines.tangent}
    Let $l\in\PIs_G(\omega)$. Then
    \begin{align}\label{eq:islines.tangent}
    T_l\PIs_G(\omega)=\left\{Q\in\Hom_A(l,A^2/l)\mid \omega(w,v)\in B^\sigma\text{ for $v\in\Is_G(\omega)$ such that $l=xA$}\right\}.
    \end{align}
\end{prop}

\begin{proof}
    We consider the differential of the quotient map $q\colon \Is_G(\omega)\to \PIs_G(\omega)$. Let $v\in \Is_G(\omega)$ then the kernel $\Ker(d_vq)=vA\cap T_v\Is_G(\omega)$. Thus for $w\in T_v\Is_G(\omega)$, we can identify the element $d_vq(w)$ with the $A$-linear map $Q\colon l\to A^2/l$ which sends $w$ to $w+vA$. Since $\omega(w,v)=\omega(w+vA,v)\in B^\sigma$, we obtain the description of the tangent space~\eqref{eq:islines.tangent}.
\end{proof}

\begin{cor}
    For every $v\in\Is_G(\omega)$, the tangent space $T_v\Is_G(\omega)$ is isomorphic to $B\oplus B^\sigma$, and $T_{vA}\PIs_G(\omega)$ is isomorphic to $B^\sigma$.
\end{cor}

\begin{rem}
    Notice that the space $\PIs_G(\omega)$ is a compact Riemannian symmetric space. The linear transformation $\begin{pmatrix}
        1 & 0 \\
        0 & -1
    \end{pmatrix}$ induces an inversion symmetry at the line $(1,0)^tA\in \PIs_G(\omega)$. In contrast, $\Is_G(\omega)$ is not a symmetric space. Indeed, for every $v\in\Is_G(\omega)$, $v\in T_v\Is_G(\omega)$. Therefore, every $Q\in\End_A(A^2)$ fixing $v$ cannot act as $-\Id$ on $T_v\Is_G(\omega)$.
\end{rem}

\section{Models for the Riemannian symmetric space of \texorpdfstring{$\Sp_2(G,\sigma)$}{Sp2(G,sigma)}}\label{G-models}

Let $(A,\sigma)$ be an involutive finite dimensional $\R$-algebra. The goal of this section is to construct different models of the symmetric space for $\Sp_2(G,\sigma)$ for $\Lie G=B$ where $B$ is a Hermitian Lie subalgebra $A$.

\subsection{Space of complex structures}\label{Comp_Str_Mod}

The first model we construct is the space of complex structures.

\begin{df}\label{df:complex.str}
Let $V$ be a right $A$-module, and $J\in\End_A(V)$. The map $J$ is called a \defin{complex structure} on $V$ if $J^2=-\Id$.
\end{df}

Let $V=A^2$ and $\omega$ be the standard symplectic form in $A^2$. For every complex structure $J$ on $A^2$, we can define the following $\sigma$-sesquilinear form
$$\begin{matrix}
h_J\colon & A^2\times A^2 & \to & A \\
& (x,y) & \mapsto & \omega(J(x),y).
\end{matrix}$$

\begin{df}
A $\sigma$-sesquilinear form $h$ on $(A^2,\omega)$ is called \defin{$(G,\sigma)$-symmetric} if $h$ is $\sigma$-symmetric and for all $v\in\bar G^2$, $h(v,v)\in B^{\sigma}$; it is called a \defin{$(G,\sigma)$-inner product} if additionally for all $v\in\bar G^2$, $h(v,v)\in B^{\sigma}_+$. For a $(G,\sigma)$-inner product $h$, a basis $(v,w)$ of $A^2$ is called \defin{$h$-orthonormal} if $h(v,v)=h(w,w)=1$ and $h(v,w)=0$.
\end{df}

We consider the following space:
$$\mathfrak C:=\left\{J\text{ complex structure on $A^2$}\midwd\begin{array}{l}
J(\Is_G(\omega))=\Is_G(\omega),\\
h_J\text{ is a $(G,\sigma)$-inner product}
\end{array}\right\}.$$

\begin{df}
The \defin{standard complex structure} on $A^2$ is the map
$$\begin{matrix}
J_0\colon & A^2 & \to & A^2 \\
& (x,y) & \mapsto & (y,-x).
\end{matrix}$$
\end{df}

\begin{rem}
The form $h_{J_0}$ is the standard $(G,\sigma)$-inner product on $A^2$, i.e., for $x,y\in A^2$, $h_{J_0}(x,y)=\sigma(x)^ty$. By Proposition~\ref{RegEl-h}, $J_0\in \mathfrak C$. Moreover, $\Aut(h_{J_0})=\OO_2(G,\sigma)$.
\end{rem}

\begin{prop}\label{CompStr-SympBas}
Let $J$ be a complex structure on $A^2$. $J\in\mathfrak C$ if and only if there exists $w\in \Is_G(\omega)$ such that $(J(w),w)$ is a $(G,\sigma)$-symplectic basis.
\end{prop}

\begin{proof}
($\Rightarrow$) Let $J\in\mathfrak C$ and $w'\in\Is_G(\omega)$. We denote $b:=h_J(w',w')\in B^{\sigma}_+$ and $w:=w'b^{-\frac{1}{2}}$, then $h_J(w,w)=1$. Furthermore,
\begin{align*}
    \omega(J(w),J(w))& =h_J(w,J(w)) =\sigma(h_J(J(w),w))=\sigma(\omega(w,w))=0,\\
    \omega(J(w),w)& =h_J(w,w)=1.
\end{align*}
Therefore, $(J(w),w)$ is a $(G,\sigma)$-symplectic basis.

\vspace{2mm}
\noindent ($\Leftarrow$) Let $w\in A^2$ such that $(J(w),w)$ is a $(G,\sigma)$-symplectic basis. Then,
\begin{align*}
    h_J(w,w)& =\omega(J(w),w)=1,\\
    h_J(J(w),J(w))& =\omega(J^2(w),J(w))=-\omega(w,J(w))=1,\\
    h_J(J(w),w)& =\omega(J^2(w),w)=-\omega(w,w)=0.
\end{align*}
Therefore, $(J(w),w)$ is an orthonormal basis for $h_J$, so $h_J$ is an $(G,\sigma)$-inner product.
\end{proof}

\begin{teo}
The group $\Sp_2(G,\sigma)$ acts transitively on $\mathfrak C$ by conjugation. The stabilizer of the standard complex structure is $\KSp_2(G,\sigma)$. In particular, $\mathfrak C$ is a model of the Riemannian symmetric space of $\Sp_2(G,\sigma)$.
\end{teo}

\begin{proof}
1. First, we prove that $\Sp_2(G,\sigma)$ acts on $\mathfrak C$ by conjugation. Let $J\in \mathfrak C$ and $g\in\Sp_2(G,\sigma)$. Consider $J':=g^{-1}Jg$. Then, $(J')^2=g^{-1}J^2g=-\Id$ so $J'$ is a complex structure on $A^2$. For $x\in\Is_G(\omega)$, $g(x)\in\Is_G(\omega)$ and we obtain
\begin{align*}
    h_{J'}(x,x)&=\omega(J'(x),x)=\omega(g^{-1}Jg(x),x)\\
    &=\omega(Jg(x),g(x))=h_J(g(x),g(x))\in B^{\sigma}_+.
\end{align*}
Moreover, the complex structure $J'$ preserves $\Is_G(\omega)$. Therefore, $h_{J'}$ is a $(G,\sigma)$-inner product on $A^2$, i.e., $J'\in \mathfrak C$.

2. Second, we prove that this action is transitive. Let $J\in \mathfrak C$, and $(J(w),w)$ be a $(G,\sigma)$-symplectic basis from Proposition~\ref{CompStr-SympBas}. By Proposition~\ref{trans_bas}, the group $\Sp_2(G,\sigma)$ acts transitively on $(G,\sigma)$-symplectic bases, i.e., there exists $g\in\Sp_2(G,\sigma)$ which maps the standard symplectic basis to $(J(w),w)$. That means, $g$ maps the standard complex structure $J_0$ to $J$. So the action is transitive.

3. Finally, we compute the stabilizer of $J_0$. An element $g\in\Stab_{\Sp_2(G,\sigma)}(J_0)$ if and only if
$g\in\Sp_2(G,\sigma)$ and $g\in \Aut(h_{J_0})=\OO_2(A,\sigma)$, i.e.
\begin{equation*}
g\in\Sp_2(G,\sigma)\cap\OO_2(A,\sigma)= \KSp_2(G,\sigma).\qedhere
\end{equation*}
\end{proof}

\begin{cor}
The orbit map
$$\begin{matrix}
\pi_\mathfrak C\colon & \Sp_2(G,\sigma)/\KSp_2(G,\sigma) & \to & \mathfrak C \\
 & M\KSp_2(G,\sigma) & \mapsto & M J_0 M^{-1}
\end{matrix}$$
is an $\Sp_2(G,\sigma)$-equivariant homeomorphism.
\end{cor}

The space $\mathfrak{C}$ is a smooth manifold with the following tangent space at $J\in\mathfrak{C}$:
$$T_J\mathfrak{C}=\left\{L\in\End_A(A^2)\midwd\begin{array}{l}
JL+LJ=0,\\
L(\Is_G(\omega))\subseteq\overline{\Is_G(\omega)}\subseteq \bar G^2,\\
h_J\text{ is a $(G,\sigma)$-symmetric}
\end{array}\right\}.$$
The group $\Sp_2(G,\sigma)$ acts by conjugation on the tangent bundle $T\mathfrak{C}$, i.e., for $g\in \Sp_2(G,\sigma)$, $g.(J,L)=(gJg^{-1},gLg^{-1})$.
For the proof of this fact, we refer to~\cite[Proposition~6.3]{HKRW}.

\subsection{Projective space models}

Let $B_\CC:=B\otimes_\R\CC\subseteq A\otimes_\R\CC=A_\CC$ be the complexification of $B$ and $G_\CC$ the Lie subgroup of $A_\mathbb C^\times$ as in Section~\ref{discon-ext}.

Slightly abusing our notation, we denote by $\sigma\colon A_\mathbb C\to A_\mathbb C$ the $\CC$-linear extension of $\sigma\colon A\to A$, i.e., for every $x,y\in B$
$$\sigma(x+iy)=\sigma(x)+i\sigma(y)$$
and by $\bar\sigma$ the $\CC$-antilinear extension of $\sigma$, i.e., for every $x,y\in B$
$$\bar\sigma(x+iy)=\sigma(x)-i\sigma(y).$$

Notice that in general neither $(B_\CC,\sigma)$ nor $(B_\CC,\bar\sigma)$ is Hermitian. Moreover, the Lie algebra $(B_\CC,\sigma)$ is never Hermitian but it is always of Jordan type. However, $(B_\CC,\bar\sigma)$ is even not always of Jordan type (cf.~Remark~\ref{rem:nonHerm.complexification}).

We extend $\omega$ in the complex linear way to $\omega_\CC$ on $A_\CC^2$. Every complex structure $J$  on $A^2$ extends to an $A_\CC$-linear operator on $A^2_\CC$ in the complex linear way. We denote this extension by $J_\CC$.

\begin{prop}
For every $J\in \mathfrak C$, there exist regular isotropic elements $x,y\in A_\CC^2$ such that $J_\CC(x)=ix$, $J_\CC(y)=-iy$. The lines $xA_\CC$, $yA_\CC$ are uniquely determined by $J$.
\end{prop}

\begin{proof}
Since $\Sp_2(G,\sigma)$ acts transitively on $\mathfrak C$, it is enough to prove the proposition for the standard complex structure $J_0$.

Since $J_0(a,b)^t=(b,-a)^t$, $(b,-a)^t=i(a,b)^t$ if and only if $b=ai$, i.e., $x=(1,i)^ta$ for $a\in A_\CC$, i.e., $xA_\CC$ is uniquely defined. Analogously, $y=(i,1)^ta$ for $a\in A_\CC$, i.e., $yA_\CC$ is uniquely defined. Moreover, $\omega_\CC(x,x)=\omega_\CC(y,y)=0$, so $x$ are $y$ are isotropic.
\end{proof}

For a complex structure $J\in\mathfrak C$, we denote by $l_J$ the line $yA_\CC$ such that $J_\CC(y)=-yi$.

\begin{rem}
For every $(G,\sigma)$-symplectic basis $(w,u)$ of $(A^2,\omega)$, elements $u+wi, w+iu$ are isotropic.
\end{rem}

\begin{df}
We call the spaces 
\begin{align*}
    \mathfrak P:=\mathfrak P^+:=&\{(u+wi)A_\CC\mid (w,u)\text{ is a $(G,\sigma)$-symplectic basis of $A^2$}\},\\
    \mathfrak P^-:=&\{(w+ui)A_\CC\mid (w,u)\text{ is a $(G,\sigma)$-symplectic basis of $A^2$}\}=\overline{\mathfrak P^+}
\end{align*}
the \defin{projective space models} of the Riemannian symmetric space of $\Sp_2(G,\sigma)$.
\end{df}
 
\noindent These spaces are smooth manifolds, and, similarly as in Proposition~\ref{prop:islines.tangent}, one can describe the tangent space at $l\in \mathfrak P^\pm$:
\begin{equation*}
    T_l\mathfrak P^\pm=\left\{Q\in\mathrm{Hom}_A(l,A_\mathbb C^2/l)\midwd\begin{array}{l} \omega_\mathbb C(Q(x),x)\in B^{\sigma}_\mathbb C\text{ for $x=u+wi$ where $(w,u)$ is} \\
    \text{a $(G,\sigma)$-symplectic basis of $A^2$}
    \end{array}\right\}.
\end{equation*}

The group $\Sp_2(G,\sigma)$ acts on the spaces $\mathfrak P^+$ and $\mathfrak P^-$ by the left matrix multiplication. Moreover, complex conjugation provides an $\Sp_2(G,\sigma)$-equivariant homeomorphism between $\mathfrak P^+$ and $\mathfrak P^-$.

The following proposition justifies the name of $\mathfrak P^\pm$:

\begin{prop}
The map
$$\begin{array}{cccl}
F_{\mathfrak C,\mathfrak P}\colon& \mathfrak C& \to & \mathfrak P\\
 & J & \mapsto & l_J
\end{array}$$
defines a homeomorphism that is equivariant under the action of $\Sp_2(G,\sigma)$. In particular, $\mathfrak P$ is a model of the Riemannian symmetric space of $\Sp_2(G,\sigma)$.
\end{prop}

\begin{proof}
1. First, we show that $l_J\in\mathfrak P$. It is again enough to prove this for the standard complex structure $J_0$. We take $v:=(i,1)^t=(0,1)^t+(1,0)^ti$, then $l_{J_0}=vA_\CC$, $((1,0)^t,(0,1)^t)$ is a $(G,\sigma)$-symplectic basis.

2. We show that $F_{\mathfrak C,\mathfrak P}$ is surjective. Let $v=u+wi$ for a $(G,\sigma)$-symplectic basis $(w,u)$ of $(A^2,\omega)$. We define the following complex structure: $J(u)=w$, $J(w)=-u$. By Proposition~\ref{CompStr-SympBas}, $J\in\mathfrak C$.
Since
$$J_\CC(v)=J_\CC(u+iw)=w-iu=-i(u+iw)=-iv,$$
we obtain $F_{\mathfrak C,\mathfrak P}(J)=vA_\CC$, i.e., $F_{\mathfrak C,\mathfrak P}$ is surjective.

3. We show that $F_{\mathfrak C,\mathfrak P}$ is injective. Let $l_J=l_{J'}=yA_\CC$ for $J,J'\in\mathfrak C$ and $y=y_1+y_2i\in A^2_\CC$. Then $J(y_1)=J'(y_1)=-y_2$, $J(y_2)=J'(y_2)=y_1$ and $(y_1,y_2)$ is a basis of $A^2$, i.e., $J=J'$.

4. Finally, we show the equivariance of $F_{\mathfrak C,\mathfrak P}$. Let $M\in\Sp_2(G,\sigma)$, $J\in\mathfrak C$ and $(w,u)$ be a $(G,\sigma)$-symplectic basis of $(A^2,\omega)$ such that $w:=J(u)$, $J(w)=-u$. Then for $v=u+wi$, $J_\CC(v)=-iv$.

Moreover, $MJM^{-1}(Mu)=Mw$, $MJM^{-1}(Mw)=-Mu$ where $(Mw, Mu)$ is also a $(G,\sigma)$-symplectic basis of $(A^2,\omega)$. Then $(MJM^{-1})_\CC(M_\CC v)=-iM_\CC v$ and
$$F_{\mathfrak C,\mathfrak P}(MJM^{-1})=(M_\CC v)A_\CC=M_\CC(vA_\CC)=M_\CC F_{\mathfrak C,\mathfrak P}(J),$$
i.e., $F_{\mathfrak C,\mathfrak P}$ is equivariant with respect to the $\Sp_2(G,\sigma)$-action.
\end{proof}

Similarly, we obtain the following Proposition:
\begin{prop}
The map
$$\begin{array}{cccl}
F_{\mathfrak C,\mathfrak P^-}\colon& \mathfrak C& \to & \mathfrak P^-\\
 & J & \mapsto & \overline{l_J}
\end{array}$$
defines a homeomorphism that is equivariant under the action of $\Sp_2(G,\sigma)$. In particular, $\mathfrak P^-$ is a model of the Riemannian symmetric space of $\Sp_2(G,\sigma)$.
\end{prop}

Since $\mathfrak P^+$ and $\mathfrak P^-$ are related by complex conjugation, in what follows we will only discuss $\mathfrak P^+$ keeping in mind that the corresponding statements apply equally to $\mathfrak P^-$.

\begin{cor}
The orbit map
$$\begin{matrix}
\pi_\mathfrak P\colon & \Sp_2(G,\sigma)/\KSp_2(G,\sigma) & \to & \mathfrak P \\
 & M\KSp_2(G,\sigma) & \mapsto & M(i,1)^tA_\CC
\end{matrix}$$
is an $\Sp_2(G,\sigma)$-equivariant homeomorphism.
\end{cor}

Consider the following $\bar\sigma$-sesquilinear form on $A^2_\CC$:
\begin{equation}\label{eq:form_h}
    h(x,y):=i\omega_\CC(\bar x,y).    
\end{equation}
It is $\bar\sigma$-symmetric:
$$h(y,x)=i\omega_\CC(\bar y,x)=\bar\sigma((-i)(-\omega_\CC(\bar x,y)))=\bar\sigma(h(x,y)),$$
and in the basis $e_1:=\left(\frac{1}{\sqrt{2}},\frac{i}{\sqrt{2}}\right)^t$, $e_2:=\left(\frac{1}{\sqrt{2}},-\frac{i}{\sqrt{2}}\right)^t$, the form $h$ is represented by the matrix $\begin{pmatrix} -1 & 0 \\ 0 & 1\end{pmatrix}$.

The following lemma is immediate:
\begin{lem}\label{G-SympBas-h=2}
Let $G_\mathbb C$ be a Lie subgroup of $A^\times_\mathbb C$ such that $\Lie(G_\mathbb C)=B_\mathbb C$ and $(B^{\sigma}_\CC)^\times\subseteq G_\CC$. In this case, the group $\Sp_2(G_\CC,\sigma)$ is well defined. For $v:=u+wi\in A_\CC^2$ such that $(w,u)$ is a basis of $A^2$, the basis $(w,u)$ is $(G,\sigma)$-symplectic if and only if  $h(v,v)=2$ and $v\in\Is_{G_\CC}(\omega_\CC)$.
\end{lem}

\begin{rem}
The group $\Sp_2(G,\sigma)$ acts on $A_\CC^2$ preserving both forms $\omega_\CC$ and $h$. So we can see $\Sp_2(G,\sigma)$ as a subgroup of $\Aut(h)$.
\end{rem}

\subsection{Precompact model}\label{Precomp_Mod_R}

In this section, we identify the symmetric space of $\Sp_2(G,\sigma)$ as an open subset of $B^\sigma_\mathbb C$ with a compact closure. 

We define the following domain $\mathfrak B$:
$$\mathfrak B\!:=\{c\in B^\sigma_\CC\mid 1-\bar c c\in (B^{\bar\sigma}_\CC)_+\}.$$
This is an open domain in $B^\sigma_\CC$, therefore, it is naturally a complex manifold, and the tangent space $T_c\mathfrak B=B^\sigma_\CC$ at every $c\in \mathfrak B$. Moreover, its closure agrees with the compact domain $D(B_\mathbb C^\sigma,\bar\sigma)$ from Proposition~\ref{comp_disc}, i.e., it is precompact.

We consider the following $\Sp_2(G_\CC,\sigma)$-transformation:
$$T:=\frac{1}{\sqrt{2}}\begin{pmatrix}
1 & i \\
i & 1
\end{pmatrix}\in \Sp_2(G_\CC,\sigma).$$
It maps $h$ (cf.~\eqref{eq:form_h}) to the indefinite form represented by the matrix $\begin{pmatrix}
    -1 & 0 \\
    0 & 1
\end{pmatrix}$: i.e., $\bar\sigma(T)^t[h]T=\begin{pmatrix}
    -1 & 0 \\
    0 & 1
\end{pmatrix}$. Since $T\in\Sp_2(G_\CC,\sigma)$, it stabilizes the set $\Is_{G_\CC}(\omega_\CC)$.

\begin{prop}\label{Proj-Precomp-R}
We consider the map $F_{\mathfrak P,\mathfrak B}:=\Phi\circ T^{-1}\colon\mathfrak P\to \mathfrak B$ where
$$\begin{array}{cccl}
\Phi\colon & T^{-1}\mathfrak P & \to & \mathfrak B\\
 & \begin{pmatrix}
     v_1 \\ v_2
 \end{pmatrix}A_\CC & \mapsto & v_1v_2^{-1}.
\end{array}$$
Both maps $F_{\mathfrak P,\mathfrak B}$ and $\Phi$ are homeomorphisms.
\end{prop}

\begin{proof} Let $vA_\CC\in\mathfrak P$ for $v=(v_1,v_2)^t\in \Is_{G_\CC}(\omega_\CC)$ and $v=u+wi$ where $(w,u)$ is a~$(G,\sigma)$-symplectic basis of $(A^2,\omega)$. Then by Lemma~\ref{G-SympBas-h=2}:
$$2=h(v,v)=-\bar\sigma(x_1)x_1+\bar\sigma(x_2)x_2\in (B^{\bar\sigma}_\mathbb C)_+$$
where $T^{-1}v=:(x_1,x_2)^t$. Therefore, $\bar\sigma(x_2)x_2=\bar\sigma(x_1)x_1+2\in (B^{\bar\sigma}_\mathbb C)_+$ because $(B^{\bar\sigma}_\mathbb C)_+$ is a~convex cone. This means that $x_2$ is invertible, i.e., $x_2\in G$. We define 
$$(c,1)^t:=(x_1x_2^{-1},1)^t\in\Is_{G_\CC}(\omega_\CC).$$ Then $(c,1)^tA_\CC=(T^{-1}v)A_\CC$. By Corollary~\ref{cor:(c,1)-isotropic}, $(c,1)^t\in\Is_{G_\CC}(\omega_\CC)$ if and only if $c\in B^{\sigma}_\CC$. Moreover, 
$$h((c,1)^t,(c,1)^t)=1-\bar c c\in (B^{\bar\sigma}_\mathbb C)_+.$$ Therefore, $\Phi(T^{-1}v)\in \mathfrak B$.

The map $\Phi$ is injective because $T$ is injective and, if $$x_1x_2^{-1}=\Phi(x_1,x_2)^t=\Phi(y_1,y_2)^t=y_1y_2^{-1},$$
then $(y_1,y_2)^t=(x_1,x_2)^tx_2^{-1}y_2$, i.e., $(x_1,x_2)^tA_\CC=(y_1,y_2)^tA_\CC$.

Finally, the map $\Phi$ is surjective because for every $c\in \mathfrak B$, $(c,1)^tA_\CC=(T^{-1}v)A_\CC$ where $v:=T(c,1)^t\sqrt{2}(1-\bar c c)^{-\frac{1}{2}}$. Then $v\in \Is_{G_\CC}(\omega_\CC)$ and $h(v,v)=2$. Therefore, by Lemma~\ref{G-SympBas-h=2}, $v=u+wi$ where $(w,u)$ is a $(G,\sigma)$-symplectic basis of $(A^2,\omega)$. Therefore, $vA_\CC\in T^{-1}\mathfrak P$.
\end{proof}

\begin{rem}
The group $T^{-1}\Sp_2(G,\sigma)T<\Sp_2(G_\CC,\sigma)$ acts transitively on $\mathfrak B$ by M\"obius transformations. Thus $F_{\mathfrak P,\mathfrak B}$ is equivariant with respect to this action of $\Sp_2(G,\sigma)$ on $\mathfrak B$ and the action on $\mathfrak P$.
\end{rem}

\subsection{Compactification and Shilov boundary}\label{Shilov_R}

Let $(B,\sigma)$ be a Hermitian Lie subalgebra of $A$. Then
$$\mathfrak B:=\{c\in B^\sigma_\CC\mid 1-\bar cc\in (B^{\bar\sigma}_\CC)_+\}$$
is precompact. We take the topological closure of $\mathfrak B$ in  $B^\sigma$:
$$\overline{\mathfrak B}:=\{c\in B^\sigma_\CC\mid 1-\bar cc\in (B^{\bar\sigma}_\CC)_{\geq 0}\}.$$
The boundary of $\overline{\mathfrak B}$ contains the following compact subspace:
$$\check{S}(\mathfrak B):=\{c\in B^\sigma_\CC\mid 1-\bar cc=0\}=\OO(G_\CC,\bar\sigma)\cap B^\sigma_\CC.$$

\begin{df}
We call $\check{S}(\mathfrak B)$ the \defin{Shilov boundary} of the precompact model $\mathfrak B$.
\end{df}

\begin{rem}
The map $\Phi^{-1}$ extends to the boundary of $\overline{\mathfrak B}$ and remains continuous and bijective. Since the boundary is compact, it is a homeomorphism. Therefore, we can see the boundary also in the projective space model. In particular, we can see the Shilov boundary there.
\end{rem}

The following proposition describes the Shilov boundary in the projective space model.

\begin{prop}
The preimage of the Shilov boundary $\check{S}(\mathfrak B)$ in $\Is_{G_\CC}(\omega)$ under the map $F_{\mathfrak P,\mathfrak B}$ gives a compact subset of the boundary of the projective space model. It consists of all lines of the form $xA_\CC$ such that $x\in\Is_G(\omega)$.
\end{prop}

\begin{proof}
Note that the line $l\in \Is_{G_\CC}(\omega)$ is of the form $xA_\CC$ for some $x\in\Is_G(\omega)$ if and only if $\bar l =l$.

Let $c\in \check{S}(\mathfrak B)$, then $\bar c^{-1}=c$ and
$$F_{\mathfrak P,\mathfrak B}\left(\overline{F_{\mathfrak P,\mathfrak B}^{-1}(c)}\right)=\Phi\left(\!\!\begin{pmatrix}0 & i\\ i & 0\end{pmatrix}\Clm{\bar c}{1}\!\!\right)=\Phi\left(\!\!\Clm{1}{\bar ci}\!\!\right)=\bar c^{-1}=c,$$
i.e., $F_{\mathfrak P,\mathfrak B}^{-1}(c)=xA_\mathbb C$ for $x=(c,1)^t\in\Is_G(\omega)$.

Let $xA_\CC$ be a line spanned by $x=(x_1,x_2)^t\in\Is_G(\omega)$, then
$$c:=F_{\mathfrak P,\mathfrak B}(xA_\CC)=(x_1-ix_2)(-ix_1+x_2)^{-1}.$$
By Corollary~\ref{cor:(c,1)-isotropic}, $x\in\Is_G(\omega)$ if and only if $c\in B^\sigma_\CC$. Therefore,
\begin{align*}
    \bar c c&=(x_1+ix_2)(ix_1+x_2)^{-1}(x_1-ix_2)(-ix_1+x_2)^{-1}\\
    &=i(x_1+ix_2)(x_1-ix_2)^{-1}(x_1-ix_2)(-ix_1+x_2)^{-1}\\
    &=i(x_1+ix_2)(-ix_1+x_2)^{-1}=(x_1+ix_2)(x_1+ix_2)^{-1}=1.    
\end{align*}
Therefore, $F_{\mathfrak P,\mathfrak B}(xA_\mathbb C)\in \check{S}(\mathfrak B)$.
\end{proof}

\begin{cor}
The space $\PIs_{G}(\omega)$ embedded into $\PIs_{G_\CC}(\omega_\mathbb C)$ as:
$$xA \mapsto xA_\CC$$
is the Shilov boundary in the projective space model. This is a  compact orbit of the action of $\Sp_2(G,\sigma)$ on the boundary of the projective space model.
\end{cor}

\subsection{Half-space models}

The last models, we discuss in this chapter are the upper half- and lower half-space models. They generalize the upper half-plane model for the hyperbolic plane and the Siegel space model for the symmetric space for the group $\Sp_{2n}(\R)$.

\begin{df}
The \defin{upper (lower) half-spaces} are defined by
$$\mathfrak U:=\mathfrak U^+:=\{z\in B_\CC^{\sigma}\mid \Imm(z)\in B^{\sigma}_+\}$$
and
$$\mathfrak U^-:=\{z\in B_\CC^{\sigma}\mid \Imm(z)\in B^{\sigma}_-\}$$
respectively.
\end{df}

The spaces $\mathfrak U^\pm$ are open domains in $B_\CC^{\sigma}$. Therefore, they are naturally complex manifolds, and the tangent spaces $T_z\mathfrak U^\pm= B_\CC^{\sigma}$, for every $z\in \mathfrak U^\pm$.

Note that $\mathfrak U^-$ and $\mathfrak U^+$ are related by complex conjugation. Thus, in what follows we will only discuss $\mathfrak U=\mathfrak U^+$. The corresponding statements for $\mathfrak U^-$ can be easily reformulated.

\begin{prop}
The group $\Sp_2(G,\sigma)$ acts on $\mathfrak U$ via M\"obius transformations:
$$z\mapsto M.z=(az+b)(cz+d)^{-1}\text{, where } M=\begin{pmatrix}a & b \\ c & d \end{pmatrix}.$$
\end{prop}

\begin{proof}
Since $\Sp_2(G,\sigma)$ is generated by the matrices
$$\begin{pmatrix}
    a & 0 \\
    0 & \sigma(a)^{-1}
    \end{pmatrix},\quad
\begin{pmatrix}
    0 & 1 \\
    -1 & 0
    \end{pmatrix},\quad
\begin{pmatrix}
    1 & b \\
    0 & 1
    \end{pmatrix}$$
where $a\in G$, $b\in B^{\sigma}$, we prove that $M.z\in \mathfrak U$ on these generators.

If $M:=\begin{pmatrix}
    1 & b \\
    0 & 1
    \end{pmatrix}$ with $b\in B^{\sigma}$, then $M.z=z+b\in B^{\sigma}_\CC$ and $\Imm(M.z)=\Imm(z)\in B^{\sigma}_+$.

If $M:=\begin{pmatrix}
    0 & 1 \\
    -1 & 0
    \end{pmatrix}$,
then $M.z=-z^{-1}\in B^{\sigma}_\CC$. If $z=x+yi$, then $$z^{-1}=y^{-1}x(y+xy^{-1}x)^{-1}-(y+xy^{-1}x)^{-1}i,$$
i.e., $\Imm(M.z)=(y+xy^{-1}x)^{-1}$. For $y\in B^{\sigma}_+$, also $y^{-1}\in B^{\sigma}_+$.

Let $y^{-1}=\sigma(p)p$ for some $p\in B^\times$, then
$$y+xy^{-1}x=y+\sigma(px)px\in B^{\sigma}_+.$$
Therefore, $\Imm(M.z)=(y+xy^{-1}x)^{-1}\in B^{\sigma}_+.$

If $M:=\begin{pmatrix}
    a & 0 \\
    0 & \sigma(a)^{-1}
    \end{pmatrix}$
for $a\in G$, then $M.z=az\sigma(a)\in B^{\sigma}_\CC$ because $B^{\sigma}$ is preserved by action of $G$. $\Imm(M.z)=a\Imm(z)\sigma(a)\in B^{\sigma}_+$ because $B^{\sigma}_+$ is closed by action of $G$.
\end{proof}

\begin{prop}
The orbit map
$$\begin{matrix}
\tilde\pi_\mathfrak U\colon & \Sp_2(G,\sigma) & \to & \mathfrak U\\
 & M & \to & M.1i
\end{matrix}$$
is continuous, proper and surjective, i.e., $\Sp_2(G,\sigma)$ acts transitively on $\mathfrak U$. The stabilizer of $1i$ is $\KSp_2(G,\sigma)$.
\end{prop}

\begin{proof}
Let $z=x+yi\in\mathfrak U$ then $y=u^2$ for some $u\in (B^{\sigma})^\times$. Then
$$\pi\left(\!\!
\begin{pmatrix}
1 & x \\
0 & 1
\end{pmatrix}
\begin{pmatrix}
u & 0 \\
0 & u^{-1}
\end{pmatrix}\!\!\right)=
\pi\left(\!\!
\begin{pmatrix}
u & xu^{-1} \\
0 & u^{-1}
\end{pmatrix}\!\!\right)=x+yi=z.$$

\noindent A matrix $M=\begin{pmatrix}
a & b \\
c & d
\end{pmatrix}$
stabilizes $1i$ if and only if
$$1i=M.1i=(ai+b)(ci+d)^{-1}=(ai+b)(-c+di)^{-1}i.$$
Thus, $a=d$ and $c=-b$, i.e., $M\in\KSp_2(G,\sigma)$.
\end{proof}

\begin{cor}
The map $\tilde\pi_\mathfrak U$ induces a homeomorphism
$$\begin{matrix}
\pi_\mathfrak U \colon & \Sp_2(G,\sigma)/\KSp_2(G,\sigma) & \to & \mathfrak U \\
 & M\KSp_2(G,\sigma) & \mapsto & M.1i\ .
\end{matrix}$$
This homeomorphism is equivariant under the action of $\Sp_2(G,\sigma)$ by left multiplication on $\Sp_2(G,\sigma)/\KSp_2(G,\sigma)$ and by M\"obius transformations on $\mathfrak U$. In particular, the space $\mathfrak U$ is a model of the Riemannian symmetric space of $\Sp_2(G,\sigma)$.
\end{cor}

\begin{cor}
    The map
    $$\begin{array}{rccc}
    F_{\mathfrak P,\mathfrak U}\colon &\mathfrak P & \to & \mathfrak U\\
    &\begin{pmatrix}
        a_1 \\ a_2
    \end{pmatrix}A_\mathbb C & \mapsto & a_1a_2^{-1}
    \end{array}$$
    provides an $\Sp_2(G,\sigma)$-equivariant homeomorphism between the upper half-space and the projective space model.
\end{cor}

The domain $\mathfrak U$ admits a Riemannian metric as described in~\cite[Excercise~8, Chapter~XIII]{Faraut} which is $\Sp_2(G,\sigma)$-invariant and together with the natural complex structure on $\mathfrak U$, it defines a structure of Hermitian symmetric space on $\mathfrak U$. 

Before we define the metric, we remind the definition of the \defin{quadratic representation} (see~\cite[Section~II.3]{Faraut}): let $(J,\circ)$ be a Jordan algebra, for $x,y,z\in J$, we denote 
$$\{x,y,z\}:=z\circ(x\circ y)+x\circ(z\circ y)-(x\circ z)\circ y.$$ 
Then the Riemannian metric on $\mathfrak U$ can be expressed as follows: for $z=x+yi\in \mathfrak U$, $v,w\in T_z\mathfrak U=B^\sigma_\mathbb C$, $v=v_x+v_yi$, $w=w_x+w_yi$:
\begin{align*}
    g_z(v,w):=&\tr\left(\Ree\left(\{y^{-\frac{1}{2}},\{\bar v,y^{-1},w\},y^{-\frac{1}{2}}\}\right)\right)\\
    =&\tr\left(\{y^{-\frac{1}{2}},\{v_x,y^{-1},w_x\},y^{-\frac{1}{2}}\}+\{y^{-\frac{1}{2}},\{v_y,y^{-1},w_y\},y^{-\frac{1}{2}}\}\right).    
\end{align*}
Using the isomorphisms between models, this metric can by transferred to all four models of the symmetric space of $\Sp_2(G,\sigma)$ that we discussed before.

Notice that if $z=i$, then the metric becomes particularly easy: 
$$g_i(v,w)=\tr(v_x\circ w_x+v_y\circ w_y)=\beta(v_x,w_x)+\beta(v_y,w_y)$$
where $\beta$ is the inner product on $B^\sigma$ defined in~\eqref{inner_prod_B}.

\section{Models for the Riemannian symmetric space of \texorpdfstring{$G$}{G}}\label{hatG-models}

Let $(A,\sigma)$ be an involutive finite dimensional $\R$-algebra. The goal of this section is to construct different models of the symmetric space for $G$ for $\Lie G=B$ where $B$ is a Hermitian Lie subalgebra $A$, seen as sub-manifolds of the corresponding models of the Riemannian symmetric space of $\Sp_2(G,\sigma)$. For this, we use the identification $G\cong \hat G=\Sp_2(G,\sigma)\cap\OO_{(1,1)}(A,\sigma)$ discussed in Section~\ref{sec:hat_G}. Moreover, we will not discuss explicitly the maps between models and the Riemannian metric because they are simply the restrictions of the corresponding ones from Section~\ref{G-models}.

Let $V=A^2$ and $\omega$ be the standard symplectic form in $A^2$. We define the following $\sigma$-sesquilinear, $(G,\sigma)$-symmetric form $h$: for $x,y\in A^2$, $h(x,y):=\sigma(x)^t H y$, where $H=\begin{pmatrix}
    0 & 1 \\
    1 & 0
\end{pmatrix}$.
Then $\OO(h)=\OO_{(1,1)}(A,\sigma)$.

\subsection{Space of complex structures}\label{G-Comp_Str_Mod}

The first model we construct is the space of almost orthogonal complex structures.

\begin{df}
Let $J\in\End_A(V)$. $J$ is called \defin{almost orthogonal} if $h(Jx,Jy)=-h(x,y)$.
\end{df}

An almost orthogonal endomorphism is not an element of $\OO_{(1,1)}(A,\sigma)$. However, a product of two of them is an element of $\OO_{(1,1)}(A,\sigma)$.

Let now $J$ be a complex strucure on $A^2$ (cf. Definition~\ref{df:complex.str}). As before, we can define the following $\sigma$-sesquilinear form
$$\begin{matrix}
h_J\colon & A^2\times A^2 & \to & A \\
& (x,y) & \mapsto & \omega(J(x),y).
\end{matrix}$$

\noindent We consider the following space:
\begin{align*}
    \mathfrak C_{\hat G}:=&\;\left\{J\text{ complex structure on $A^2$}\midwd\begin{array}{l}
    J(\Is_G(\omega))=\Is_G(\omega),\\
    h(J(\cdot),J(\cdot))=-h(\cdot,\cdot),\\
    h_J\text{ is a $(G,\sigma)$-inner product}
    \end{array}\right\}\\
    =&\;\left\{J\text{ complex structure on $A^2$}\midwd\begin{array}{l}
    J(\Is_G(\omega))=\Is_G(\omega),\\
    J(\Is(h))=\Is(h),\\
    h_J\text{ is a $(G,\sigma)$-inner product}
    \end{array}\right\}.    
\end{align*}

\noindent Notice that the standard complex structure $J_0$ on $A^2$ is almost orthogonal, i.e., $J_0\in \mathfrak C_{\hat G}$.

\begin{teo}
The group $\hat G$ acts transitively on $\mathfrak C_{\hat G}$ by conjugation. The stabilizer of the standard complex structure is $\OO(G,\sigma)$. In particular, $\mathfrak C_{\hat G}$ is a model of the Riemannian symmetric space of $\hat G$.
\end{teo}

\begin{proof}
Let $J\in \mathfrak C_{\hat G}$. Since $J$ is an almost orthogonal complex structure, then $g_J:=JJ_0$ is an element of $\OO_{(1,1)}(A,\sigma)$ and of $\Sp_{2}(G,\sigma)$. Therefore,
$$g_J=\begin{pmatrix}
    g & 0 \\
    0 & \sigma(g)^{-1}
\end{pmatrix}\in \OO_{(1,1)}(A,\sigma)\cap \Sp_2(G,\sigma)=\hat G,$$ where $g\in G$, i.e., $J=\begin{pmatrix}
    0 & -g \\
    \sigma(g)^{-1} & 0
\end{pmatrix}$. 
The condition that $h_J$ is a $(G,\sigma)$-inner product implies that $-g\in B^\sigma_+$. Let $b\in B^\sigma_+\subseteq G$ such that $b^2=-g$. Then
$$\begin{pmatrix}
    b & 0 \\
    0 & b^{-1}
\end{pmatrix} J_0 \begin{pmatrix}
    b & 0 \\
    0 & b^{-1}
\end{pmatrix}^{-1}=J.$$
Thus the action of $\hat G$ on $\mathfrak C_{\hat G}$ is transitive. The stabilizer of the standard complex structure agrees with the subgroup $\OO(\hat G,\sigma):=\left\{\begin{pmatrix}
    g & 0 \\
    0 & g
\end{pmatrix}\midwd g\in \OO(G,\sigma)\right\}$ which is a~maximal compact subgroup of $\hat G$.
\end{proof}

The space $\mathfrak{C}_{\hat G}$ is a smooth manifold with the following tangent space at $J\in\mathfrak{C}_{\hat G}$:
$$T_J\mathfrak{C}_{\hat G}=\left\{L\in\End_A(A^2)\midwd\begin{array}{l}
JL+LJ=0,\\
L(\Is_G(\omega))\subseteq\overline{\Is_G(\omega)}\subseteq \bar G^2,\\
h(L(\cdot),\cdot)+h(\cdot,L(\cdot)=0,\\
h_J\text{ is a $(G,\sigma)$-symmetric}
\end{array}\right\}.$$
The group $\hat G$ acts by conjugation on the tangent bundle $T\mathfrak{C}$, i.e., for $g\in \hat G$, $g.(J,L)=(gJg^{-1},gLg^{-1})$.
For the proof of this fact, we refer to~\cite[Proposition~4.3]{HKRW}.

\subsection{Projective space models}

We define the following spaces: 
\begin{align*}
    \mathfrak P_{\hat G}:=\mathfrak P_{\hat G}^+:=&\{uA \mid h(u,u)\in B^\sigma_+\}\cap\PIs_G(\omega)\\
    \mathfrak P_{\hat G}^-:=&\{uA \mid h(u,u)\in B^\sigma_-\}\cap\PIs_G(\omega)=-\mathfrak P_{\hat G}^+.
\end{align*}
Notice that in fact to define $\mathfrak P_{\hat G}^\pm$, it is enough to take the intersection of $\{uA \mid h(u,u)\in B^\sigma_\pm\}$ with $\PIs(\omega)$ because the conditions $h(u,u)\in B^\sigma_+$ and $\omega(u,u)=0$ imply that $uA\in \PIs_G(\omega)$. The spaces $\mathfrak P_{\hat G}^\pm$ are smooth manifolds, and one can describe the tangent space at $l\in \mathfrak P_{\hat G}^\pm$:
\begin{equation*}
    T_l\mathfrak P_{\hat G}^\pm=\left\{Q\in\mathrm{Hom}_A(l,A/l)\midwd\begin{array}{l} h(Q(x),x)\in B^{\sigma}\text{ for $x\in\Is_G(\omega)$}\\ \text{such that $l=xA$ and $h(x,x)\in B^\sigma_\pm$}
    \end{array}\right\}.
\end{equation*}

The group $\hat G$ acts on the spaces $\mathfrak P^+$ and $\mathfrak P^-$ by left matrix multiplication. Moreover, multiplication by $-1$ provides a $\hat G$-equivariant homeomorphism between $\mathfrak P^+$ and $\mathfrak P^-$.

The following proposition justifies the name of $\mathfrak P^\pm$:

\begin{prop}
    The group $\hat G$ acts transitively on $\mathfrak P_{\hat G}$. The stabilizer of the line $(1,1)^tA$ agrees with $\OO(G,\sigma)$. In particular, $\mathfrak P_{\hat G}$ is a model of the Riemannian symmetric space of $\hat G$.  
\end{prop}

\begin{proof}
    First notice that the line $\ell:=(1,1)^tA \in \mathfrak P_{\hat G}$ because $h((1,1)^t,(1,1)^t)=2\in B^\sigma_+$ and $(1,1)^t\in \Is(\omega)$.

    Let $u=(u_1,u_2)^t\in\Is(\omega)$ such that $h(u,u)\in B^\sigma_+$ which spans a line $\ell_u\in \mathfrak P_{\hat G}$. This implies $\sigma(u_1)u_2\in B^\sigma_+$. In particular, this means that $u_1,u_2$ are invertible, and we can assume $u_2=1$ and $u_1\in B^\sigma_+$. Let $g=u_1^{\frac{1}{2}}$, then the matrix $\hat g:=\begin{pmatrix}
        g & 0\\
        0  & \sigma(g)^{-1}
    \end{pmatrix}\in \hat G$ maps $(1,1)^t$ to $$(u_1^{\frac{1}{2}}, u_1^{-\frac{1}{2}})^t=(u_1,1)u_1^{-\frac{1}{2}}.$$
    That means $\hat g$ maps $\ell$ to $\ell_u$, i.e., $\hat G$ act transitively on $\mathfrak P_{\hat G}$.

    Let $\hat g=\begin{pmatrix}
        g & 0 \\
        0 & \sigma(g)^{-1}
    \end{pmatrix}$ for $g\in G$ stabilizes $\ell$. This is equivalent to $g=\sigma(g)^{-1}$, i.e., $g\in\OO(G,\sigma)$.
\end{proof}

\begin{rem}
    The closure of $\mathfrak P^\pm_{\hat G}$ inside $\PIs(\omega)$ is compact:
    $$\overline{\mathfrak P_{\hat G}^\pm}:=\{uA \mid \pm h(u,u)\in B^\sigma_{\geq 0}\}\cap\PIs(\omega).$$
    The group $\hat G$ acts on $\overline{\mathfrak P_{\hat G}^+}$. There is a unique compact orbit of this action that is shared by $\overline{\mathfrak P_{\hat G}^+}$ and $\overline{\mathfrak P_{\hat G}^-}$, namely: $\PIs_G(\omega)\cap\PIs(h)$.
\end{rem}

\subsection{Precompact model}\label{G-Precomp_Mod_R}

In this section, we identify the symmetric space of $\hat G$ as an open subset of $B^\sigma$ with a compact closure. 

We define the following domain $\mathfrak B_{\hat G}$:
$$\mathfrak B_{\hat G}\!:=\{c\in B^\sigma\mid 1-c^2\in B^{\sigma}_+\}.$$
This is an open domain in $B^\sigma$, therefore, it is naturally a complex manifold, and the tangent space $T_c\mathfrak B_{\hat G}=B^\sigma$ at every $c\in \mathfrak B$. Moreover, its closure
$$\overline{\mathfrak B_{\hat G}}\!:=\{c\in B^\sigma\mid 1-c^2\in B^{\sigma}_{\geq 0}\}$$
is contained in the compact domain $D(\bar G,\sigma)$ from Proposition~\ref{comp_disc}, i.e., it is precompact.

We define the following transformation $T:=\frac{1}{2}\begin{pmatrix}
    1 & 1 \\
    -1 & 1
\end{pmatrix}\in \Sp_2(G,\sigma)$. The group $T\hat G T^{-1}$ acts on $\mathfrak B_{\hat G}$ transitively by M\"obius transformations. The action is transitive and the stabilizer of the point $0\in\mathfrak B_{\hat G}$ agrees with the group $T\OO(G,\sigma)T^{-1}$. In particular, $\mathfrak B_{\hat G}$ is a model of the Riemannian symmetric space of $\hat G$. 

Furthermore, this action extends to the closure $\overline{\mathfrak B_{\hat G}}$. There is a unique closed orbit of this action, which is:
$$\{c\in B^\sigma\mid 1-c^2=0\}.$$

\subsection{Half-space models}

The last models we discuss in this chapter are the upper half- and lower half-space models. Let $\mathfrak U_{\hat G}:=B^\sigma_+$ respectively $\mathfrak U_{\hat G}^-:=B^\sigma_-.$

The spaces $\mathfrak U_{\hat G}^\pm$ are open unbounded domains in $B^{\sigma}$. Therefore, the tangent spaces $T_z\mathfrak U_{\hat G}^\pm= B_\CC^{\sigma}$ for every $z\in \mathfrak U_{\hat G}^\pm$.

Note that $\mathfrak U_{\hat G}^-$ and $\mathfrak U_{\hat G}^+$ are related by multiplication by $-1$. Thus, in what follows we will only discuss $\mathfrak U_{\hat G}=\mathfrak U_{\hat G}^+$. The corresponding statements for $\mathfrak U^-$ can be easily reformulated.

The group $\hat G$ acts on $\mathfrak U_{\hat G}$ by M\"obius transformations. This is equivalent to saying that $G$ acts by $\sigma$-congruence on $\mathfrak U_{\hat G}$, i.e., for $g\in G$ and $u\in \mathfrak U_{\hat G}$, $u\mapsto g u \sigma(g)$. This action is transitive and the stabilizer of the point $1\in\mathfrak U_{\hat G}$ is $\OO(G,\sigma)$. This means, in particular, that $\mathfrak U_{\hat G}$ is a model of the symmetric space of $G$.

\begin{rem}
    The space of complex almost orthogonal structures $\mathfrak C_{\hat G}$ naturally embeds into the space of complex structures $\mathfrak C$. However, in the last three models, to embed the models of the symmetric space of~$\hat G$ into the corresponding model of the symmetric space of $\Sp_2(G,\sigma)$, one needs to multiply corresponding elements by~$i$. Since these models of the symmetric space of~$\hat G$ can be described as real manifolds, we chose to present them directly, without passing to complexifications. Nevertheless, to obtain the corresponding embeddings, this additional complexification is necessary.
\end{rem}

\section{Spin group as \texorpdfstring{$\Sp_2(G,\sigma)$}{Sp2(G,sigma)}}\label{Clifford}

If $(A,\sigma)$ is an involutive associative algebra, then $A$ becomes a Lie algebra under the commutator bracket. Moreover, since $A$ is closed under multiplication, it is automatically of Jordan type. This includes, in particular, matrix algebras over $\R$, $\CC$, $\HH$, and other (skew-)fields, which yield numerous classical examples: the groups $\Sp_{2n}(\R)$, $\Sp_{2n}(\CC)$, $\mathrm{U}(n,n)$, $\GL_{2n}(\CC)$, $\SO^*(4n)$, $\OO(4n)$, and $\Sp(n,n)$ can all be realized as $\Sp_2(A,\sigma)$ for suitable choices of $(A,\sigma)$. We refer to~\cite{ABRRW} for details on this ``full-algebra'' case $B=A$.

The purpose of this section is to go beyond this situation and exhibit examples of proper Lie subalgebras $B\subsetneq A$. Clifford algebras over real vector spaces provide such examples. We focus on our principal case of interest -- the spin group $\Spin_0(m,n)$ -- and describe its realization as $\Sp_2(G,\sigma)$. When $m=2$, we further construct explicit models of the associated Hermitian symmetric space of $\Spin_0(2,n)$.

\subsection{Clifford algebra}\label{ex:Clifford}
    
    Let $V$ be a real vector space of dimension $m+n>0$ with a symmetric non-degenerate bilinear form $b$ of signature $(m,n)$. We denote by $\Cl(b)$ the Clifford algebra generated by $(V,b)$. We remind, $\Cl(b)$ is a unital algebra generated by all elements of $V$ subject to the relation $v^2=b(v,v)$ for $v\in V$. From this relation follows that for $v,w\in V$, $vw+wv=2b(v,w)$.

    The Clifford algebra $\Cl(b)$ contains a subalgebra $\Cl_{even}(b)$ that is generated by elements $\{vw\mid v,w\in V\}$. It is called \defin{even Clifford algebra}.

    We fix a linear involutive isometry $\sigma\colon V\to V$ with respect to $b$. We denote the $+1$-eigenspace of $\sigma$ by $V^\sigma$, and the $-1$-eigenspace of $\sigma$ by $V^{-\sigma}$. Then $V^{\pm\sigma}=(V^{\mp\sigma})^{\perp_b}$, and $V=V^\sigma\oplus V^{-\sigma}$ is a direct $b$-orthogonal decomposition of $V$. The map $\sigma$ extends uniquely to an anti-involution on $\Cl(b)$.

    An involutive isometry $\sigma$ is called \defin{compatible} with $b$ if $b|_{V^\sigma}$ is positive definite and $b|_{V^{-\sigma}}$ is negative definite.

    We fix an orthonormal basis $(f_1,\dots,f_m,e_1,\dots,e_n)$ of $V$ such that $b(e_i,e_j)=\delta_{ij}$, $b(f_i,f_j)=-\delta_{ij}$, $b(f_i,e_j)=0$ and, furthermore, there are $0\leq k\leq m$, $0\leq l\leq n$ such that 
    $$V^{\sigma}=\Span_\R(f_1,\dots,f_k,e_1,\dots,e_l),$$ $$V^{-\sigma}=\Span_\R(f_{k+1},\dots,f_m,e_{l+1},\dots,e_n).$$This provides an identification of $(V,b)$ with $\R^{m,n}$, which is the vector space $\mathbb R^{m+n}$ equipped with the standard symmetric non-degenerate bilinear form of signature $(m,n)$. The Clifford algebra corresponding to $\R^{m,n}$ is denoted by $\Cl(m,n)$.

    We consider the following vector subspace of $\Cl_{even}(b)$:
    $$B(m,n)=B(b)=\Span_\R(vw \mid v,w\in V).$$
    A direct computation shows that this is a Lie subalgebra by $\Cl(m,n)$ with respect to the Lie bracket $[x,y]:=xy-yx$ that is closed under $\sigma$. Further,
    \begin{align}\label{eq:B^{-sigma}}
    \begin{aligned}
    B^{\sigma}(m,n)& =B^\sigma(b)=\R\oplus V^{-\sigma}V^\sigma=\Span_\R(1, xy \mid x\in V^{-\sigma}, y\in V^{\sigma}).\\
    B^{-\sigma}(m,n)& =B^{-\sigma}(b)=\Span_\R( xx',yy'\mid x,x'\in V^\sigma, y,y'\in V^{-\sigma}, b(x,x')=b(y,y')=0)\\
    & = \Span_\R(xx'\mid x,x'\in V^\sigma, b(x,x')=0)\oplus \Span_\R(yy'\mid y,y'\in V^{-\sigma}, b(y,y')=0)\\
    &=B^{-\sigma}(V^\sigma)\oplus B^{-\sigma}(V^{-\sigma}). 
    \end{aligned}
    \end{align}
    Notice that $B^{\sigma}(V^\sigma)=B^{\sigma}(V^{-\sigma})=\R$.

\begin{prop}\label{semi-Herm1}
The Lie algebra $(B(m,n),\sigma)$ is of Jordan type if and only if $k+l\leq 1$ or $k+l\leq m+n-1$.
\end{prop}

\begin{proof} The condition is equivalent to: $\dim(V^\sigma)\leq 1$ or $\dim(V^{-\sigma})\leq 1$.

If $\dim(V^\sigma)=0$ or  $\dim(V^{-\sigma})=0$, then $B^{\sigma}(V)=\mathbb R$ is of Jordan type.

If $\dim(V^\sigma)=1$ and $V^\sigma=\Span_\R(x)$, then for $y,y'\in V^{-\sigma}$, $xy,xy'\in B^\sigma(b)$, and we have $xyxy'=-x^2yy'=-b(x,x)e_ie_j\in B(b)$. Therefore, $B(b)$ is of Jordan type.

The case  $\dim(V^{-\sigma})=1$ can be proven analogously. 
 
Let $\dim(V^\sigma)\geq 2$ and $\dim(V^{-\sigma})\geq 2$. Let $x,x'\in V^\sigma$ and $y,y'\in V^{-\sigma}$ pairwisely orthogonal vectors. Then $xy,x'y'\in B^\sigma(b)$. However, $xyx'y'\notin B(b)$.
\end{proof}

\begin{rem}
    Notice that powers of elements of $B^{-\sigma}(b)$ in general do not belong to $B$: let $e_1,e_2,e_2,e_4\in V^\sigma$ be pairwisely orthogonal, then the element $e=e_1e_2+e_3e_4\in B^{-\sigma}(b)$. However, 
    $$e^2=(e_1e_2)^2+(e_3e_4)^2+e_1e_2e_3e_4+e_3e_4e_1e_2=-2+2e_1e_2e_3e_4\notin B.$$    
\end{rem}

\begin{rem}\label{rem:nonHerm.complexification} If a real Lie algebra $(B,\sigma)$ is of Jordan type, than its complexification $(B_\mathbb C,\sigma)$ is also of Jordan type. However, $(B_\mathbb C,\bar\sigma)$ is in general not of Jordan type. Let us illustrate this phenomenon.

We consider the real Lie algebra $B=B(m,n)$ with $\sigma$ as above such that $\dim V^{-\sigma}=1$ where $V^{-\sigma}=\Span_\R(f_m)$. We take for simplicity $m=1$ and $n\geq 3$, but the argument below remains true also for any $m$.

We consider the complexification $V_\mathbb C$ of $V$. Then
\begin{align*}
    V^{-\sigma}_\mathbb C& =\Span_\mathbb C(f_1),\\
    B_\CC& =B(1,n)\otimes_\R\CC=\Span_\CC(1,e_le_j,f_1e_k\mid j,k,l\in\{1,\dots n\}),\\
    B_\CC^{\sigma}& =\Span_\CC(1,f_1e_k\mid k\in\{1,\dots n\}).
\end{align*}
Therefore, $(B_\mathbb C,\sigma)$ is a complex Lie algebra of Jordan type. However,
\begin{align*}
    B_\CC^{\bar\sigma}&=\Span_\R(1,f_1e_k,ie_le_j\mid j,k,l\in\{1,\dots n\})
\end{align*}
and, therefore, $(B_\CC,\bar\sigma)$ is not of Jordan type because $if_1e_1e_2e_3\notin B_\mathbb C$.
\end{rem}
 
\begin{prop}\label{prop:Cliff.weaklyHerm}
    The Lie algebra $(B(b),\sigma)=(B(m,n),\sigma)$ with $m\leq n$ is weakly Hermitian if and only if either $\dim V^{-\sigma}=0$ or $m=1$ and $\sigma$ is compatible with $b$.
\end{prop}

\begin{proof} If $\dim(V^{-\sigma})=0$, then $B^\sigma(b)=\R$ and thus $(B(b),\sigma)$ is weakly Hermitian. Since the property to be weakly Hermitian implies the property to be of Jordan type, by Proposition~\ref{semi-Herm1}, we can now assume that $V^{-\sigma}=f_m\R$.

To prove Proposition~\ref{prop:Cliff.weaklyHerm}, we the following Lemma:

\begin{lem}\label{lem:weak.Herm.Clifford}
    Let $V^{-\sigma}=f_m\R$. If the Lie algebra $(B(b),\sigma)$ is weakly Hermitian then $\sigma$ is compatible with $b$ and $b|_{V^{\sigma}}$ is positive definite, i.e.,  $b$ has signature $(1,n)$.
\end{lem}

\begin{proof}
    By contradiction, let $x\in V^{\sigma}$ be an element such that $b(x,x)\leq 0$. Then $f_mx\in B^\sigma(b)$ and $(f_mx)^2=b(x,x)\in B^\sigma_+(b)$. As $B^\sigma(b)$ contains no nilpotent elements, $b(x,x)< 0$. But $1^2=1\in B^\sigma_+(b)$. Therefore, $B^\sigma_+(b)$ contains the line $\R$, i.e., $B(b)$ is not weakly Hermitian. Thus, if the Lie algebra $(B(b),\sigma)$ is weakly Hermitian then $\sigma$ is compatible with $b$ and $b|_{V^{\sigma}}$ is positive definite.
\end{proof}

To finish the proof of Proposition~\ref{prop:Cliff.weaklyHerm}, it is enough to prove the converse of Lemma~\ref{lem:weak.Herm.Clifford}, i.e., that $B(1,n)$ with a compatible anti-involution $\sigma$ is weakly Hermitian.

Let $\sigma$ be an involutive isometry on $(V,b)=\R^{1,n}$. For an element $e\in V^\sigma$, we denote its Euclidean norm as:
$$\|e\|:=\sqrt{b(e,e)}\geq 0.$$

Obviously, every $x\in B^{\sigma}(1,n)$ can be uniquely written as follows $x=a_0+rf_1e$ where $a_0\in\R$, $r\geq 0$ and $e\in V^\sigma$ with $\|e\|=1$ such that $x=a_0+rf_1e$.

\begin{lem}\label{lem:Clif_cone}
The space
$$B^{\sigma}_{\geq 0}(1,n)=\left\{t+uf_1e\mid t\geq 0,\,0\leq u\leq t,\,e\in V^\sigma,\, \|e\|=1\right\}.$$
is a closed proper convex cone.
\end{lem}

\begin{proof}
Let $x:=a_0+rf_1e\in B^{\sigma}(1,n)$ where $a_0\in\R$, $r\geq 0$ and $e\in V^\sigma$ with $\|e\|=1$. Then $x^2=a_0^2+r^2+2a_0rf_1e.$
Let $t:=a_0^2+r^2\geq 0$ and $u:=2a_0r=2\sgn(a_0)\sqrt{t-r^2}r$. For fixed $t\geq 0$, $u=u(r)$ takes all values between $-t$ and $t$. So we get
\begin{align*}
B^{\sigma}_{\geq 0}(1,n)& =\left\{t+uf_1e\mid t\geq 0,\, |u|\leq t, e\in V^\sigma,\, \|e\|=1\right\}\\
& =\left\{t+uf_1e\mid t\geq 0,\, 0\leq u\leq t, e\in V^\sigma,\, \|e\|=1\right\}.
\end{align*}
This is a convex cone. Indeed, take $x=t+uf_1e$, $x'=t'+u'f_1e'$ from $B^{\sigma}_{\geq 0}(1,n)$. Then $$x+x'=(t+t')+f_1(ue+u'e')=(t+t')+\tilde vf_1\tilde e$$
where
$v=\|ue+u'e'\|$, $\tilde e=\frac{ue+u'e'}{v}.$
Using the triangle inequality we get:
$$0\leq v=\|v\tilde e\|=\|ue+u'e'\|\leq\|ue\|+\|u'e'\|=u+u'<t+t'.$$
Therefore, $x+x'\in B^{\sigma}_{\geq 0}(1,n)$.

It is also a proper cone because for
$$x\in B^{\sigma}_{\geq 0}(1,n)\cap(-B^{\sigma}_{\geq 0}(1,n)),$$
$t=0$ and so $u=0$.
\end{proof}

\begin{lem}\label{lem:Herm1}
The space $B^\sigma(1,n)$ contains no nilpotent elements.
\end{lem}

\begin{proof} We show that for $b\in B^\sigma(1,n)$, $b=0$ if and only if $b^2=0$. Let $b=a_0+a_1fe\in B^\sigma(1,n)$ with $b(e,e)=1$. Then $b^2=a_0^2+a_1^2+2a_0a_1fe=0$, therefore $a_0=a_1=0$, i.e., $b=0$.
\end{proof}

Thus, we conclude that the Lie algebra $B(1,n)$ is weakly Hermitian.
\end{proof}

\subsection{Spectral decomposition in the Clifford algebra}

In this subsection, we find what do the Jordan frames in the formally real Jordan algebra $B^{\sigma}(1,n)$ look like.

\begin{prop}
Every nontrivial idempotent $c\in B^{\sigma}(1,n)$ is of the following form: $c=\frac{1+e}{2}$ for some $e\in V^\sigma$ such that $\|e\|=1$.
\end{prop}

\begin{proof}
Let $c=x+yf_1e\in B^{\sigma}(1,n)$ be an idempotent. This means, in particular, that $c\in B_{\geq 0}^{\sigma}(1,n)$, i.e., $0 \leq y \leq x$ and $\|e\|=1$. Then
$$c^2=(x^2+y^2)+2xyf_1e=x+yf_1e=c.$$
Thus $2xy=y$. If $y=0$, then $c=1$ is trivial idempotent. If $y\neq 0$, then $x=y=\frac{1}{2}$.
\end{proof}

\begin{teo}\label{JB-clif}
Every Jordan frame in $B^{\sigma}(1,n)$ is of the following form: $(c_1,c_2)$ where $c_1=\frac{1+e}{2}$, $c_2=\frac{1-e}{2}$ for some $e\in\Span_\R(e_1,\dots,e_n)$, $b(e,e)=1$.
\end{teo}

\begin{proof}
Let $c_1=\frac{1+f_1e_1}{2}$, $c_2=\frac{1+f_1e_2}{2}$ be two orthogonal idempotents with $\|e_1\|=\|e_2\|=1$. Then:
\begin{align*}
    0=c_1\circ c_2&=\frac{1+f_1e_1}{2}\circ\frac{1+f_1e_2}{2}=\frac{(1+f_1e_1)(1+f_1e_2)+(1+f_1e_2)(1+f_1e_1))}{8}\\
    &=\frac{1+f_1e_1+f_1e_2+f_1e_1f_1e_2+1+f_1e_1+f_1e_2+f_1e_2f_1e_1}{8}\\
    &=\frac{2+2b(e_1,e_2)+2f_1(e_1+e_2)}{8}.
\end{align*}
Therefore, $e_1=-e_2$. Moreover, $c_1+c_2=1$, i.e., $(c_1,c_2)$ is a complete system of orthogonal idempotents. Thus, every complete system of idempotents has at most two elements. In particular, all Jordan frames are of the form $(c_1,c_2)$ where $c_1=\frac{1+e}{2}$, $c_2=\frac{1-e}{2}$ as above.
\end{proof}

\subsection{Clifford group and its Lie algebra}

In this section, we describe a Lie group of which the Lie algebra is $B(m,n)$.

The group of all invertible elements $\Cl(m,n)^\times$ of $\Cl(m,n)$ acts on $\Cl(m,n)$ in the following way
\begin{equation}\label{eq:Clifford-tau}
    \begin{matrix}
    \tau\colon & \Cl(m,n)^\times\times\Cl(m,n) & \to & \Cl(m,n) \\
     & (x,y) & \mapsto & \alpha(x)y x^{-1}
\end{matrix}    
\end{equation}
where $\alpha$ is the involution on $\Cl(m,n)$ induced by the automorphism $v\mapsto -v$ on $\R^{m,n}$.

\begin{df}\label{df:ClGroup}
The (even) \defin{Clifford group} of signature $(m,n)$ is the group of all invertible elements of $\Cl_{even}(m,n)$ that stabilize $(V,b)=\R^{m,n}$ under the action $\tau$, i.e.
$$\ClGr(m,n):=\ClGr(b):=\{x\in\Cl_{even}^\times(m,n)\mid \tau(x)v\in V\text{ for all }v\in V\}.$$

\noindent We denote by $\ClGr_0(m,n)$ the connected component of $1$ in $\ClGr(m,n)$. We call it also Clifford group.
\end{df}

\begin{rem}
The Lie algebra of $\ClGr(m,n)$ can be described as follows:
$$\clgr(m,n):=\clgr(b):=\{x\in\Cl(m,n)\mid \forall v\in V:\alpha(x)v-vx\in V\}.$$
\end{rem}

We recall the following well known properties of the Clifford group (for more details see~\cite{Vaz19}). Definition~\ref{df:ClGroup} implies that $\ClGr(m,n)$ acts on $V$ by $\tau$ preserving $b$ so that the following sequences are exact:
$$1\to \R^\times\xrightarrow{\subset}\ClGr(m,n)\xrightarrow{\tau}\SO(m,n)\to 1. $$
$$1\to \R_+\xrightarrow{\subset}\ClGr_0(m,n)\xrightarrow{\tau}\SO_0(m,n)\to 1. $$
In particular,
$$\dim_{\R}\ClGr(m,n)=\dim_{\R}\R+\dim_{\R}\SO(m,n)=1+\frac{(m+n)(m+n-1)}{2}.$$

\begin{fact}
The following map is well defined:
$$\begin{matrix}
N\colon & \ClGr(m,n) & \to & \R^\times\\
& x & \mapsto & x^tx
\end{matrix}$$
where $(\cdot)^t$ is the anti-involution on $\Cl(m,n)$ induced by the trivial automorphism on $V$.
\end{fact}

We remind the definition of the spin group:
\begin{df} The \defin{spin group} of signature $(m,n)$ is the following subgroup of the Clifford group:
$$\Spin(m,n):=\{x\in \ClGr(m,n)\mid N(x)=1\},$$
$$\Spin_0(m,n):=\{x\in \ClGr_0(m,n)\mid N(x)=1\}.$$
\end{df}

\begin{rem}
If $m>0$ and $n>0$ then $\Spin(m,n)$ has two connected components. If $m=0$ or $n=0$, then $\Spin(m,n)$ is connected, i.e., $\Spin(m,n)=\Spin_0(m,n)$.
\end{rem}

\begin{cor}\begin{itemize}
\item For $x\in\ClGr_0(m,n)$, $N(x)>0$;
\item The map $$x\mapsto\frac{x}{\sqrt{N(x)}}$$ maps $\ClGr_0(m,n)$ surjectively to $\Spin_0(m,n)$;
\item The map $$x\mapsto\frac{x}{\sqrt{|N(x)|}}$$ maps $\ClGr(m,n)$ surjectively to $\Spin(m,n)$;
\end{itemize}
\end{cor}

\begin{prop} The Lie algebra $\clgr(m,n)$ agrees with $B(m,n)$.
\end{prop}

\begin{proof}
A direct calculation shows that $B(m,n)\subseteq \clgr(m,n)$.
Moreover,
\begin{align*}
    \dim_{\R}B(m,n)&=1+\frac{m^2-m}{2}+\frac{n^2-n}{2}+mn=1+\frac{m^2-m+n^2-n+2mn}{2}\\
    &=1+\frac{(m+n)^2-(m+n)}{2}=\dim_{\R}\clgr(m,n).
\end{align*}
So we obtain that $B(m,n)=\clgr(m,n)$.
\end{proof}

\begin{prop}\label{U-Clif-Comp}
The maximal compact subgroup $U(\ClGr_0(m,n),\sigma)$ of $\ClGr_0(m,n)$ agrees with $\Spin(m)\times\Spin(n)$.
\end{prop}

\begin{proof} Both groups are connected. So it is enough to show that their Lie algebras agree. We refer to~\eqref{eq:B^{-sigma}}:
\begin{align*}
    \Lie(U(\ClGr_0(m,n),\sigma))&=B^{-\sigma}(m,n)\\
    &=B^{-\sigma}(m,0)\oplus B^{-\sigma}(0,n) \\
    &=\spin(m)\oplus\spin(n).\qedhere
\end{align*}
\end{proof}

\begin{teo} Let $\sigma$ be a compatible involutive isometry on $\R^{1,n}$. The group $U(\ClGr_0(1,n),\sigma)$ acts transitively on the set of Jordan frames. In particular, $B(1,n)$ is Hermitian.
\end{teo}

\begin{proof}
Let $\left(c_1=\frac{1+f_1e}{2},c_2=\frac{1-f_1e}{2}\right)$ and $\left(c_1'=\frac{1+f_1e'}{2}, c_2'=\frac{1-f_1e'}{2}\right)$ be two Jordan frames. Since $b(e,e)=b(e',e')=1$, there exists an orthogonal transformation $u\in \SO(n)$ such that $u(e)=e'$. Take some preimage $v$ of $u$ in $\Spin(n)=U(\ClGr_0(1,n),\sigma)$. Then
$$\sigma(v)c_1v=v^{-1}c_iv=\frac{1+(v^{-1}f_1v)(v^{-1}ev)}{2}=\frac{1+u(f_1)u(e)}{2}=\frac{1+f_1e'}{2}=c_1'.$$
Similarly, $\sigma(v)c_2v=c_2$.

Finally, by Proposition~\ref{prop:Cliff.weaklyHerm}, the Lie algebra $B(1,n)$ is weakly Hermitian, and, by Proposition~\ref{U-Clif-Comp}, the group $U(\ClGr(1,n),\sigma)$ is compact.
\end{proof}

\begin{rem}
    As we have seen, the Lie algebra $(B(b),\sigma)$ such that $V^\sigma=V$ or $V^{-\sigma}=V$ is always weakly Hermitian. In this case, $U(\ClGr_0(b),\sigma)=\Spin_0(m,n)$ where $(m,n)$ is the signature of the form $b$, which is not compact unless $m=0$ or $n=0$. Therefore, $(B(b),\sigma)$ is Hermitian if and only of $m=0$ or $n=0$, i.e., if and only if $\sigma$ is compatible with $b$.
\end{rem}

\begin{rem}
    Let $\sigma$ be a compatible involutive isometry on $\R^{1,n}$. Notice that the space $B^\sigma(1,n)^\times$ is (in general) strictly contained in $\ClGr^\sigma(1,n)$. For example, let $n\geq 4$ then the elements $f_1e_1$, $f_1e_2$, $f_1e_3$, $f_1e_4$ are in $B^\sigma(1,n)^\times\subseteq \ClGr^\sigma(1,n)$. However, their product is not contained in  $B^\sigma(1,n)$ because 
    $$f_1e_1f_1e_2f_1e_3f_1e_4=e_1e_2e_3e_4\notin B.$$
    But it is an element of $\ClGr^\sigma(1,n)$ and, more precisely, of the $\sigma$-symmetric part of its unipotent subgroup $U(\ClGr(1,n),\sigma)\cap \ClGr^\sigma(1,n)$.
\end{rem}

\subsection{Lie algebra \texorpdfstring{$\spin(m,n)$}{spin(m,n)} as \texorpdfstring{$\spp_2(B,\sigma)$}{sp2(B,sigma)}}

In this section, we identify $\spp_2(B(m,n),\sigma)$ with $\spin(m+1,n+1)$. We recall the definition of $\spin(m,n)$:
$$\spin(m,n)=\Span_\R(e_ie_j, f_ke_i, f_kf_l\mid i,j\in\{1,\dots,n\}, k,l\in\{1,\dots,m\})$$
where $\{f_i\}_{i=1}^{m}$, $\{e_i\}_{i=1}^n$ build an orthonormal basis of $(V,b)=\R^{m,n}$ as before. Note that $B(m,n)\cong\mathbb R\oplus\spin(m,n)$, where $\R$ is central. Let $\sigma$ be an involutive isometry on $\R^{m,n}$ such that $\sigma(f_m)=-f_m$ and $\sigma(f_i)=f_i$ for $1\leq i\leq m-1$, and $\sigma(e_i)=e_i$ for all $i$.

The following theorem is immediate:
\begin{teo}
The following map is an isomorphism of Lie algebras:
$$\begin{array}{rrcll}
\varphi\colon & \spp_2(B(m,n),\sigma) & \to & \spin(m+1,n+1) \\
& \diag(a,a) & \mapsto & a & \text{for }a\in B^{-\sigma}(m,n)\\
& \diag(f_me_i,-f_me_i) & \mapsto & f_me_i & \text{for }1\leq i\leq n\\
& \diag(f_mf_i,-f_mf_i) & \mapsto & f_mf_i & \text{for }1\leq i\leq m-1\\
& \diag(1,-1) & \mapsto & f_{m+1}e_{n+1} \\
& S(f_mv) & \mapsto & f_{m+1}v & \text{for }v\in V^\sigma\\
& A(f_mv) & \mapsto & ve_{n+1} & \text{for }v\in V^\sigma\\
& S(1) & \mapsto & -f_me_{n+1}\\
& A(1) & \mapsto & f_mf_{m+1}
\end{array}$$
where $S(x)=\begin{pmatrix} 0 & x \\ x & 0 \end{pmatrix}$, $A(x)=\begin{pmatrix} 0 & -x \\ x & 0 \end{pmatrix}$.
\end{teo}

\subsection{Group \texorpdfstring{$\Spin_0(m,n)$}{Spin0(m,n)} as \texorpdfstring{$\Sp_2(G,\sigma)$}{Sp2(G,sigma)}}

In this section, we want to identify the groups $\Spin_0(m+1,n+1)$ and $\Sp_2(\ClGr(m,n),\sigma)$.

First, note that a generic matrix in $\Sp_2(\ClGr(m,n),\sigma)$ can be decomposed in a unique way as follows:
\begin{equation}\label{eq:sp2-clifford-decomposition}
\begin{pmatrix}
1 & 0 \\
y & 1\
\end{pmatrix}
\begin{pmatrix}
\lambda & 0 \\
0 & \lambda^{-1}\
\end{pmatrix}
\begin{pmatrix}
x & 0 \\
0 & \sigma(x)^{-1}\
\end{pmatrix}
\begin{pmatrix}
1 & z \\
0 & 1\
\end{pmatrix}    
\end{equation}
where $x\in\Spin(m,n)$, $\lambda>0$, $y,z\in B^{\sigma}(m,n)$.

The embedding:
$$\R^{m,n}=\Span(f_1,\dots,f_m,e_1,\dots,e_n)\subset \R^{2,n+1}=\Span(f_1,\dots,f_{m+1},e_1,\dots,e_{n+1})$$
induces the embedding of spin groups $\iota \colon \Spin(m,n)\hookrightarrow\Spin(m+1,n+1)$. Moreover, we can embed the entire $\Spin(m,n)$ into $\Spin_0(m+1,n+1)$ using the following map:
$$\iota_0\colon\Spin(m,n)\to \Spin_0(m+1,n+1)$$
\begin{itemize}
\item if $x\in \Spin_0(m,n)$, then $\iota_0(x):=\iota(x)$;
\item if $x\in \Spin(m,n)\bs\Spin_0(m,n)$, then $\iota_0(x)=\iota(x)f_{m+1}e_{n+1}$.
\end{itemize}
The map $\iota_0$ is an injective group homomorphism.

\begin{teo}\label{Spin_as_Sp2}
The following map is an isomorphism of Lie groups:
$$\begin{array}{rcll}
\Phi\colon \Sp_2(\ClGr(m,n),\sigma) & \to & \Spin_0(m+1,n+1) & \\
 \diag(y,\sigma(y)^{-1}) & \mapsto & \iota_0(y) & \text{for }y\in\Spin(m,n)\\
 \diag(\lambda,\lambda^{-1}) & \mapsto & \frac{\lambda+\lambda^{-1}}{2}+f_{m+1}e_{n+1}\frac{\lambda-\lambda^{-1}}{2} & \text{for }\lambda>0\\
 R(f_mv) & \mapsto & 1+\frac{f_{m+1}+e_{n+1}}{2}v & \text{for }v\in V^\sigma \\
 L(f_mv) & \mapsto & 1+\frac{f_{m+1}-e_{n+1}}{2}v & \text{for }v\in V^\sigma\\
 R(1) & \mapsto & 1-\frac{f_{m+1}+e_{n+1}}{2}f_m\\
 L(1) & \mapsto & 1-\frac{f_{m+1}-e_{n+1}}{2}f_m
\end{array}$$
where $R(x)=\begin{pmatrix} 1 & x \\ 0 & 1 \end{pmatrix}$, $L(x)=\begin{pmatrix} 1 & 0 \\ x & 1 \end{pmatrix}$. The map $\varphi$ is the differential of $\Phi$ at $\Id$.
\end{teo}

\begin{proof} Because there is a neighborhood of $\Id\in \Sp_2(\ClGr(1,n),\sigma)$ consisting of generic elements and because of uniqueness of the decomposition~\eqref{eq:sp2-clifford-decomposition}, one can see that $\varphi$ is the differential of $\Phi$ at $\Id$. Because $\varphi$ is an isomorphism of Lie algebras, $\Phi$ is a group homomorphism and a $k:1$-covering map for some $k\geq 1$. Finally, since this map is bijective on the level of diagonal matrices, $k=1$. 
\end{proof}

\subsection{Models of the symmetric space of  \texorpdfstring{$\Spin_0(2,n)$}{Spin0(2,n)}}

In this section, we construct different models of the symmetric space of $\Spin_0(2,n+1)\cong\Sp_2(\ClGr(1,n),\sigma)$.

\subsubsection{Upper half-space model} We remind:
\begin{align*}
    B(1,n)&=\Span_\R(1,e_ie_j,f_1e_k\mid i,j,k\in\{1,\dots n\})\\
    B^{\sigma}(1,n)&=\Span_\R(1,f_1e_k\mid k\in\{1,\dots n\})=f_1V\\
    B^{\sigma}_{+}(1,n)&=\left\{t+uf_1e\mid t>0,\, u\in[0,t), e\in\Span_\R(e_1,\dots,e_n),\, \|e\|=1\right\}\\
    &=f_1V_-\ ,
\end{align*}
where $V_-=\{v\in V\mid b(v,v)<0,\; b(f_1,v)>0\}$. Thus,
$$\mathfrak U(\Spin_0(2,n+1))=\{x+yi\mid x\in B^{\sigma}(1,n),\;y\in B^{\sigma}_+(1,n)\}=f_1(V+V_-i).$$
Further,  $\mathfrak U(\ClGr(1,n))=B^{\sigma}_{+}(1,n)=f_1V_-$. It sits naturally inside $\mathfrak U(\Spin_0(2,n+1))$ as the imaginary part.

\subsubsection{Precompact model}
We take the complexification
$$B_\CC:=B_\CC(1,n)=\Span_\CC\{1,f_1e_i,e_ie_j \mid i,j\in\{1,\dots,n\}\}$$
with complex linear extension of $\sigma$ denoted also by $\sigma$ and the complex antilinear extension of $\sigma$ denoted by $\bar\sigma$. Then
$$B^\sigma_\CC=\Span_\CC\{1,f_1e_i\mid i\in\{1,\dots,n\}\}.$$
We consider the space, which is the precompact model of the symmetric space for $\Spin_0(2,n+1)$:
\begin{align*}
\mathfrak{B}(\Spin_0(2,n+1))&=\{b\in B_\CC^\sigma\mid 1-\bar bb\in \theta_\CC(B^\sigma_\CC)^\times\}\\
&=\{x+yf_1e\mid 1-(\bar x+\bar y f_1e)(x+yf_1e)\in \theta_\CC(B^\sigma_\CC)^\times,\; \|e\|=1\}\\
&=\{x+yf_1e\mid 1-(\bar xx+\bar yy)-(\bar xy+\bar yx)f_1e\in \theta_\CC(B^\sigma_\CC)^\times,\; \|e\|=1\}.
\end{align*}

We first find out what the set $\theta_\CC(B^\sigma_\CC)$ is.
Let $r_1\exp(i\phi_1)+r_1\exp(i\phi_1)f_1e\in B^\sigma_\CC$, $r_1,r_2\geq 0$, $\phi_1,\phi_2\in\R$, $\|e\|=1$ then
$$\theta_\CC(r_1\exp(i\phi_1)+r_1\exp(i\phi_1)f_1e)=(r_1^2+r_2^2)+ 2r_1r_2\cos(\phi_2-\phi_1)f_1e.$$
We denote $r_1^2+r_2^2=:r$. then
$$\theta_\CC(r_1\exp(i\phi_1)+r_2\exp(i\phi_2)f_1e)=r+2r_1\sqrt{r-r_1^2}\cos(\phi_2-\phi_1)f_1e$$
For fixed $r\geq 0$, the expression $2r_1\sqrt{r-r_1^2}\cos(\phi_2-\phi_1)$ can take every value in the interval $[-r,r]$, i.e.
$$\theta_\CC(B^\sigma_\CC)=\{r+pf_1e\mid r\geq 0,\;p\in [-r,r],\;\|e\|=1\}=B^{\sigma}_{\geq 0}.$$
Analogously,
$$\theta_\CC(B^\sigma_\CC)^\times=B^{\sigma}_+.$$
Therefore,
$$\mathfrak{B}(\Spin_0(2,n+1))=\{x+yf_1e\in B_\CC^\sigma\mid 1-(\bar xx+\bar yy)-(\bar xy+\bar yx)f_1e\in B^{\sigma}_+,\;\|e\|=1\}$$
is the precompact model of the symmetric space for $\Spin_0(2,n+1)$. Further,  $\mathfrak B (\ClGr(1,n))=\{x+yf_1e\in B^\sigma\mid 1-(x^2+y^2)-2xyf_1e\in B^\sigma_+\}$. Is sits naturally inside $\mathfrak B(\Spin_0(2,n+1))$ as the imaginary part.

\subsubsection{Projective space model}
We take the upper half-space model:
$$\mathfrak U(\Spin_0(2,n+1))=\{x+yi\mid x\in B^{\sigma}(1,n),\;y\in B^{\sigma}_+(1,n)\}.$$
We know that the map $z\mapsto (z,1)^tA_\CC$ is a homeomorphism between the upper half-space models and the projective space model. So we obtain:
$$\mathfrak P(\Spin_0(2,n))=\{(x+yi,1)^tA_\CC\mid x\in B^{\sigma}(1,n),\;y\in B^{\sigma}_+(1,n)\}=$$
$$=\left\{\begin{pmatrix}(x_1+y_1i)+f_1(ex_2+e'y_2)i\\1\end{pmatrix}A_\CC\left|
\begin{array}{l}
x_1\in\R,\,y_1\in\R_+,\,x_2\in\R_{\geq 0},\\
y_2\in[0,x_2),\,\|e\|=\|e'\|=1
\end{array}
\right.\right\}.$$

We can also interpret the projective space model for $\Spin_0(2,n+1)$ in terms of lines in $\CC^{n+3}$. First, we note that the stabilizer of the line $(i,1)^tA_\CC\subset A_\CC^2$ corresponds under the map $\Phi$ (cf.~Theorem~\ref{Spin_as_Sp2}) to the stabilizer of the line $(f_2+f_1i)\CC\subset\CC^{n+3}$ where $\Spin_0(2,n+1)$ acts on $\CC^{n+3}$ by the (complexified) map $\tau$ (cf.~\eqref{eq:Clifford-tau}). So we can take the following injective map:
$$\begin{array}{rcl}
F\colon \mathfrak P(\Spin_0(2,n+1)) & \to & \{l\in\CC P^{n+2}\mid l=v\CC,\,v\in\CC^{n+3},\,b_\CC(v,v)=0\}\\
 g(i,1)^tA_\CC & \mapsto & \tau(\Phi(g))(f_2+f_1i)\CC=\Phi(g)(f_2+f_1i)\Phi(g)^{-1}\CC
\end{array}$$
where $g\in\Sp_2(G,\sigma)$ and $b_\CC$ the complex bilinear extension of $b$.

Since $\Spin_0(2,n+1)$ acts on $\R^{n+3}$ preserving $b$, it acts on $\CC^{n+3}$ preserving $b_\CC$ and the $\CC$-sesquilinear extension $\tilde b$ of $b$, i.e.
$$\tilde b(v_1+v_2i,w_1+w_2i):=b(v_1,w_1)+b(v_2,w_2)+(b(v_1,w_2)-b(v_2,w_1))i$$
for $v,w\in\R^{n+3}$. Notice that the Hermitian form $\tilde b$ on $\CC^{n+3}$ has signature $(2,n+1)$. Therefore, a direct computation shows that
$F$ maps injectively $\mathfrak P(\Spin_0(2,n+1))$ onto
$$\mathfrak P'(\Spin_0(2,n+1)):=\left\{l\in \CC P^{n+2}\midwd
l=v\CC,\,v\in\CC^{n+3},\,b_\CC(v,v)=0,\,\tilde b(v,v)<0
\right\}.$$
Let $v=v_1+v_2i\in\CC^{n+3}$ be a vector such that $b_\CC(v,v)=0$, $\tilde b(v,v)=-2$. Then $b(v_1,v_1)=b(v_2,v_2)=-1$, $b(v_1,v_2)=0$. There exists an $\SO(2,n+1)$-transformation that maps $(f_2,f_1)$ to $(v_1,v_2)$. Therefore, $\Spin_0(2,n+1)$ acts transitively on $\mathfrak P'(\Spin_0(2,n+1))$, so it is a model of the symmetric space of $\Spin_0(2,n+1)$.

The projective space model for the symmetric space of $\ClGr(1,n)$ can be similarly interpreted as follows:
$$\mathfrak P'(\ClGr(1,n)):=\left\{l\in \CC P^{n+2}\midwd\begin{array}
 ll=(w+vi)\CC,\, v\in\Span_\R(e_1,\dots,e_n,f_1)\\
 w\in\Span_\R(e_{n+1},f_2),
b_\CC(v,v)=0,\,\tilde b(v,v)<0
\end{array}
\right\}.$$
This space is naturally embedded into $\mathfrak P'(\Spin_0(2,n+1))$. Notice that the subset of $\mathfrak P'(\ClGr(1,n))$ with $w=f_2$ provides a projective space model for the Riemannian symmetric of $\Spin(1,n)$, which can be also reinterpreted as follows:
$$\mathfrak P'(\Spin(1,n)):=\left\{l\in \R P^{n}\midwd 
 l=v\R,\, v\in\Span_\R(e_1,\dots,e_n,f_1),\; b(v,v)<0
\right\}.$$
This is the classical projective space model of the hyperbolic $n$-space.

\subsubsection{Space of complex structures}
We consider the space of complex structures for $\Sp_2(G,\sigma)\cong\Spin_0(2,n+1)$:
$$\mathfrak C(\Sp_2(G,\sigma)):=\left\{J\text{ complex structure on $A^2$}\left|\begin{array}{l}
J(\Is_G(\omega))=\Is_G(\omega),\\h_J\text{ is a $(G,\sigma)$-inner product}
\end{array}\right.\right\}$$
where $h_J(x,y)=\omega(J(x),y)$. This model admits the following equivalent interpretation:

\vspace{2mm}

\noindent Notice that $\mathfrak C(\Spin_0(2,n+1))\subseteq\Sp_2(G,\sigma)$ because the standard complex structure $J_0\in\Sp_2(G,\sigma)$ and $\Sp_2(G,\sigma)$ acts on $\mathfrak C(\Spin_0(2,n+1))$ transitively by conjugation.

\noindent If we take the standard complex structure $J_0=\Om$, then $\Phi(J_0)=-f_1f_2$ where $\Phi$ is the isomorphism from $\Sp_2(G,\sigma)$ to $\Spin_0(2,n+1)$ defined in Theorem~\ref{Spin_as_Sp2}. For every $v\in\R^{n+3}$ there exist unique elements $e\in\Span_\R(e_1,\dots,e_n)$, $f\in\Span_\R(f_1,f_2)$ such that $v=e+f$. Then
$$\Phi(J_0)v\Phi(J_0)^{-1}=e-f.$$

\noindent Let $J$ be another element of $\mathfrak C$. There exists $g\in\Sp_2(G,\sigma)$ such that $J=g^{-1}J_0g$. Let $v=\Phi(g)^{-1}(e+f)\Phi(g)$, $e,f$ as above,
$$\Phi(J)v\Phi(J)^{-1}=\Phi(g)^{-1}\Phi(J_0)(e+f)\Phi(J_0)^{-1}\Phi(g)=\Phi(g)^{-1}(e-f)\Phi(g).$$
Since $\Spin_0(2,n+1)$ acts on $(\R^{n+3},b)$ preserving $b$, the restriction of $b$ to the linear subspace $\Phi(g)^{-1}\Span_\R(e_1,\dots,e_{n+1})\Phi(g)$ is positive definite and the restriction of $b$ to the linear subspace $\Phi(g)^{-1}\Span_\R(f_1,f_2)\Phi(g)$ is negative definite.

\noindent Consider the following space of $b$-orthogonal splittings of $\R^{n+3}$:
$$\mathfrak C'(\Spin_0(2,n+1)):=\left\{(V_+,V_-)\midwd \begin{array}{l}
\R^{n+3}=V_+\oplus V_-\\
b|_{V_+}\text{ is positive definite}\\
b|_{V_-}\text{ is negative definite}\\
V_+\perp_b V_-
\end{array}\right\}.$$
We have a surjective map $F\colon\Sp_2(G,\sigma)\to\mathcal D$,
$$F(g):=(\Phi(g)^{-1}\Span_\R(e_1,\dots,e_{n+1})\Phi(g),\Phi(g)^{-1}\Span_\R(f_1,f_2)\Phi(g)).$$
Notice,
$$F^{-1}(\Span_\R(e_1,\dots,e_n),\Span_\R(f_1,f_2))=\Spin(2)\times\Spin(n+1).$$
Therefore, $\mathfrak C'(\Spin_0(2,n+1))$ is isomorphic to $\Spin_0(2,n+1)/(\Spin(2)\times\Spin(n+1))$, i.e., $\mathfrak C'(\Sp_2(G,\sigma))$ is the model of the symmetric space of $\Spin_0(2,n+1)$. 

To identify the spaces $\mathfrak C'(\ClGr(1,n))$ and $\mathfrak C'(\Spin(1,n))$, we notice that $\iota_0(\Spin(1,n))$ inside $\Spin_0(2,n+1)$ stabilizes $\Span(e_{n+1})$ and $\Span(f_2)$. Therefore, $\mathfrak C'(\Spin(1,n))$ can be identified with the space of the following splittings of $\R^{n+3}$:
$$\mathfrak C'(\Spin(1,n)):=\left\{(V_+,V_-)\midwd \begin{array}{l}
\R^{n+3}=V_+\oplus V_-\oplus\Span(e_{n+1})\oplus\Span(f_{2})\\
b|_{V_+}\text{ is positive definite, }V_+\subset W\\
b|_{V_-}\text{ is negative definite, }V_-\subset W\\
V_+\perp_b V_-
\end{array}\right\},$$
where we denote $W:=\Span(e_1,\dots,e_n,f_1)$.

However, $\iota_0(\ClGr(1,n))$ does not preserve $\Span(e_{n+1})$ and $\Span(f_2)$ separately, but only the entire subspace $\Span(e_{n+1},f_2)$. Therefore, we can identify: 
$$\mathfrak C'(\ClGr(1,n)):=\left\{(V_+,V_-,l_+,l_-)\midwd \begin{array}{l}
\R^{n+3}=V_+\oplus V_-\oplus l_+\oplus l_-\\
b|_{V_+}\text{ is positive definite, }V_+\subset W\\
b|_{V_-}\text{ is negative definite, }V_-\subset W\\
b|_{l_+}\text{ is positive definite, }l_+\subset \Span(e_{n+1},f_2)\\
b|_{l_-}\text{ is negative definite, }l_-\subset \Span(e_{n+1},f_2)\\
V_+\perp_b V_-,\;l_+\perp_b l_-
\end{array}\right\},$$
where as above $W=\Span(e_1,\dots,e_n,f_1)$.

The Riemannian symmetric spaces $\mathfrak C'(\Spin(1,n))$ and $\mathfrak C'(\ClGr(1,n))$ naturally embed into $\mathfrak C'(\Spin_0(2,n+1))$ as follows: 
$$(V_+,V_-,l_+,l_-)\mapsto (V_+\oplus l_+, V_-\oplus l_-),$$
where $l_+=\Span(e_{n+1})$ and $l_-=\Span(f_2)$ in case of $\mathfrak C'(\Spin(1,n))$.

\begin{rem}
In constructing the models of the Riemannian symmetric space associated to the group 
$\Sp_2(G,\sigma)\cong \Spin_0(m+1,n+1)$, we restricted attention to the case $m=1$, 
since only in this situation the Lie algebra $(B(m,n),\sigma)$ is Hermitian. 
This allows one to define the positive cone $B^\sigma_+(m,n)$, which is required to build 
these models, and moreover ensures that the stabilizer of every point in the symmetric space, 
which is conjugate to the group $\KSp_2(G,\sigma)$, is compact. 

In the remaining cases, when $(B(m,n),\sigma)$ is of Jordan type but not Hermitian, 
the group $\Sp_2(G,\sigma)\cong \Spin_0(m+1,n+1)$ and its subgroup 
$\KSp_2(G,\sigma)$ are still well defined, although the latter is no longer compact. 
This nevertheless allows one to define the symmetric space 
$\Sp_2(G,\sigma)/\KSp_2(G,\sigma)$, which is then of pseudo--Riemannian type. 
Models of these pseudo--Riemannian symmetric spaces can also be formulated in terms of 
$B(m,n)$; however, a detailed study of their geometric structure lies beyond the scope of this article and will be pursued in future work.
\end{rem}

\bibliographystyle{abbrv}
\bibliography{bibl}

\bigskip

\noindent\small{\textsc{School of Mathematics, Korea Institute for Advanced Study}\\
  85 Hoegi-ro, Dongdaemun-gu, Seoul 02455, Republic of Korea}\\
\emph{E-mail address}:  \texttt{erogozinnikov@gmail.com}

\bigskip

\end{document}